\documentclass[11pt]{article}
\usepackage[utf8]{inputenc}
\usepackage[margin=1in]{geometry}

\usepackage{amsmath}
\usepackage{hyperref}
\usepackage{amsfonts}
\usepackage{amssymb}
\usepackage{graphicx}
\usepackage{multicol,float}
\usepackage{color}
\usepackage{multirow}
\usepackage{amsthm}
\usepackage{dsfont}
\usepackage{setspace}
\usepackage{bm}
\usepackage{caption}
\usepackage{subcaption}
\usepackage{comment}
\usepackage{mathtools}
\mathtoolsset{showonlyrefs=true}
\usepackage{epstopdf}
\usepackage{xcolor}
\DeclareMathAlphabet{\mathpzc}{OT1}{pzc}{m}{it}
\usepackage[numbers]{natbib}
\usepackage{enumitem}
\usepackage{booktabs} 
\newtheorem{theorem}{Theorem}[section]

\newtheorem{assumption}{Assumption}[section]

\newtheorem{corollary}{Corollary}[section]

\newtheorem{definition}{Definition}[section]

\newtheorem{proposition}[theorem]{Proposition}
\newtheorem{remark}{Remark}[section]

\newcommand \drm {\mathrm{d}}
\newcommand \E {\mathbb{E}}

\newcommand \EE {\mathbb{E}}

\newcommand \eps {\epsilon}
\newcommand \gam {\gamma}

\newcommand \la {\lambda}
\newcommand \ba {\tilde{a}}
\newcommand \hp {\widehat{\pi}}
\newcommand \noi {\noindent}
\newcommand \varp {\varphi}
\newcommand \sig {\sigma}
\newcommand \e {\mathrm{e}}

\newenvironment{bluepar}
{\begingroup\color{blue}}
{\endgroup}

\usepackage{mathtools}                              
\mathtoolsset{showonlyrefs=true}    
\numberwithin{equation}{section}
\renewcommand{\theequation}{\arabic{section}.\arabic{equation}}

\newcommand{\blue}[1]{\textcolor{blue}{#1}}

\newcommand{\bpi}{\boldsymbol{\pi}}
\providecommand{\smn}{{-}} 

\allowdisplaybreaks
\begin{document}
	\title{Equilibrium Strategies for the N-agent Mean-Variance Investment Problem over a Random Horizon}
	
	\author{Xiaoqing Liang%
		\thanks{School of Sciences, Hebei University of Technology, Tianjin, 300401, P. R. China. Email: liangxiaoqing115 \break @hotmail.com.}
		\and Jie Xiong%
		\thanks{Department of Mathematics and SUSTech International center for Mathematics, Southern University
			of Science and Technology, Shenzhen, 518055, P. R. China. Email: xiongj@sustech.edu.cn.}
		\and Ying Yang%
		\thanks{Department of Mathematics, Southern University
			of Science and Technology, Shenzhen, 518055, P. R. China. Email:12331007@mail.sustech.edu.cn}
	}
	\date{\today}
	\maketitle
	\begin{abstract}
		We study equilibrium feedback strategies for a family of dynamic mean-variance problems with competition among a large group of agents. We assume that the time horizon is random and each agent's risk aversion depends dynamically on the current wealth. We consider both the finite population game and the corresponding mean-field one. Each agent can invest in a risk-free asset and a specific individual stock, which is correlated with other stocks by a common noise. By applying stochastic control theory, we derive the extended Hamilton-Jacobi-Bellman (HJB) system of equations for both $n$-agent and mean-field games. Under an exponentially distributed random horizon, in each case, we explicitly obtain the equilibrium feedback strategies and the value function. Our results show that the agent's equilibrium feedback strategy depends not only on his/her current wealth but also on the wealth of other competitors. Moreover, when the risk aversion is state-independent and the risk-free interest rate is zero, the equilibrium strategies degenerate to constants, which is identical to the unique equilibrium obtained in \citet{lacker2019mean} with exponential risk preferences; when the competition parameter goes to zero and the risk aversion equals some specific value, the equilibrium strategies coincide with the ones derived in \citet{landriault2018equilibrium}.

		\medskip
		
		\noi \textit{MSC2020 Codes}: Primary 93E20, 91G10  \, Secondary  91G80, 60H30.

		\medskip
		
		\noi \textit{Keywords}: mean-variance criterion; equilibrium feedback strategy; random horizon; state-dependent risk aversion; investment

	\end{abstract}

	\section{Introduction}
Stochastic optimal control problems aim to analyze a set of control variables that enable the controller to achieve a desired target for a controlled state process. In traditional stochastic control problems, Pontryagin's maximum principle and Bellman's dynamic programming are two commonly used methods for finding optimal solutions (e.g., see \citet{yong2012stochastic}). A crucial requirement for applying Bellman's dynamic programming is that the problem must be time-consistent, meaning that an optimal strategy chosen at the current time will remain optimal in the future. However, this property is violated in time-inconsistent control problems, particularly in the mean-variance (MV) problems, where the presence of the variance term leads to the failure of the iterated expectation property.

In the existing literature, two main approaches are typically used to address the issue of time inconsistency. The first is the pre-commitment approach, which optimizes the objective function at the initial time and requires the controller to commit to this strategy throughout the control period. In MV criterion, the pre-commitment strategy involves setting the expectation of terminal wealth as a static constant, allowing the dynamic programming principle to be applied. A great deal of research focus on solving MV optimization problems from the pre-commitment perspective. See, for instance, \citet{li2000optimal} and  \citet{zhou2000continuous} embedded the original MV problem into a stochastic linear-quadratic control problem and derived the optimal investment strategies from tractable auxiliary problems in both discrete-time and continuous-time frameworks. \citet{bielecki2005continuous} studied a continuous-time MV portfolio selection problem with random coefficients and bankruptcy constraints. \citet{xiong2007mean} studied a continuous-time portfolio selection problem under the MV framework in an incomplete information market, where only the past prices of the stocks and the bond are available to the investors.

However, pre-commitment strategies are not optimal at future time points, making them time-inconsistent. An alternative approach to addressing the MV criterion is to frame the problem within a game-theoretic context, aiming to determine the subgame-perfect Nash equilibrium. In this framework, the problem is modeled as a non-cooperative game, where each time point is treated as a separate player.
\citet{strotz1973myopia} was the first to formally apply a game-theoretic approach to deal with a time-inconsistent deterministic Ramsay problem. \citet{bjork2017time} considered a general class of time-inconsistent objective functions in a continuous-time Markovian setting, deriving an extended Hamilton-Jacobi-Bellman (HJB) equation and providing a verification theorem. \citet{bjork2014mean} developed a time-consistent strategy for a continuous-time MV portfolio problem with state-dependent risk aversion. \citet{landriault2018equilibrium} applied stochastic control theory to derive the extended HJB system in both discrete and continuous-time frameworks, investigating equilibrium feedback strategies with a state-dependent risk aversion and random time horizon. More recently, this methodology has been extended to many different frameworks, for example, \citet{pun2018time} and \citet{yan2020robust} on portfolio selection, \citet{zeng2013time} and \citet{chen2020robust} on reinsurance and investment, \citet{liang2014optimal} and \citet{zhang2025time} on pension fund optimization.

Most of the research mentioned above has focused on a single-agent or two-agent scenarios. However, in practice, agents often compete with many others in large-scale markets, where relative performance becomes a crucial factor in their decision-making. In this context,  \citet{lacker2019mean} studied optimal portfolio management with relative concerns in both finite and infinite populations, deriving time-consistent equilibrium strategies for constant absolute risk aversion (CARA) and constant relative risk aversion (CRRA). \citet{lacker2020many} further examined portfolio optimization for competitive agents with CRRA utilities over a finite time horizon, where each agent's utility depends on both their absolute wealth and consumption and their relative wealth and consumption compared to the average among other agents. \citet{guan2022time} derived time-consistent investment and reinsurance strategies for MV problems in multi-agent settings and mean-field games. \citet{dos2022forward} investigated the multi-player games of investment-consumption under the forward performance framework. \citet{liang2023predictable} considered the relative performance concerns in a framework of discrete-time predictable forward performance processes for both finite player and mean-field games. \citet{bo2024mean} studied multiple agents in mean-field interactions, analyzing their terminal wealth under exponential utility with relative performance. They established both finite $n$-agent game equilibria and mean-field game equilibria, exploring their convergence.

In this paper, we explore a class of dynamic mean-variance problems involving competition among both finite and infinite numbers of agents. In order to have a more realistic model, we assume a random time horizon and each agent's risk aversion dynamically depends on their current wealth. Each agent is allowed to invest in a risk-free asset and a specific individual stock, which is correlated with other stocks through a common noise factor. We solve this problem using a two-step approach. In the first step, we choose a representative agent, and then determine his or her ``best response'' given the other agents' investment strategies. Using stochastic control theory, we prove a general verification theorem for this multi-agent, state-dependent optimization problem, which includes an extended HJB system of equations for both $n$-agent and mean-field games. When the time horizon follows an exponential distribution, the problem becomes time-independent. Thus, solving the extended HJB equation reduces to solving a series of time-independent ordinary differential equations (ODEs).  In both $n$-agent and mean-field games, we explicitly derive the equilibrium feedback strategies and value functions by solving these ODEs. In the second step, we search for a fixed point of the whole system such that the candidate investment strategy is indeed a Nash equilibrium. We provide explicit expressions for the Nash equilibrium for two special scenarios: The scenario where the agents are homogeneous and the scenario where there are two heterogeneous agents. Our results show that the agent's equilibrium feedback strategy depends on both the agent's current wealth and the wealth of other competitors. Furthermore, we establish a connection between the finite-agents game and the mean-field game when the agents are homogeneous.

Our study makes the following contributions. First, we extend the work of \citet{lacker2019mean} to a multi-agent MV frameworks. In \citet{lacker2019mean}, the equilibrium strategy consists of two parts: one is the traditional Merton portfolio, which does not consider relative performance concerns, and the other depends on competition parameters. However, the equilibrium strategy in that model is independent of the agent's wealth. In contrast, our model incorporates state-dependent risk aversion and a random horizon, resulting in a more realistic equilibrium strategy that depends on the agent's current wealth, the competitors' current wealth, and both competition-dependent and independent terms. Furthermore, we find that when risk aversion is state-independent and the risk-free interest rate is zero, our equilibrium strategies reduce to constant values, which match the unique equilibria derived in \citet{lacker2019mean} with exponential risk preferences.

Second, \citet{guan2022time} also studied a time-consistent MV optimization problem in multi-agent and mean-field game frameworks. However, they assumed state-independent risk aversion, meaning the equilibrium strategies in their model are independent of agents' wealth. By incorporating state-dependent risk aversion, our model presents a significantly more complex problem. However, using the time-consistent theory developed in \citet{bjork2017time}, we derive an explicit form of the equilibrium strategy for the case where the random time horizon is exponentially distributed.

Finally, we extend the work of \citet{landriault2018equilibrium} to multi-agent game and the mean-field game frameworks. We show that when the competition parameter approaches zero and risk aversion reaches specific values, the equilibrium strategies coincide with those in \citet{landriault2018equilibrium}. Additionally, we examine the limiting behavior of the equilibrium strategies as the hazard rate of the random horizon approaches infinity. Furthermore, we show that the time-consistent equilibrium of the \( n \)-agent game converges to the equilibrium of the mean-field game (MFG) as $n \to \infty$.

The remainder of this paper is organized as follows. In Section \ref{sec:ngame}, we study an $n$-agent mean-variance investment problem with a random horizon and state-dependent risk aversion. Subsection \ref{subsec:model} formulates the $n$-agent game. In Subsection \ref{subsec:verif}, we solve this problem using a two-step approach. In the first step, for a representative agent, we determine his or her ``best response'' with respect to the other agents' investment strategies which are given in advance. Using stochastic control theory, we prove a verification theorem, which includes an extended HJB system of equations. Then, we consider the case where the random horizon follows an exponential distribution. Under this assumption, the extended HJB equation simplifies to a series of time-independent ordinary differential equations, and we derive the equilibrium feedback strategies by solving these ODEs. In the second step, we search for a fixed point of the whole system such that the candidate investment strategy is indeed the Nash equilibrium. We provide explicit expression for the Nash equilibrium for two special scenarios. In Subsection \ref{sec:homogenous case}, we assume that the agents are homogeneous, and in Subsection \ref{sec:heterogeneous case}, we study the scenario when there are two heterogeneous agents. In Section \ref{sec:MF}, we study the limit as $n \to \infty$ and analyze the equilibrium strategies in the mean-field game. The problem is formulated in Subsection \ref{subsec:MF-model}, and, similar to Subsection \ref{subsec:verif}, we prove a verification theorem and derive the equilibrium strategy by solving the corresponding extended HJB system for the representative agent in Subsection \ref{subsec:MBR}. In Subsection \ref{subsec:MFH}, under a homogeneous mean-field model framework, we find the explicit expression of the equilibrium investment strategy for the whole system, and establish the connection between the $n$-agent game and the mean-field game. Finally, in Section \ref{sec:NA}, we provide several numerical examples to further explore the effects and sensitivity of the parameters on the equilibrium strategies.

	\section{The $n$-agent game}\label{sec:ngame}
	
	In this section, we formulate the continuous-time $n$-agent investment problem using a mean-variance criterion over a random horizon. In Subsection \ref{subsec:model}, we introduce the financial market, the objective, the definition of the equilibrium feedback strategy, and the corresponding equilibrium value function. In Subsection \ref{subsec:verif}, we first prove a general verification theorem, and then we use it to find the solutions of the $n$-agent game when the random horizon is exponentially distributed, that is, the hazard rate of the random horizon is a positive constant.  In Subsection
	\ref{sec:homogenous case}, we consider the equilibrium strategy for the homogeneous agents system. Finally, in Subsection
	\ref{sec:heterogeneous case}, we derive the equilibrium strategy for the heterogeneous 2-agent system.
	
	\subsection{Model formulation}\label{subsec:model}
	
	Let \((\Omega, \mathcal{F}, \mathbb{F}=\{\mathcal{F}_t\}_{t \geq 0}, \mathbb{P})\) be a filtered, complete probability space satisfying the usual conditions. Suppose that \(\bm{W}=(W_1(t),\ldots,W_n(t))_{t \geq 0}\) is an \(n\)-dimensional standard Brownian motion and \(B=(B(t))_{t \geq 0}\) is a one-dimensional standard Brownian motion, both of them are defined on this probability space and are mutually independent.
	
	There are \(n\) agents in the financial market, indexed by \(i\in\{1,\ldots,n\}\), along with \(n\) risky assets and one risk-free asset. Each agent can invest in the risk-free asset and a stock. More precisely,
	the $i$-th agent can invest in a risk-free asset, whose price is denoted by $S_0(t)$ satisfying the following equation
	\begin{equation}
		\frac{\mathrm{d}S_0(t) }{S_0(t)}=r\mathrm{d}t,  \qquad \qquad S_0(0) = s_0,
	\end{equation}
	in which $r \ge 0$ is the constant interest rate. She can also invest in a stock, we call it the $i$-th stock, whose price process $S_i= \{S_i(t)\}_{t \ge 0}$ satisfies the following stochastic differential equation (SDE):
	\begin{align}
		\frac{\drm S_i (t)}{S_i (t)} &= b_i \, \drm t + \xi_i \, \drm W_i (t) + \sigma_i \, \drm B(t), \quad\quad S_i(0) = s_i,
	\end{align}
	in which \(b_i (>r)\) is the yield of stock, \(\xi_i (\ge 0)\) is its idiosyncratic volatility, and \(\sigma_i (\ge 0)\) is the common volatility with $\xi_i + \sigma_i>0$.
	In our paper, we use the Brownian motion \(\bm{W}\) to represent the idiosyncratic risk factors for the stocks, and \(B\) to represent the common risk factor that affects the entire market.
	
	For $i\in\{1,\ldots,n\}$, let $\pi_i(t)$ represent the amount of money that the $i$-th agent invests in the $i$-th stock. Then, the $i$-th agent's wealth process $X^{\boldsymbol{\pi}_i}_i=\{X^{{\boldsymbol{\pi}}_i}_i(t)\}_{t\geq 0}$ under strategy ${\boldsymbol{\pi}}_i = \{\pi_i(t)\}_{t \ge 0}$ follows the dynamics
\begin{align}
\drm X^{\boldsymbol{\pi}_i}_i (t)
&= [r X^{\boldsymbol{\pi}_i}_i(t) + (b_i - r) \pi_i(t)] \, \drm t + \xi_i \pi_i(t) \, \drm W_i(t) + \sigma_i \pi_i(t) \, \drm B(t).\label{eq:X}
	\end{align}
We use \(\bm{X}^{\boldsymbol{\pi}}=\big(X^{\boldsymbol{\pi}_1}_1(t),\ldots, X^{\boldsymbol{\pi}_n}_n(t)\big)_{t\geq 0}\) to denote the \(n\) agents' wealth process under the strategy $\boldsymbol{\pi} = (\boldsymbol{\pi}_1,\ldots, \boldsymbol{\pi}_n)$.
\blue{Let $\tau$ be a random variable defined on the probability space $(\Omega, \mathcal{F}, \mathbb{F}=\{\mathcal{F}_t\}_{t \geq 0}, \mathbb{P})$, which is independent of the Brownian motions $(\bm{W}, B)$. Let $\lambda(t)$ denote the hazard rate of $\tau$, which is a deterministic time-dependent function, and then the survival probability is given by
	\[
	\mathbb{P}(\tau > t) = \exp\left(-\int_0^t \lambda(s) \drm s\right).
	\]
}
To address the investment problem of each agent, we introduce the definition of an \textit{admissible strategy} for the $i$-th agent.
\blue{	
\begin{definition}[Admissible strategy]\label{def:admissible}
For any \(i \in \{1,\ldots,n\}\), a strategy \( \boldsymbol{\pi}_i = \big(\pi_i(t)\big)_{t\geq 0} \) is called admissible if it satisfies the following conditions:
\begin{itemize}
\item[$(1)$] $\boldsymbol{\pi}_i$ is an $\mathbb{F}$-progressively measurable real-valued process, and for any $x_i \in \mathbb{R}$,
 there exist constants $\mathfrak{c}_i, \mathfrak{l}_i$, and $\mathfrak{s}_i$, such that $\E_{0, {x_i}} \big[\left({\bpi_i}(t)\right)^2\big]\le \mathfrak{c}_i \e^{\mathfrak{s}_i t} + \mathfrak{l}_i$, for any $t\ge0$.
\item[$(2)$] The SDE \eqref{eq:X} has a unique strong solution and $\E_{0, {x_i}} \big[\left({X}^{\bpi_i}_i(\tau)\right)^2\big]<\infty$.
\item[$(3)$] There exist constants $\mathfrak{a}_i, \mathfrak{b}_i$, and $\mathfrak{d}_i$ such that $\E_{0, {x_i}} \big[\left({X}^{\bpi_i}_i(t)\right)^2\big]\le \mathfrak{a}_i \e^{\mathfrak{d}_i t} + \mathfrak{b}_i$, for any $t\ge0$.
\end{itemize}
We use \( \boldsymbol{\Pi}_i \) to denote the set of the \(i\)-th agent's all admissible strategies.
\end{definition}
}

We consider a mean-variance optimization portfolio problem over a random horizon $(0, \tau]$. For each agent, they are not concerned about their own wealth but also taking into account other's relative wealth. Define
	$$\overline{X}^{\bpi}(t):=\frac{1}{n}\sum_{i = 1}^{n}X^{\bpi_i}_i(t)$$
as the $n$ agents' average wealth at time $t$ under the strategy $\bpi$.

\blue{
In single-period mean variance analysis, if the mean variance utility
function (with constant risk aversion $\gam$) is applied not to the wealth itself, but to
the return rate, that is, the objective function is given by
\begin{align}
\E\left[\dfrac{X^{\bpi}_1}{x}\right] - \frac{\gam}{2}\text{Var}\left[\dfrac{X^{\bpi}_1}{x}\right],
\end{align}
and we can rewrite it as
\begin{align}
\dfrac{1}{x^2} \Big\{ 2 x \cdot \E[X^{\bpi}_1] - {\gam} \text{Var}[X^{\bpi}_1] \Big\}.
\end{align}
\citet{bjork2014mean} studied an equivalent mean-variance portfolio optimization problem when the risk aversion depends dynamically on current wealth, that is, $\gam(x) = \gam/x$ when $x>0$. \citet{landriault2018equilibrium} also considered a similar continuous-time mean-variance investment problem with a state-dependent risk aversion over a random horizon.}

Given this, and to ensure a more realistic model, we assume that the \(i\)-th agent seeks to maximize $J_i(t,\bm{x};\boldsymbol{\pi})$, in which $J_i$ has the following form:
\begin{align}
J_i(t,\bm{x};\boldsymbol{\pi})=({\mu_{i1}x_i+\mu_{i2}})\mathbb{E}_{t,\bm{x},t +}[X^{\bpi_i}_i(\tau)-\phi_i\overline{X}^{\bpi}(\tau)]-{\gamma_i}\text{Var}_{t,\bm{x},t +}[X^{\bpi_i}_i(\tau)-\phi_i\overline{X}^{\bpi}(\tau)]. \label{eq:n-cost}
\end{align}
Here,\footnote{\blue{In this paper, we work on the entire real line \(\mathbb{R}\) for \(x_i\), rather than restricting it to \(x_i\ge 0\) or \(\mu_{i1}x_i+\mu_{i2}>0\), which would introduce state constraints and make the resulting optimization problem extremely challenging to solve using the extended Hamilton-Jacobi-Bellman (EHJB) equation approach.}} \textcolor{blue}{\(\bm{x}=(x_1,\ldots,x_n)\in\mathbb{R}^n\)} is the agents' initial wealth at time $t$, and \(\mathbb{E}_{t,\bm{x},t +}[\cdot]:=\mathbb{E}[\cdot|\bm{X}^{\bpi}(t)=\bm{x},\tau > t]\). Moreover, we use \(X^{\bpi_i}_i(t)-\phi_i\overline{X}^{\bpi}(t)\) to denote the \(i\)-th agent's relative wealth with \(\phi_i\in[0,1]\) representing the \(i\)-th agent's risk preference regarding their own wealth relative to the overall average wealth. If $\phi_i$ is large, the agent probably focuses more on the relative wealth than their own wealth. $\gam_i>0$ is the \(i\)-th agent's risk aversion towards variance, and \(\mu_{i1}\) and \(\mu_{i2}\) are positive constants.\footnote{\blue{The constants \(\mu_{i1}\) and \(\mu_{i2}\) cannot be zero simultaneously. For example, if \(\mu_{i1}=2\) and \(\mu_{i2}=0\), the risk-aversion function becomes \(\gamma_i/(2x)\). In this case, the effective risk-aversion diverges to infinity as wealth approaches zero. Economically, this implies that the agent abruptly ceases all market participation and liquidates all risky positions as wealth depletes.}}

Due to the presence of $\overline{X}^{\bpi}(\tau)$ in the objective, the optimal decision of the $i$-th agent combines with the other competitors' optimal control strategies. We introduce an auxiliary process
$Y^{\bpi_{\smn i}}_{\smn i}:=\{Y^{\bpi_{\smn i}}_{\smn i}(t)\}_{t \ge 0}$ as
	\[Y^{\bpi_{\smn i}}_{\smn i}(t):=\frac{1}{n}\sum_{\substack{j = 1\\j\neq i}}^{n}X^{\bpi_j}_j(t),\]
in which
${\bpi_{{\smn i}}}:= (\boldsymbol{\pi}_1,\ldots, \boldsymbol{\pi}_{i-1}, \boldsymbol{\pi}_{i+1},\ldots, \boldsymbol{\pi}_n)$, and then $Y^{\bpi_{\smn i}}_{\smn i}$ satisfies the following SDE:
	\begin{align}\label{eq:Y}
		\drm Y^{\bpi_{\smn i}}_{\smn i}(t)
		=\left[r Y^{\bpi_{\smn i}}_{\smn i}(t)+\widehat{\underline{br}\pi}(t)\right]\drm t+\frac{1}{n}\sum_{\substack{j = 1\\j\neq i}}^{n}\xi_j\pi_j(t)\drm W_j(t)+\widehat{\sigma\pi}(t)\drm B(t),
	\end{align}
	with \(\widehat{\underline{br}\pi}(t)=\frac{1}{n}\sum_{j=1, j \neq i}^{n}(b_j - r)\pi_j(t)\) and \(\widehat{\sigma\pi}(t)=\frac{1}{n}\sum_{j=1, j \neq i}^{n}\sigma_j\pi_j(t)\). By using the auxiliary process $Y^{\bpi_{\smn i}}_{\smn i}$, we can rewrite \eqref{eq:n-cost} as follows
	\begin{align}
		J_i(t, x_i, y_{\smn i}; \bm{\pi}) &= (\mu_{i1} x_i + \mu_{i2}) \mathbb{E}_{t, x_i, y_{\smn i}, t+}\left[\left(1 - \frac{\phi_i}{n}\right) X^{\bpi_i}_i(\tau) - \phi_i Y^{\bpi_{\smn i}}_{\smn i}(\tau)\right] \\
		&\quad - \gamma_i \text{Var}_{t, x_i, y_{\smn i}, t+}\left[\left(1 - \frac{\phi_i}{n}\right) X^{\bpi_i}_i(\tau) - \phi_i Y^{\bpi_{\smn i}}_{\smn i}(\tau)\right], \label{eq:fun_Ji}
	\end{align}
	with \(y_{\smn i} := \frac{1}{n} \sum_{j=1, j \neq i}^{n} x_j\).

	\subsection{Best responses} \label{subsec:verif}

	Inspired by the usual strategy of solving Nash equilibria games, we first consider the best response of an agent to the actions of all other competitors, whose strategies are arbitrary and admissible but fixed and given in advance.
We show the definition of the best-response equilibrium feedback strategy for an agent in terms of the given strategies of all other agents.
	
	\begin{definition}[Time-consistent equilibrium feedback strategy]\label{definitionE}
Fix $i \in \{1, \ldots, n\}$. Assume that each agent $j \neq i$ follows $\widehat{\bm{\pi}}_j \in \bm{\Pi}_j$. For a given feedback trading strategy \(\widehat{\bm{\pi}}_i = \{\widehat{\pi}_i(t, {X}^{\widehat{\bpi}_i}_i(t), Y^{\widehat{\bpi}_{\smn i}}_{\smn i}(t))\}_{t \ge 0} \in \boldsymbol{\Pi}_i\), any fixed \((t, x_i, y_{\smn i}) \in [0,\infty) \times \mathbb{R} \times \mathbb{R}\), a fixed number \(\epsilon > 0\), and a real number $\pi_i$; then, define a new feedback strategy \({\bm{\pi}}^\epsilon_i\) by
		\[
		\pi_i^\epsilon(s, x_i, y_{\smn i}) =
		\begin{cases}
			\pi_i, & \text{for } {t \leq s < (t+\epsilon) \wedge \tau,} \\
			\widehat{\pi}_i(s, x_i, y_{\smn i}), & \text{for } (t+\epsilon) \wedge \tau \leq s \leq \tau.
		\end{cases}
		\]		
		If for all \((t, x_i, y_{\smn i}) \in [0,\infty) \times \mathbb{R} \times \mathbb{R}\),
		\begin{align}
		\liminf_{\epsilon \to 0^+} \frac{J_i(t, x_i, y_{\smn i}; \widehat{\bm{\pi}}_1, \ldots, \widehat{\bm{\pi}}_i, \ldots, \widehat{\bm{\pi}}_n) - J_i(t, x_i, y_{\smn i}; \widehat{\bm{\pi}}_1, \ldots, \bm{\pi}^\epsilon_i, \ldots, \widehat{\bm{\pi}}_n)}{\epsilon} \geq 0, \label{eq:limit_equi}
		\end{align}
		then $\widehat{\bpi}=(\widehat{\bm{\pi}}_1, \ldots, \widehat{\bm{\pi}}_i, \ldots, \widehat{\bm{\pi}}_n)$ is a time-consistent equilibrium feedback strategy,
		and  the corresponding equilibrium value function for the \(i\)-th agent is given by
		\[
		V_i(t, x_i, y_{\smn i}) = J_i(t, x_i, y_{\smn i}; \widehat{\bm{\pi}}), ~~~~\text{for}~~i=1,\ldots,n.
		\]
	\end{definition}
	We first prove a verification theorem for our problem by employing a decomposition approach similar to that in \citet{kryger2010some}. For the sake of brevity, the proof is deferred to Appendix \ref{app:A}. After that, we use this theorem to derive the time-consistent equilibrium feedback strategies for the \(n\)-agent game when the hazard rate of the random horizon is a positive constant. Given a function \(g(t,x,y) \in C^{1,2,2}([0,\infty) \times \mathbb{R} \times \mathbb{R})\), and for any real numbers $\pi = (\pi_1, \pi_2, \ldots, \pi_n)$, define the infinitesimal generator \(\mathcal{L}_i^{\pi}\), \(i = 1,\ldots,n\), by
	\begin{align}
		\mathcal{L}_i^{{\pi}} g(t,x,y) = & \, g_t + g_x[rx + (b_i - r)\pi_i] + g_y[ry + \widehat{\underline{br}\pi}] \\
		& + \frac{1}{2} g_{xx} \left[\xi_i^2 \pi_i^2 + \sigma_i^2 \pi_i^2\right] + \frac{1}{2} g_{yy} \left[\widehat{\xi\pi}^2 + \widehat{\sigma\pi}^2\right] \\
		& + g_{xy} \sigma_i \pi_i \widehat{\sigma\pi},
	\end{align}
\textcolor{blue}{ 	in which  \(\widehat{\underline{br}\pi}=\frac{1}{n}\sum_{j=1, j \neq i}^{n}(b_j - r)\pi_j\),  \(\widehat{\sigma\pi} = \frac{1}{n}\sum_{j=1, j \neq i}^{n}\sigma_j\pi_j\), \(\widehat{\xi\pi}^2 = \sum_{j=1, j \neq i}^{n} \left(\frac{1}{n}\xi_j \pi_j\right)^2\), and \(\widehat{\sigma\pi}^2 = \sum_{j=1, j \neq i}^{n} \left(\frac{1}{n}\sigma_j \pi_j\right)^2\). }
	Define a function \(f \in C^{1,2,2,2,2}([0,\infty) \times \mathbb{R}^4)\) by
	\begin{align}\label{eq:f}
		f(t, x, y, u, v) = (\mu_{i1} x + \mu_{i2}) u - \gamma_i (v - u^2),
	\end{align}
 where $u$ and $v$ are two functions depending on $(t,x,y)$.	
	\begin{theorem}[Verification theorem]\label{th:verification}
Let $i \in \{1, \ldots, n\}$ and assume each agent $j \neq i$ follows $\widehat{\bm{\pi}}_j \in \bm{\Pi}_j$. Suppose that there exist three real-valued functions $V_i(t,x_i,y_{\smn i})$, $G_i(t,x_i,y_{\smn i})$, $H_i(t,x_i,y_{\smn i})\in C^{1,2,2}([0,\infty)\times\mathbb{R}\times\mathbb{R})$ satisfying the following conditions:
		\begin{itemize}
			\item[$(1)$] For all \( (t,x_i,y_{\smn i})\in [0,\infty)\times \mathbb{R}\times \mathbb{R}\),
			\begin{align}
				&\lambda(t)\bigg\{V_i(t,x_i,y_{\smn i})-(\mu_{i1}x_i+\mu_{i2})\Big(\big(1-\frac{\phi_i}{n}\big)x_i-\phi_i y_{\smn i}\Big)+{\gamma_i}\Big[G_i(t,x_i,y_{\smn i})-\Big((1-\frac{\phi_i}{n})x_i-\phi_iy_{\smn i}\Big)\Big]^2\bigg\}\\	&=\sup_{{\pi}}\{\mathcal{L}_i^{\widehat{\pi}(i)}V_i(t,x_i,y_{\smn i})-\mathcal{L}_i^{\widehat{\pi}(i)}f(t,x_i,y_{\smn i},G_i,H_i)+f_G(t,x_i,y_{\smn i},G_i,H_i)\mathcal{L}_i^{\widehat{\pi}(i)} G_i(t,x_i,y_{\smn i})\\
				& \quad +f_H(t,x_i,y_{\smn i},G_i,H_i)\mathcal{L}_i^{\widehat{\pi}(i)}H_i(t,x_i,y_{\smn i})\}, \label{eq:EHJB}
			\end{align}
			in which $\widehat{\pi}(i) = (\widehat{\pi}_1, \ldots, \pi_i, \ldots, \widehat{\pi}_n)$. Let $\widehat{\pi}_i$ denote the function of $(t, x_i, y_{\smn i})$ that attains the supremum in \eqref{eq:EHJB} and assume that $\widehat{\bm{\pi}}_i$ is admissible, that is, $\widehat{\bm{\pi}}_i \in \bm{\Pi}_i$.
			\item[$(2)$] For all \( (t,x_i,y_{\smn i})\in [0,\infty)\times \mathbb{R}\times \mathbb{R}\),
			\begin{align}\label{eq:LG}
				\lambda(t)\Big[ G_i(t,x_i,y_{\smn i})-\Big(\big(1-\frac{\phi_i}{n}\big)x_i-\phi_i y_{\smn i}\Big)\Big]=\mathcal{L}^{\widehat{{\pi}}}_iG_i(t,x_i,y_{\smn i}).
			\end{align}
			\item[$(3)$] For all \( (t,x_i,y_{\smn i})\in [0,\infty)\times \mathbb{R}\times \mathbb{R}\),
			\begin{align}\label{eq:LH}
				\lambda(t)\Big[H_i(t,x_i,y_{\smn i})-\Big(\big(1-\frac{\phi_i}{n}\big)x_i-\phi_i y_{\smn i}\Big)^2\Big]=\mathcal{L}_i^{\widehat{{\pi}}}H_i(t,x_i,y_{\smn i}).
			\end{align}
			\item[$(4)$] A transversality condition holds. Specifically, for all \( (t,x_i,y_{\smn i})\in [0,\infty)\times \mathbb{R}\times \mathbb{R}\),
			\begin{align}\label{eq:T}
				\lim_{s\to \infty}\mathbb{E}_{t,x_i,y_{\smn i}}\Big[\e^{-\int_t^s \lambda(v)\drm v} j\big(s,X^{\widehat{\bm{\pi}}_i}_i(s),
Y^{\widehat{\bpi}_{\smn i}}_{\smn i}(s)\big)\Big]=0,
			\end{align}
			for \(j=G_i\), \(H_i\), \(V_i\), and \(f\).
		\end{itemize}
		Then,
		\begin{align}
			V_i(t,x_i,y_{\smn i})&=J_i(t,x_i,y_{\smn i};\widehat{\bm{\pi}}_1,\ldots,\widehat{\bm{\pi}}_i,\ldots,\widehat{\bm{\pi}}_n),\\
			G_i(t,x_i,y_{\smn i})&=\mathbb{E}_{t,x_i,y_{\smn i},t+}\Big[\big(1-\frac{\phi_i}{n}\big)X_i^{\widehat{\bm{\pi}}_i}(\tau)-\phi_i Y^{\widehat{\bpi}_{\smn i}}_{\smn i}(\tau)\Big],\\
			H_i(t,x_{i},y_{\smn i})&=\mathbb{E}_{t,x_i,y_{\smn i},t+}\left[\Big(\big(1-\frac{\phi_i}{n}\big)X_i^{\widehat{\bm{\pi}}_i}(\tau)-\phi_i Y^{\widehat{\bpi}_{\smn i}}_{\smn i}(\tau)\Big)^2\right],
		\end{align}
		in which $X^{\widehat{\bm{\pi}}_i}_i$ is the $i$-th agent's wealth process under the equilibrium feedback strategy  $\widehat{\bm{\pi}}_i$.
	\end{theorem}
	We use Theorem \ref{th:verification} to deduce the equilibrium feedback strategy and the corresponding value function for the $i$-th agent when the trading strategies $(\widehat{\pi}_1, \ldots, \widehat{\pi}_{i-1}, \widehat{\pi}_{i+1}, \ldots, \widehat{\pi}_n)$ for the other competitors are fixed. Furthermore, we assume that the hazard rate \(\lambda(t)\) is a positive constant, that is, \(\lambda(t) \equiv \lambda\). Under this assumption, the problem is time homogeneous and the three partial differential equations (PDEs) in Theorem \ref{th:verification} degenerate to a system of ODEs.
	
	For simplicity, we omit the subscript \(i\) from notations such as \(x_i\), \(V_i\), \(G_i\), \(H_i\), \(\mu_{i1}\), \(\mu_{i2}\), \(b_i\), \(\xi_i\), \(\sigma_i\), \(\phi_i\) and \(\gamma_i\). This convention also applies to other terms introduced later, including \(\widehat{\pi}_i\), \(\rho_i\), \(p_i\), \(q_i\), \(\varrho_i\), \(A_i\), \(C_i\), \(D_i\), \(E_i\), \(F_i\), \(I_i\), \(a_i\), \(c_i\), \(\alpha_i\), \( \tilde{a}_i\), \( \tilde{c}_i\), \( \tilde{d}_i\), \( \tilde{e}_i\), \( \tilde{\beta}_i\), and \( \tilde{l}_i\). We also omit the subscript \(\smn i\) from \(y_{\smn i}\).
	
	By the linear structure of the dynamics, we make the following ansatz:
	\begin{align}
		V(x,y)&=A x^2+C y^2+D xy+E x+F y+I ,
		\label{eq:W}\\
		G(x,y)&=a x+c y+\alpha ,
		\label{eq:G}
	\end{align}
	and
	\begin{align}
		H(x,y)= \tilde{a} x^2+ \tilde{c} y^2+ \tilde{d} xy+ \tilde{e} x+ \tilde{\beta} y+ \tilde{l}.
		\label{eq:H}
	\end{align}
	Using the definition of \(f\) given in \eqref{eq:f}, we can easily compute the partial derivatives of \(f\). Then, by substituting these derivatives of \(f\), as well as the derivatives of \(V(x,y)\), \(G(x,y)\), and \(H(x,y)\) into the extended HJB equation \eqref{eq:EHJB}, we obtain
	\begin{align}
		&\lambda\bigg\{\Big[A+\gamma(a-{\varphi}_{n})^2-\mu_1 {\varphi}_{n}\Big]x^2+[C+\gamma(c+\phi)^2]y^2+\Big[D+2\gamma(a-{\varphi}_{n})(c+\phi)\\
		&\quad\quad+\mu_1{\phi}\Big]xy+\Big[E+2\gamma \alpha\Big(a-{\varphi}_{n}\Big)-\mu_2{\varphi}_{n}\Big]x+[F+2\gamma \alpha(c+\phi)+\mu_2\phi]y+I  +\gamma \alpha^2  \bigg\}\\
		&=\sup_{{\pi}}\Big\{ 	(2Ar-\mu_1ar)x^2+2Cry^2+(2Dr-\mu_1cr)xy+\big(Er+D\widehat{\underline{br}\widehat{\pi}}-\mu_1\alpha r\big)x\\
		&\quad\quad\quad\quad+\big(Fr+2\widehat{\underline{br}\widehat{\pi}}C\big)y+\widehat{\underline{br}\widehat{\pi}}F +C\Big(\widehat{\xi\widehat{\pi}}^2+\widehat{\sigma\widehat{\pi}}^2\Big)-\gamma c^2\Big(\widehat{\xi\widehat{\pi}}^2+\widehat{\sigma\widehat{\pi}}^2\Big)+\Big[E(b-r)\\	&\quad\quad\quad\quad+D\sigma\widehat{\sigma\widehat{\pi}}-\mu_1a(b-r)x
	-\mu_1c(b-r)y-\mu_1\alpha(b-r)-\mu_1c\sigma\widehat{\sigma\widehat{\pi}}-2\gamma ac\sigma\widehat{\sigma\widehat{\pi}}\\
	&\quad\quad\quad\quad+2A(b-r)x+D(b-r)y\Big]\pi +(A-\mu_1a-\gamma a^2)(\xi^2+\sigma^2)\pi^2\Big\},	\label{eq:SEHJB2}
	\end{align}		
	in which $\varphi_{n} := 1 - \phi/n$, \(\widehat{\xi\widehat{\pi}}^2=\frac{1}{n^2}\sum_{j=1,j\neq i}^n(\xi_i\widehat{\pi}_j)^2\),  \(\widehat{\underline{br}\widehat{\pi}}=\frac{1}{n}\sum_{j=1,j\neq i}^n(b_j-r)\widehat{\pi}_j\), and \(\widehat{\sigma\widehat{\pi}}=\frac{1}{n}\sum_{j=1,j\neq i}^n\sigma_j\widehat{\pi}_j\). If \(A-\mu_1 a-\gamma a^2 <0\), then from the first-order condition, the equilibrium  feedback  $\widehat{\pi}(x,y)$ is given by
	\begin{align}
		\widehat{\pi}(x,y):=&\rho(p x+q y+\varrho),\label{eq:NE}
	\end{align}
	where
	\begin{align}
		\rho &= \frac{1}{2(\xi^2+\sigma^2)},\label{eq:rho}\\
		p  &= \frac{(2A -\mu_1a)(b-r) }{\mu_1 a +\gamma a ^2-A }, \label{eq:p} \\
		q  &= \frac{(b-r)(D -\mu_1c) }{\mu_1 a +\gamma a ^2-A }, \label{eq:q} \\
		\varrho  &= \frac{(b-r)(E-\mu_1 \alpha) +\sigma \widehat{\sigma\widehat{\pi}}(D-\mu_1 c - 2 \gamma a c)}{\mu_1 a +\gamma a ^2-A }. \label{eq:k}
        	\end{align}
	From \eqref{eq:NE},	we see to determine the expression of $\hp$, we need to find the values of the three constants $p, q$, and $\varrho$. We first look for the equation for \(p\). By substituting \(H(t,x,y)\) along with its derivatives and \(\widehat{\pi}(x,y)\) from \eqref{eq:NE} into \eqref{eq:LH}, we obtain
	\begin{align}		
		&	\Big[2 \tilde{a} r+2 \tilde{a}\rho p(b-r)+\frac{1}{2} \tilde{a}\rho p^2-\lambda  \tilde{a}+\lambda \varphi_n^2\Big]x^2+\Big[2 \tilde{c}r+ \tilde{d}\rho q(b-r)+\frac{1}{2} \tilde{a} \rho q^2 -\lambda  \tilde{c}+\lambda \phi^2\Big]y^2\\		
		&+\Big[2 \tilde{a}\rho q(b-r)+2 \tilde{d}r+ \tilde{d}\rho p(b-r)+ \tilde{a}pq\rho-\lambda  \tilde{d}-2\lambda \phi\varphi_n\Big]xy+\big[ \tilde{e}r+2 \tilde{a}\rho \varrho(b-r)\\		
		&+ \tilde{e}\rho p(b-r)+ \tilde{d}\sigma \rho p\widehat{\sigma\widehat{\pi}}+ \tilde{a}\rho \varrho p-\lambda  \tilde{e}+ \tilde{d}\widehat{\underline{br}\widehat{\pi}}\big]x+\big[2 \tilde{c}\widehat{\underline{br}\widehat{\pi}}		
		+ \tilde{\beta}r+ \tilde{d}\rho \varrho(b-r)+ \tilde{e}\rho q(b-r)\\		
		&+ \tilde{d}\sigma\rho q\widehat{\sigma\widehat{\pi}}+ \tilde{a}\varrho q\rho-\lambda  \tilde{\beta}\big]y+ \tilde{\beta}\widehat{\underline{br}\widehat{\pi}}+ \tilde{c}\Big(\widehat{ \xi\widehat{\pi}}^2+\widehat{\sigma\widehat{\pi}}^2\Big)+ \tilde{e}(b-r)\rho \varrho+ \tilde{d}\sigma \rho \varrho \widehat{\sigma\widehat{\pi}}+\frac{1}{2} \tilde{a}\rho \varrho^2-\lambda  \tilde{l}=0.\label{eq:LH1}	
	\end{align}
	By setting the coefficient of \(x^2\) to be zero, we obtain
	\begin{align}\label{eq:bara0}
		2 \tilde{a} r+2 \tilde{a}\rho p(b-r)+\frac{1}{2} \tilde{a}\rho p^2 -\lambda  \tilde{a}+\lambda \varphi_n^2=0.
	\end{align}
	If \(2r + 2\rho p(b - r) + \rho^2 p^2(\xi^2 + \sigma^2) - \lambda \neq 0\), then we derive $\ba$ from \eqref{eq:bara0} as
	\begin{align}\label{eq:bara}
		 \tilde{a}=\frac{\lambda \varphi_n^2}{\lambda-[2r+2\rho p(b-r)+\frac{1}{2}\rho p^2]},
	\end{align}
	Moreover, under the equilibrium feedback strategy \(\widehat{\pi}\), we know that
	\begin{align}\label{eq:VGH}
		V(x,y) = (\mu_1 x + \mu_2) G(t,x,y) - \gamma (H(t,x,y) - G(t,x,y)^2).
	\end{align}
	By substituting the expressions for \(V\), \(G\), and \(H\) from \eqref{eq:W}, \eqref{eq:G}, and \eqref{eq:H} into \eqref{eq:VGH} and equating the coefficients of $x^2$ on both sides, we get
	\begin{align}
		A = \mu_1 a - \gamma  \tilde{a} + \gamma a^2,
		\label{eq:A}
	\end{align}
	and thus, \(A - \mu_1 a - \gamma a^2 < 0\) is equivalent to \( \tilde{a} > 0\), which implies that
	\begin{align}\label{eq:conditionP}
		\lambda>2r+2\rho p(b-r)+\frac{1}{2}\rho p^2.
	\end{align}
	
	By substituting the expression and derivatives of \(G(x, y)\) along with the expression of \(\widehat{\pi}(x,y)\), into \eqref{eq:LG}, and by further equating the coefficients of \(x\) and \(y\), and the constant terms on both sides, we obtain the following three equations
	\begin{align}
		&ar-\lambda\Big(a-\varp_{n}\Big)+a\rho p(b-r) =0,
		\label{eq:a}\\
		&cr-\lambda(c+\phi)+a\rho q(b-r)=0,
		\label{eq:c}\\
		&c\widehat{\underline{br}\widehat{\pi}}-\lambda \alpha+a\rho \varrho(b-r)=0.
		\label{eq:e}
	\end{align}
	By solving \eqref{eq:a}, \eqref{eq:c}, and \eqref{eq:e} for $a, c$, and $\alpha$, we obtain
	\begin{align}
		a &= \lambda \varphi_n[\lambda-r-\rho p(b-r)]^{-1}, \label{eq:a1} \\
		c &=(\lambda-r)^{-1}[a\rho q(b-r)-\lambda\phi], \label{eq:c1} \\
		\alpha &=\la^{-1}\big({c\widehat{\underline{br}\widehat{\pi}}+a\rho \varrho(b-r)}\big), \label{eq:e1}
	\end{align}
	in which we assume that $\la \neq r$. Furthermore, by inserting \eqref{eq:bara}, \eqref{eq:A}, and \eqref{eq:a1} into \eqref{eq:p},  we obtain the following cubic equation for $z=\rho(b-r)p$:
	\begin{align}
		&\big(\gamma\varphi_n-\mu_{1}/2\big)z^3+\big\{\mu_{1}\big[\frac{1}{2} (\lambda-r)-2\rho(b-r)^2\big]+\gamma\varphi_n[2\rho(b-r)^2-\lambda+2r]\big\}z^2+\{\rho(b-r)^2\\
		&\cdot[\mu_{1}(3\lambda-4r)+4r\gamma\varphi_n]+\gamma\varphi_n(\lambda-r)^2\}z+\rho(b-r)^2[2\gamma r^2\varphi_n-\mu_{1}(\lambda-r)(\lambda-2r)]=0.
		\label{eq:p1}	
	\end{align}	
	Next, we derive the equation for $q$. Inserting \eqref{eq:A} into \eqref{eq:q} yields that
	\begin{align}
		q&=(\ba\gamma)^{-1}(b-r)(D -c \mu_1). \label{eq:q1}
	\end{align}
	Since $c$ is given in \eqref{eq:c1}, we only need to derive the expression for \(D\). By substituting the expression of \(\widehat{\pi}\) from \eqref{eq:NE} into \eqref{eq:SEHJB2}, rearranging terms and simplifying, we obtain
	\begin{align}
		&\Big\{2Ar-\mu_1ar -\lambda\Big[A+\gamma(a-\varphi_n)^2-\mu_1 \varphi_n\Big]+\frac{\rho}{2}\gamma  \tilde{a}p^2\Big\}x^2+\Big\{2Cr-\lambda [ C+\gamma(c+\phi)^2 ]\\	
		&+\frac{\rho}{2}\gamma \tilde{a}q^2\Big\}y^2+\Big\{2Dr-\mu_1cr-\lambda\Big[D+2\gamma(a-\varphi_n)(c+\phi)+\mu_1{\phi} \Big] +\rho\gamma \tilde{a}pq\Big\}xy+\Big\{ Er\\	
		&+D\widehat{\underline{br}\widehat{\pi}}-\mu_1\alpha r-\lambda\Big[ E+2\gamma \alpha(a-\varphi_n)-\mu_2\varphi_n\Big]+\rho \gamma \tilde{a}\varrho p\Big\}x+\Big\{Fr+2\widehat{\underline{br}\widehat{\pi}}C\\	
		&-\lambda[F+2\gamma \alpha(c+\phi)+\mu_2\phi]+\rho \gamma \tilde{a}\varrho q\Big\}y+ \widehat{\underline{br}\widehat{\pi}}F +C\Big(\widehat{\xi\widehat{\pi}}^2+\widehat{\sigma\widehat{\pi}}^2\Big)-\gamma c^2\Big(\widehat{\xi\widehat{\pi}}^2+\widehat{\sigma\widehat{\pi}}^2\Big)\\	
		&-\lambda I-\lambda\gamma \alpha^2+\frac{\rho}{2}\varrho^2\gamma \tilde{a}=0.	\label{eq:EHJB3}	
	\end{align}
	Due to the arbitrariness of \(x\) and \(y\) in the above equation, we set the coefficients of $y^2$, \(xy\), $x$, $y$, and the constant terms to zero, thereby obtaining five equations for \(C, D, E, F, I\), respectively. Hence, we deduce the expressions of \(C, D, E, F, I\) as follows: \begin{align}
		C &=(2r-\lambda)^{-1}[\lambda\gamma(c+\phi)^2-\frac{\rho}{2}\gamma \tilde{a}q^2], \label{eq:C1} \\
		D &=(2r-\lambda)^{-1}[\mu_1 cr-\gamma  \tilde{a}\rho pq+2\gamma \lambda(a-{\varphi}_{n})(c+\phi)+\lambda\mu_1\phi], \label{eq:D1} \\
		E &=(r-\lambda)^{-1}[\mu_1\alpha r-D\widehat{\underline{br}\widehat{\pi}}-\rho\gamma  \tilde{a}\varrho p+2\gamma\lambda \alpha (a-{\varphi}_{n})-\mu_2\lambda \varp_n],\label{eq:E1} \\
		F &=(r-\lambda)^{-1}[ 2\lambda\gamma \alpha(c+\phi)+\lambda\mu_2\phi-\rho\gamma \tilde{a}\varrho   q-2C\widehat{\underline{br}\widehat{\pi}}], \label{eq:F1} \\
		I &=\lambda^{-1}\Big[(C-\gamma c^2)\Big(\widehat{\xi\widehat{\pi}}^2+\widehat{\sigma\widehat{\pi}}^2\Big)-\lambda \gamma \alpha^2+\frac{\rho}{2}\gamma \tilde{a}\varrho^2+F\widehat{\underline{br}\widehat{\pi}}\Big]. \label{eq:I1}
	\end{align}
	By substituting \eqref{eq:bara}, \eqref{eq:c1}, and \eqref{eq:D1} into \eqref{eq:q1} and simplifying the expressions, we obtain the equation of \(q\):
	\begin{align}\label{eq:q2}
		\{\gamma \tilde{a}(\lambda-r)[2r-\lambda+(b-r)\rho p]-[2\gamma\lambda(a-\varphi_n)+\mu_1(\lambda-r)]a\rho(b-r)^2\}q=-2\gamma\lambda\phi r(a-\varphi_n)(b-r).
	\end{align}
Finally, by substituting \eqref{eq:bara}, \eqref{eq:A}, \eqref{eq:a1}, \eqref{eq:e1}, \eqref{eq:D1}, and \eqref{eq:E1} into \eqref{eq:k}, we derive the expression of \(\varrho\).
Define
	\begin{align}\label{eq:Qi}
	Q :=\gamma \tilde{a}[r-\lambda+\rho p (b - r)]-a \rho(b-r)^2[2\gamma(a-1+\phi/n)+\mu_{1}],
	\end{align}
	and assume that \(Q \neq 0\). By using \(\varrho\) in \eqref{eq:k}, we obtain an expression of \(\widehat{\pi}(x, y)\) in terms of other competitors' equilibrium  feedback  strategy as follows
	\begin{align}\label{eq:pi_i}
		\widehat{\pi}(x, y)= \rho p x + \rho q y + k_{1} \widehat{\sigma\widehat{\pi}} + k_{2}\widehat{\underline{br}\widehat{\pi}}+k_{3},
	\end{align}
	in which
	\begin{align}	
k_{1}&=Q^{-1}\rho\sigma(r-\lambda)(D-\mu_{1}c-2\gamma a c),\label{eq:k1}\\
k_{2}&=Q^{-1}\rho(b-r)[2\gamma(a-1+\phi/n)c+\mu_{1}c -D],\label{eq:k2}\\
k_{3}&=-Q^{-1}\mu_{2}\rho\lambda(1-\phi/n)(b-r).\label{eq:k3}
	\end{align}
We summarize the above results in the following theorem. See Appendix \ref{app:B} for the proof of Theorem \ref{th:main}.
\begin{bluepar}	
For the $i$-th agent, we can rewrite \eqref{eq:pi_i} as follows:
\begin{align}
		\widehat{\pi}_i(x_i, y_{\smn i}) &= \rho_i p_i x_i + \rho_i q_i y_{\smn i} + k_{i1} \widehat{\sigma\widehat{\pi}} + k_{i2}\widehat{\underline{br}\widehat{\pi}}+k_{i3} \\
		&= \rho_i p_i x_i + g(y_{\smn i}, \{\widehat{\pi}_j\}_{j \neq i}),  \label{eq:pi_i1}
\end{align}
in which $g$ is a linear combination of $y_{\smn i}$ and $\{\widehat{\pi}_j\}_{j \neq i}$. Given $\{\widehat{\bm{\pi}}_j\}_{j\neq i}$ are admissible strategies, we can find constants $\mathfrak{A}, \mathfrak{Q}, \mathfrak{B}$, and $\mathfrak{C}$ such that
\begin{align}\label{eq:up_bound_X}
&\E_{t, x_i}\left[\int_t^s \widehat{\pi}_i^2\left(X_i^{\widehat{\bm{\pi}}_i}(v), Y^{\widehat{\bpi}_{\smn i}}_{\smn i}(v)\right)\drm v \right] \\
&\le 2\rho^2_i p^2_i \int_t^s \E_{t, x_i}[\lvert  X^{\widehat{\bpi}_i}_i(v) \rvert^2]\drm v
+ \mathfrak{A}\e^{\mathfrak{Q}(s-t)} + \mathfrak{B}(s-t)+\mathfrak{C}.
\end{align}
Define
\begin{align}\label{eq:fkp}
\mathfrak{k}_i = (b_i - r) + \xi_i^2 + \sigma_i^2, \qquad \qquad
\mathfrak{P}_i = b_i+r + 2\rho^2_i p^2_i \mathfrak{k}_i.
\end{align}
To ensure that the above constructed strategy $\widehat{\bpi}_i = \{\widehat{\pi}_i\}$ given in \eqref{eq:pi_i1} is admissible, we make the following assumption:
\begin{assumption}\label{assum:1a}
$\la > \max\{\mathfrak{Q},\mathfrak{P}_i, 2r+2\rho_i p_i(b_i-r)+\frac{1}{2}\rho_i p^2_i\}.$
\end{assumption}
\end{bluepar}
\begin{theorem}\label{th:main}
For $i \in \{1, \ldots, n\}$, and let $Q_i \neq 0$ and the competitors investment strategies $\widehat{\bm{\pi}}_j \in \Pi_j, j\neq i$ be given. Then, under Assumption \ref{assum:1a},
 the $i$-th agent's equilibrium investment strategy $\widehat{\bm{\pi}}_i = \{\widehat{\pi}_i\}$ is given by
 \eqref{eq:pi_i1},
 in which \(\rho_i\) is defined in \eqref{eq:rho}, \(p_i=z_i/[\rho_i(b_i-r)]\) with \(z_i\) solving \eqref{eq:p1}, \(q_i\) solves \eqref{eq:q2}, and \(k_{ij}\), for \(j=1,2,3\), are given in \(\eqref{eq:k1}-\eqref{eq:k3}\).
The corresponding equilibrium value function is given by
\begin{align}\label{eq:VV}
V_i( x_i,y_{\smn i})=A_i x_i^2+C_i y_{\smn i}^2+D_i x_iy_{\smn i}+E_i x_i+F_i y_{\smn i}+I_i,
\end{align}
in which the functions $A_i $, $C_i $,\ $D_i$,\ $E_i $, \(F_i\),\ and $I_i $ are given in \eqref{eq:A},\ \eqref{eq:C1},\ \eqref{eq:D1},\ \eqref{eq:E1},\ \eqref{eq:F1}, and \eqref{eq:I1}\ respectively.
\end{theorem}

\begin{remark}\label{re:prop}
	From \eqref{eq:q2} and \eqref{eq:c1}, we observe that \(q_i\) and \(c_i\) are proportional to the competitive parameter \(\phi_i\). Moreover, from \eqref{eq:D1}, we know that \(D_i\) is also proportional to \(\phi_i\). Hence, \(k_{ij}\) (for \(j=1,2\)) given in \eqref{eq:k1} and\eqref{eq:k2} are proportional to \(\phi_i\). In view of this, the equilibrium  feedback  strategy given in \eqref{eq:pi_i} can be divided into two parts: One part is \(\rho_i p_i x_i+k_{i3}\), which is independent of the competition parameter \(\phi_i\); the other part is $ \rho_i q_i y_{\smn i} + k_{i1} \widehat{\sigma\widehat{\pi}} + k_{i2}\widehat{\underline{br}\widehat{\pi}}$, proportional to \(\phi_i\).
\end{remark}
	Define
	\begin{align}
		\psi_{i,n} = \left(1- \frac{\phi_i}{n}\right)\xi_i^2+\sigma_i^2,\quad\quad\Psi_n=\sum_{i=1}^n\frac{\phi_i\sigma_i^2}{n\psi_{i,n}} \le 1,\label{eq:Psi}
	\end{align}
	and
	\begin{align}
		\Phi_n=\sum_{i=1}^n \frac{\mu_{i2}\sigma_i(b_i-r)}{2n\gamma_i\psi_{i,n}},\quad\quad \Theta_{n}=\sum_{i=1}^n \frac{ \sigma_i b_i}{n\gamma_i\psi_{i,n}}.  \label{eq:Phii0}
	\end{align}

In the following corollary, we investigate the equilibrium feedback strategy and the value function when the risk aversion function is state-independent and the risk-free interest rate is zero.

	\begin{corollary}\label{cor:020}
		Suppose that \(\mu_{i1}=0\), \(\mu_{i2}=2\), and \(r=0\).
		Then, for the \(n\)-agent game with a state-independent risk aversion function,
		\begin{itemize}
			\item[$(1)$] if \(\Psi_n<1\), the equilibrium feedback strategy is a constant, given by
			\begin{align}\label{eq:01pi}
				\widehat{\pi}_i =\frac{\phi_i\sigma_i}{\psi_{i,n}} \frac{\Theta_{n}}{1-\Psi_n}+\frac{b_i}{\gamma_i\psi_{i,n}},
			\end{align}
			in which \(\psi_{i,n}\), \(\Psi_n\), and \(\Theta_{n}\) are given in \eqref{eq:Psi} and \eqref{eq:Phii0}.
		\item[$(2)$] if  \(\Psi_n=1\), the equilibrium feedback strategy does not exist.
		\end{itemize}
	\end{corollary}

	\begin{remark}
		In this remark, we consider the relative performance investment problem for agents with CARA preferences. We assume that the financial market is the same as that in Subsection \ref{subsec:model} except that the interest rate \(r = 0\), and the value function is given by
		\begin{align}
			\sup_{\bpi_i\in\bm{\Pi}_i}\mathbb{E}\Bigg[-\exp\Bigg(-\frac{1}{\gamma_i}\Big(X^{\bpi_i}_{i}(\tau)-\phi_i\bar{X}^{\bpi}(\tau)\Big)\Bigg)\Bigg],
		\end{align}
		where the definition of the admissible set \(\bm{\Pi}_i\) is the same as that in Subsection \ref{subsec:model}.
		Then, it is straightforward to show that the equilibrium investment strategy for this problem is also given by \eqref{eq:01pi},
		indicating that when the risk-free interest rate $r=0$, the equilibrium investment strategy under CARA preferences is the same as that of the mean-variance problem with a state independent risk aversion.
	\end{remark}
	
	At the end of this subsection, we consider the situation where the \(i\)-th agent is not concerned about the performance of the others, that is, \(\phi_i = 0\). The result for this case is presented in the following corollary. Define
\[\iota_i:=\frac{1}{2}\frac{(b_i-r)^2}{\xi_i^2+\sigma_i^2}.\]
	\begin{corollary}\label{cor:phi0}
		Suppose \(\phi_i = 0\), then the equilibrium  feedback  strategy equals
		\begin{align}\label{eq:phi0}
			\widehat{\pi}_i(x_i)=\zeta_ix_i + \varsigma_i,
		\end{align}
		in which \(\zeta_i=\mathfrak{z}_i/(b_i-r)\), and \(\mathfrak{z}_i\) solves the following equation
		\begin{align}
			&\bigg(\gamma_i-\frac{\mu_{i1}}{2}\bigg)\mathfrak{z}_i^3+\bigg\{\mu_{i1}\bigg[\frac{1}{2} (\lambda-r)-2\iota_i\bigg]+\gamma_i[2\iota_i-\lambda+2r]\bigg\}\mathfrak{z}_i^2 \\
			&+\{\iota_i[\mu_{i1}(3\lambda-4r)+4r\gamma_i]+\gamma_i(\lambda-r)^2\}\mathfrak{z}_i+\iota_i[2\gamma_i r^2-\mu_{i1}(\lambda-r)(\lambda-2r)]=0, \label{eq:pphi0}
		\end{align}
		with condition
		$\lambda>2r+2\mathfrak{z}_i+\frac{1}{2\iota_i}\mathfrak{z}_i^2,$
		and,
		\begin{align}\label{eq:varsigmai}
			\varsigma_i=-\frac{\mu_{i2}\lambda\iota_i(b_i-r)^{-1}}{\gamma_i \tilde{a}_i[r-\lambda+\mathfrak{z}_i]-a_i\iota_i[2\gamma_i(a_i-1)+\mu_{i1}]},
		\end{align}
		with
		\begin{align}
			a_i&=\frac{\lambda}{\lambda-r-\mathfrak{z}_i},\\
			 \tilde{a}_i&=\frac{\lambda}{\lambda-[2r+2\mathfrak{z}_i+\frac{1}{2\iota_i}\mathfrak{z}_i^2]}.
		\end{align}
		

		Furthermore,  if \(\mu_{i1}= 2\) and $\mu_{i2}=0$, then \(\varsigma_i\) defined in \eqref{eq:varsigmai} equals 0. Hence, in this case, the equilibrium  feedback  strategy equals
		\begin{align}\label{eq:pi00}
			\widehat{\pi}_{i}(x_i)=\nu_ix_i,
		\end{align}
		in which \(\nu_i=\mathfrak{o}_i/(b_i-r)\), and \(\mathfrak{o}_i\) solves the following equation
		\begin{align}
			&\bigg(1-\frac{1}{\gamma_i}\bigg)\mathfrak{o}_i^3+\bigg\{2\iota_i\bigg(1-\frac{2}{\gamma_i}\bigg)+(\lambda-r)\bigg(\frac{1}{\gamma_i}-
			2\bigg)+\lambda\bigg\}\mathfrak{o}_i^2+\bigg\{\frac{2\iota_i}{\gamma_i}(3\lambda-4r)+4r\iota_i+(\lambda-r)^2\bigg\}\mathfrak{o}_i\\
			&+2\iota_i\bigg\{ r^2-\frac{1}{\gamma_i}(\lambda-r)(\lambda-2r)\bigg\}=0,
		\end{align}
		with condition
		$\lambda>2r+2\mathfrak{o}_i+\frac{1}{2\iota_i}\mathfrak{o}_i^2.$
		
	\end{corollary}
	
	\begin{remark}
		In this remark, we consider that there is only one agent and one stock, that is, \(\phi_i = 0\), \(\xi_i = 0\), and \(i = 1\). When the risk-aversion function is state-independent, that is, \(\mu_{i1}=0\) and \(\mu_{i2}=2\), the equilibrium strategy \(\widehat{\pi}_i\) given in \eqref{eq:phi0} is the same as the equilibrium  feedback  strategy shown in Theorem 3.2 in \citet{landriault2018equilibrium}. Another interesting result we find is that, when the risk aversion function is state-dependent, that is, \(\mu_{i1}=2\) and \(\mu_{i2}=0\), the equilibrium strategy \(\widehat{\pi}_i\) given in \(\eqref{eq:pi00}\) is consistent with the equilibrium  feedback  strategy given in Theorem 3.3 in \citet{landriault2018equilibrium}. Note that in our objective function, we introduce a linear form of ``pseudo'' risk aversion, which multiplies the expectation of terminal wealth, although the objective in the one-agent scenario \((\mu_{i1}=2\) and \(\mu_{i2}=0)\) investigated in our paper is the same as that considered in \citet{landriault2018equilibrium}, the approach we adopt to deal with time-inconsistency differs from theirs.
	\end{remark}

Under the general model framework, it remains difficult to solve the system of equations \eqref{eq:pi_i1} for an arbitrary number of heterogeneous agents, that is, the agents interact dynamically and competitively. In the following two subsections, we consider two special cases: (i) the agents are homogeneous, that is, the parameters of all agents are identical; (ii) the agents are heterogeneous with $n=2.$
In both cases, we assume that the individual's hazard rate is a positive constant that is, \(\la(t)\equiv \la\).	

\subsection{Homogeneous n-agent game}\label{sec:homogenous case}
In this subsection, we assume that the agents are homogeneous, meaning that the parameters for all agents in the network are identical. 	Under this setting, the parameters \(\mu_{i1}\),~\(\mu_{i2}\),~\(b_i\),~\(\xi_i\),~\(\sigma_i\),~\(\phi_i\), and \(\gamma_i\) for \(i=1,\cdots,n\) are independent of the index  \(i\) and can thus be abbreviated to \(\mu_{1}\),~\(\mu_{2}\),~\(b\),~\(\xi\),~\(\sigma\),~\(\phi\), and \(\gamma\). Accordingly, all agents are faced with the symmetric optimization problem.  Throughout the following discussion, a subscript \(i\) or \(-i\) indicates a quantity that is specific to agent \(i\), while its absence denotes a common parameter or value shared by all agents.

 In this homogeneous \(n\)-agent game, we postulate that all equilibria take the affine form
 \begin{align}\label{eq:PQRpi}
 	\widehat{\pi}_i(x_i,y_{\smn i}) = \mathfrak{G}x_i + \mathfrak{D}y_{-i} + \mathfrak{R},
 \end{align}
 for any \(i=1,\cdots,n\), where the coefficients \(\mathfrak{G}\), \(\mathfrak{D}\), and \(\mathfrak{R}\) are identical constants across all agents, and need to be further determined. From Theorem \ref{th:main}, for any fixed $i \in \{1,\cdots,n\}$, the equilibrium investment strategy \(\hat{\bm{\pi}}_i\) for the \(i\)-th agent can be expressed by
  \begin{align}\label{eq:pifinalH}
  	\widehat{\pi}_i(x_i,y_{\smn i}) = \rho px_i + \rho qy_{-i} + \delta \widehat{{\pi}}_{-i} + k_3,
  \end{align}
with \(\widehat{{\pi}}_{-i}:=\frac{1}{n}\sum^n_{j=1, j\neq i}\widehat{\pi}_j\), $\delta : = k_{1}\sigma + k_2{(b-r)}$, and $\rho, p, q, k_i, i=1,2,3$ are constants given in Theorem \ref{th:main}. To obtain explicit expression for \(\widehat{\pi}_i\), we first compute the expression of  \(\widehat{{\pi}}_{-i}\), and then we show the values of \(\mathfrak{G}\), \(\mathfrak{D}\), and \(\mathfrak{R}\) below using a fixed-point method.

Note that \(y_{-j}=\frac{1}{n}\sum^n_{k=1, k\neq j}x_k = y_{-i}+\frac{x_i-x_j}{n}\). Then, for any \(j\neq i\), we have
  \begin{align*}
  \widehat{\pi}_j(x_j,y_{\smn j}) & = \mathfrak{G}x_j + \mathfrak{D}y_{-j} + \mathfrak{R} \\
  	& = \mathfrak{G}x_j + \mathfrak{D}\left(y_{-i}+\dfrac{x_i-x_j}{n}\right) + \mathfrak{R}\\
  & =  \left(\mathfrak{G}-\dfrac{\mathfrak{D}}{n}\right)x_j + \mathfrak{D} y_{\smn i} + \frac{\mathfrak{D}}{n}x_i + \mathfrak{R}.
  \end{align*}
  Summing over \(j\neq i\), dividing by \(n\), and using the relation \(\sum^n_{j=1, j\neq i}x_j = ny_{\smn i}\), it follows that
  \begin{align}
  \widehat{{\pi}}_{-i}	= \frac{1}{n}\sum^n_{j=1, j\neq i}\widehat{\pi}_j &= \frac{1}{n}\left(\mathfrak{G}-\frac{\mathfrak{D}}{n}\right)\sum_{j\neq i}x_j + \frac{\mathfrak{D}(n-1)}{n}y_{\smn i} + \frac{(n-1)\mathfrak{D}}{n^2}x_i  + \frac{n-1}{n}\mathfrak{R}\\
& = \left(\mathfrak{G}+\frac{(n-2)\mathfrak{D}}{n}\right)y_{\smn i} + \frac{(n-1)\mathfrak{D}}{n^2}x_i + \frac{n-1}{n}\mathfrak{R}.
  \label{eq:tilde_pi_expr}
  \end{align}
  By substituting \eqref{eq:tilde_pi_expr} into \eqref{eq:pifinalH} and regrouping the terms, we have
  \begin{align*}
  	\widehat{\pi}_i(x_i, y_{\smn i}) &= \left[\rho p+\delta\dfrac{(n-1)\mathfrak{D}}{n^2}\right]x_i + \left[\rho q+\delta\left(\mathfrak{G}+\dfrac{(n-2)\mathfrak{D}}{n}\right)\right]y_{\smn i}
  + \dfrac{n-1}{n} \delta\mathfrak{R} + k_3.
  \end{align*}

  Comparing the coefficients of the above equation with our ansatz \eqref{eq:PQRpi}, we obtain the following linear system of equations for \((\mathfrak{G}, \mathfrak{D}, \mathfrak{R})\):
  \begin{equation}
  	\begin{cases}
  		\mathfrak{G} = \rho p + \frac{n-1}{n^2}\delta\mathfrak{D}, \\
  		\mathfrak{D} = \rho q + \delta\left(\mathfrak{G}+\frac{n-2}{n}\mathfrak{D}\right), \\
  		\mathfrak{R}= \frac{n-1}{n}\delta\mathfrak{R} + k_3.
  	\end{cases}
  \end{equation}

By solving the above system of equations, and assuming that $n + \delta \neq 0$ and $(1 - \delta)n + \delta \neq 0$\footnote{Otherwise, the system of equations will degenerate, leading to no solutions or infinitely many solutions.}, we obtain
  \begin{align}
  		\mathfrak{G} &= \rho p + \frac{(n-1)\delta\rho (q + \delta p)}{(\delta+(1-\delta)n)(\delta+n)}, \label{eq:fG}\\
      \mathfrak{D} &= \frac{n^2 \rho (q + \delta p)}{(\delta+(1-\delta)n)(\delta+n)}, \label{eq:fD}\\
  		\mathfrak{R} &= \frac{nk_3}{\delta + (1 - \delta)n}. \label{eq:fR}
  \end{align}

  We formalize the complete characterization of the equilibria in the following theorem.
 \begin{theorem}
 Assume that $n + \delta \neq 0$ and $(1 - \delta)n + \delta \neq 0$, then for all $i \in \{1, \ldots, n \}$,
the equilibrium investment strategy for the homogeneous \(n\)-agent game is given as follows:
\begin{align}
 	\widehat{\pi}_i(x_i,y_{\smn i}) = \mathfrak{G}x_i + \mathfrak{D}y_{-i} + \mathfrak{R},
 \end{align}
in which \(\mathfrak{G},  \mathfrak{D}, \mathfrak{R}\) are given in \eqref{eq:fG}, \eqref{eq:fD}, \eqref{eq:fR}. \(\rho\) is defined in \eqref{eq:rho}, \(p={z}/{\rho(b-r)}\) with \(z\) solving \eqref{eq:p1}, \(q\) solves \eqref{eq:q2}, and \(k_i\) for \(i=1,2,3\), are given in \eqref{eq:k1}-\eqref{eq:k3}.
\end{theorem}
 \begin{remark}
    In the homogeneous $n$-agent game with state-independent risk aversion, the absence of wealth-dependent risk preferences forces all equilibrium strategies to be constant. Essentially, the optimal investment behavior is identical for all agents ($\hat{\pi}_i \equiv \hat{\pi}$), which reduces the multi-agent interaction into a single agent framework. Furthermore, the homogeneous $n$-agent framework developed in this subsection  provides the foundation for our subsequent asymptotic analysis.  As we will formally demonstrate in the next section, by taking the limit as $n \to +\infty$, the homogeneous $n$-agent game converges to the corresponding mean field game (MFG).
 \end{remark}

    \subsection{Heterogeneous 2-agent game}\label{sec:heterogeneous case}
    In this section, we study the heterogeneous scenario and focus on the game of 2 agents, which allows us to find the explicit equilibrium strategy for each agent. Define
    \[\overline{\sigma\widehat{\pi}}^{(2)}:=\frac{1}{2}\sum_{i = 1}^2\sigma_i\widehat{\pi}_i \quad  \text{and} \quad \overline{\underline{br}\widehat{\pi}}^{(2)}:=\frac{1}{2}\sum_{i = 1}^2(b_i - r)\widehat{\pi}_i.
    \]
    Then, by noting that $y_{-i} = \dfrac{x_2}{2}$, we rewrite the linear system of equations \eqref{eq:pi_i} for $i=1,2$ as follows:
\begin{align}
		\widehat{\pi}_1( x_1,x_2) & =\Xi_1^{-1}\left[\rho_1 p_1x_1+ \dfrac{\rho_1 q_1}{2}x_2+k_{11}\overline{\sigma\widehat{\pi}}^{(2)}+k_{12}\overline{\underline{br}\widehat{\pi}}^{(2)}+k_{13}\right], \label{eq:hpi1}\\
		\widehat{\pi}_2( x_1,x_2) & =\Xi_2^{-1}\left[\dfrac{\rho_2 q_2 }{2}x_1 + \rho_2 p_2x_2 +k_{21}\overline{\sigma\widehat{\pi}}^{(2)}+k_{22}\overline{\underline{br}\widehat{\pi}}^{(2)}+k_{23}\right], \label{eq:hpi2}
\end{align}
where  \begin{align}\label{eq:Ni}
\Xi_i=1 + \dfrac{k_{i2}(b_i-r)+k_{i1}\sigma_i}{2},   ~~~~\mbox{for} ~~~i=1,2.
\end{align}

	We now derive the closed forms of \(\overline{\sigma\widehat{\pi}}^{(2)}\) and \(\overline{\underline{br}\widehat{\pi}}^{(2)}\).  Through straightforward calculations, we obtain two equations for \(\overline{\sigma\widehat{\pi}}^{(2)}\) and \(\overline{\underline{br}\widehat{\pi}}^{(2)}\):
	\begin{align}
		&\overline{\sigma\rho px}^{(2)}+\overline{\sigma\rho qy}^{(2)}+\big(\overline{\sigma k_1}^{(2)}-1\big)\overline{\sigma\widehat{\pi}}^{(2)}+\overline{\sigma k_2}^{(2)}\cdot\overline{\underline{br}\widehat{\pi}}^{(2)}+\overline{\sigma k_3}^{(2)}=0,\label{eq:system1}\\
		& \overline{\underline{br}\rho px}^{(2)}+\overline{\underline{br}\rho qy}^{(2)}+\overline{\underline{br}k_1}^{(2)}\overline{\sigma\widehat{\pi}}^{(2)}+\big(\overline{\underline{br}k_2}^{(2)}-1\big)\overline{\underline{br}\widehat{\pi}}^{(2)}+\overline{\underline{br}k_3}^{(2)}=0,\label{eq:system2}
	\end{align}
in which
\begin{align}
\overline{\sigma\rho px}^{(2)}&=\frac{\sigma_1\rho_1 p_1x_1}{2\Xi_1} + \frac{\sigma_2\rho_2 p_2x_2}{2\Xi_2},& \overline{\sigma\rho qy}^{(2)}&=\frac{\sigma_1\rho_1 q_1x_2}{2\Xi_1} + \frac{\sigma_2\rho_2 q_2x_1}{2\Xi_2},\label{eq:coeff1}\\
\overline{\underline{br}\rho px}^{(2)}&=\frac{(b_1-r)\rho_1 p_1x_1}{2\Xi_1}+\frac{(b_2-r)\rho_2 p_2x_2}{2\Xi_2},& \overline{\underline{br}\rho qy}^{(2)}&=\frac{(b_1-r)\rho_1 q_1x_2}{2\Xi_1} + \frac{(b_2-r)\rho_2 q_2x_1}{2\Xi_2},\\ \label{eq:coeff2} \\
\overline{\sigma k_j}^{(2)}&= \frac{\sigma_1k_{1j}}{2\Xi_1}+\frac{\sigma_2k_{2j}}{2\Xi_2},
&\overline{\underline{br}k_j}^{(2)}&= \frac{(b_1-r)k_{1j}}{2\Xi_1}+\frac{(b_2-r)k_{2j}}{2\Xi_2},\label{eq:coeff3}
	\end{align}
	for \(j=1,2,3\).
	
	By further computations, we derive that if \(\big(\overline{\underline{br}k_2}^{(2)}-1\big)\big(\overline{\sigma k_1}^{(2)}-1\big)-\overline{\sigma k_2}^{(2)}\cdot\overline{\underline{br}k_1}^{(2)} \neq 0\), the closed forms of \(	\overline{\sigma\widehat{\pi}}^{(2)}\) and \(	\overline{\underline{br}\widehat{\pi}}^{(2)}\) are given by
		\begin{align}
		&\overline{\sigma\widehat{\pi}}^{(2)} =\frac{\big(1-\overline{\underline{br}k_2}^{(2)}\big)\big(\overline{\sigma\rho px}^{(2)}+\overline{\sigma\rho qy}^{(2)}+\overline{\sigma k_3}^{(2)}\big)+\overline{\sigma k_2}^{(2)}\big(\overline{\underline{br}\rho px}^{(2)}+\overline{\underline{br}\rho qy}^{(2)}+\overline{\underline{br} k_3}^{(2)}\big)}{\big(\overline{\underline{br}k_2}^{(2)}-1\big)\big(\overline{\sigma k_1}^{(2)}-1\big)-\overline{\sigma k_2}^{(2)}\cdot\overline{\underline{br}k_1}^{(2)}},\\
        \label{eq:overpisigma} \\
&\overline{\underline{br}\widehat{\pi}}^{(2)}=\frac{\big(1-\overline{\sigma k_1}^{(2)}\big)\big(\overline{\underline{br} \rho px}^{(2)}+\overline{\underline{br}\rho qy}^{(2)}+\overline{\underline{br} k_3}^{(2)}\big)+\overline{\underline{br}k_1}^{(2)}\big(\overline{\sigma\rho px}^{(2)}+\overline{\sigma\rho qy}^{(2)}+\overline{\sigma k_3}^{(2)}\big)}{\big(\overline{\underline{br}k_2}^{(2)}-1\big)\big(\overline{\sigma k_1}^{(2)}-1\big)-\overline{\sigma k_2}^{(2)}\cdot\overline{\underline{br}k_1}^{(2)}}.\\ \label{eq:overpibr}
	\end{align}

\blue{
	By plugging \eqref{eq:overpisigma} and \eqref{eq:overpibr} into \eqref{eq:hpi1} and \eqref{eq:hpi2}, we obtain the equilibrium feedback strategy \footnote{\blue{In this case, the values of $\overline{\sigma\widehat{\pi}}^{(2)}$ and $\overline{\underline{br}\widehat{\pi}}^{(2)}$ are uniquely determined. However, the value of the equilibrium feedback strategy $\widehat{\bf{\pi}}$ is not unique, see the item 1 in Remark \ref{remark:2.5}}. }.
}If \(\big(\overline{\underline{br}k_2}^{(2)}-1\big)\big(\overline{\sigma k_1}^{(2)}-1\big)-\overline{\sigma k_2}^{(2)}\cdot\overline{\underline{br}k_1}^{(2)} = 0\), the linear equations \eqref{eq:system1} and \eqref{eq:system2} are solvable if and only if there exists a constant \(\kappa\) such that
	\begin{align}
		&\overline{\sigma\rho px}^{(2)}+\overline{\sigma \rho qy}^{(2)}+\overline{\sigma k_3}^{(2)}=\kappa\left[\overline{\underline{br} \rho px}^{(2)}+\overline{\underline{br} \rho qy}^{(2)}+\overline{\underline{br}k_3}^{(2)}\right],\label{eq:kappa1}\\
		&\overline{\sigma k_1}^{(2)}-1=\kappa\overline{\underline{br} k_1}^{(2)},\quad\quad
		\overline{\sigma k_2}^{(2)}=\kappa\left(\overline{\underline{br} k_2}^{(2)}-1\right). \label{eq:kappa2}
	\end{align}
	
	We summarize the above analysis in the following theorem, which presents the equilibrium  feedback  strategy for a two-agent mean-variance problem.
	\begin{theorem}[n=2]\label{th:main2}
		Let   \(Q_i\neq 0\), and \(\Xi_i\neq 0\), for $i=1,2.$ Then, for the mean-variance problem,
		\begin{itemize}
			\item[$(1)$]  when \(\big(\overline{\underline{br}k_2}^{(2)}-1\big)\big(\overline{\sigma k_1}^{(2)}-1\big)-\overline{\sigma k_2}^{(2)}\cdot\overline{\underline{br}k_1}^{(2)}\neq 0\), the expression for \(\widehat{\pi}_i(x_1,x_2)\), $i=1,2$ is given by \eqref{eq:hpi1} and \eqref{eq:hpi2}, in which \(\rho_i\) is defined in \eqref{eq:rho}, \(p_i=z_i/[\rho_i(b_i-r)]\) with \(z_i\) solving \eqref{eq:p1}, \(q_i\) solves \eqref{eq:q2}, and \(k_{ij}\), for \(j=1,2,3\), are given in \(\eqref{eq:k1}-\eqref{eq:k3}\). The expressions for \(\overline{\sigma\widehat{\pi}}^{(2)}\) and \(\overline{\underline{br} \widehat{\pi}}^{(2)}\) are presented in \eqref{eq:overpisigma} and \eqref{eq:overpibr}, respectively.
\item[$(2)$]  when \(\big(\overline{\underline{br}k_2}^{(2)}-1\big)\big(\overline{\sigma k_1}^{(2)}-1\big)-\overline{\sigma k_2}^{(2)}\cdot\overline{\underline{br}k_1}^{(2)}= 0\),
\begin{itemize}
\item[$(a)$] if there exist a constant \(\kappa\) such that \eqref{eq:kappa1} and \eqref{eq:kappa2} hold, the equilibrium feedback strategy exists, and the equilibrium feedback strategy \(\widehat{\pi}_i\) (for \(i=1,2\)) satisfies the equation \eqref{eq:system1} or \eqref{eq:system2}.
				
					\item[$(b)$]  if no such constant \(\kappa\) exists, the equilibrium feedback strategy does not exist.
			\end{itemize}
	
		\end{itemize}
	\end{theorem}
	
\begin{remark}\label{remark:2.5}
If \(\big(\overline{\underline{br}k_2}^{(2)}-1\big)\big(\overline{\sigma k_1}^{(2)}-1\big)-\overline{\sigma k_2}^{(2)}\cdot\overline{\underline{br}k_1}^{(2)}\neq 0\), because \(p_i=z_i/[\rho_i(b_i-r)]\) with \(z_i\) satisfying \eqref{eq:p1}, then there can be more than one but at most three linear equilibrium feedback strategies,
\begin{bluepar}
which is consistent with \citet{landriault2018equilibrium}. To ensure uniqueness, we impose an admissibility criterion: we select the specific real root that satisfies the conditions of Theorem \ref{th:verification} and yields the  maximal mean-variance objective value. This selection uniquely identifies the economically relevant equilibrium strategy.
\end{bluepar}
If \(\big(\overline{\underline{br}k_2}^{(2)}-1\big)\big(\overline{\sigma k_1}^{(2)}-1\big)-\overline{\sigma k_2}^{(2)}\cdot\overline{\underline{br}k_1}^{(2)}= 0\)\footnote{\blue{The degenerate condition arising from the coupling within the coefficient system, which prevents closed-form expressions for the equilibrium  coefficients. In the numerical implementation, parameter configurations leading to such degeneracy are excluded.}}, and there exists a constant \(\kappa\) such that \eqref{eq:kappa1} and \eqref{eq:kappa2} hold, then there may be infinitely many linear equilibrium feedback strategies.
\end{remark}

In the following corollary, we state the limiting equilibrium feedback strategy as \(\lambda\to \infty\). Recall the definitions of $\psi_{i,n}, \Psi_n$, and $\Phi_n$ given in \eqref{eq:Psi} and \eqref{eq:Phii0}.
\begin{align}\label{eq:Upsilon}
\Upsilon_2=\frac{\mu_{11}\sigma_1(b_1-r)}{4\gamma_1\psi_{1,2}}x_1 + \frac{\mu_{21}\sigma_2(b_2-r)}{4\gamma_2\psi_{2,2}}x_2.
\end{align}
	\begin{corollary}\label{coro:1.8}
		As \(\lambda\to \infty\), the limiting equilibrium feedback strategy equals
		\begin{itemize}
			\item[$(1)$] If \(\Psi_2 < 1\), then
             \begin{align}
				\widehat{\pi}_1(x_1, x_2)& = \left(\mathfrak{p}_1 + \frac{\phi_1\sigma_1}{\psi_{1,2}}\mathfrak{q}_1 \right) x_1 + \frac{\phi_1\sigma_1}{\psi_{1,2}}\mathfrak{p}_2 x_2 + \mathfrak{t}_1, \label{eq:pi1.8}\\
               \widehat{\pi}_2(x_1, x_2)& =\left(\mathfrak{p}_2 + \frac{\phi_2\sigma_2}{\psi_{2,2}}\mathfrak{q}_2\right) x_1 + \frac{\phi_2\sigma_2}{\psi_{2,2}}\mathfrak{q}_1 x_2 + \mathfrak{t}_2, \label{eq:pi2.8}
			\end{align}
where, for $i=1,2$,
\begin{align}
\mathfrak{p}_i = \dfrac{\mu_{i1}(b_i-r)}{2\gamma_i\psi_{i,2}}, ~~
\mathfrak{q}_i = \dfrac{\mu_{i1}\sig_i(b_i-r)}{4\gamma_i\psi_{i,2}(1 - \Psi_2) }, ~~
\mathfrak{t}_i =  \dfrac{\phi_i \sig_i}{\psi_{i,2}}\cdot \dfrac{\Phi_2}{1 - \Psi_2}  + \dfrac{\mu_{i2}(b_i-r)}{2\gamma_i\psi_{i,2}}.
\end{align}
	\item[$(2)$] If \(\Psi_2=1\), then
\begin{itemize}
\item[$(a)$] if \(\Upsilon_2+\Phi_2=0\), there are infinitely many equilibrium feedback strategies, which are given by
    \begin{align}\label{eq:picor1.80}
    \widehat{\pi}_i(x_i)=\frac{\mu_{i1}(b_i-r)}{2\gamma_i\psi_{i,2}}x_i+\frac{\phi_i\sigma_i}{\psi_{i,2}}\overline{\sigma\widehat{\pi}}^{(2)}+\frac{\mu_{i2}(b_i-r)}{2\gamma_i\psi_{i,2}},
    ~~~~\mbox{for}~~i=1,2.
    \end{align}
    where \(\overline{\sigma\widehat{\pi}}^{(2)}\) is an arbitrary real number.
 \item[$(b)$] if \(\Upsilon_2+\Phi_2\neq0\), then the equilibrium feedback strategy does not exist.
\end{itemize}
\end{itemize}
\end{corollary}
\begin{proof}
 See Appendix \ref {app:C}.
\end{proof}

	\begin{remark}
		When the risk aversion function is state-independent and the investment time horizon is infinitesimally short, that is, \(\mu_{i1}=0\), \(\mu_{i2}=2\), and \(\lambda\to \infty\). The equilibrium  feedback  strategy is
		\[
		\widehat{\pi}_i=\frac{\phi_i\sigma_i}{\psi_{i,2}} \cdot
		\frac{\Phi_2}{1 -\Psi_2}+\frac{b_i-r}{\gamma_i\psi_{i,2}},
		\]
		which is the same as the equilibrium investment strategy obtained in Theorem 4 in \citet{guan2022time} when there are no insurance terms, \(T-t=0\), and \(n=2\) in their model.
	\end{remark}

	\section{The mean-field game (MFG)}\label{sec:MF}
	In this section, we study the equilibrium feedback strategy for the multi-agent game as the number of agents tends to infinity. In this case,  the \(n\)-agent game converges to a mean-field game with a continuum of agents. In Subsection~\ref{subsec:MF-model}, we introduce the type space, the objective, and the admissible set of the equilibrium strategy. In Subsection~\ref{subsec:MBR}, given the average processes $\mathbb{E}[(b-r)\pi(t)\mid\mathcal{F}_t^B]$ and $\mathbb{E}[\sigma\pi(t)\mid\mathcal{F}_t^B]$, we present the definition of the time-consistent mean-field equilibrium (MFE) feedback strategy for the representative agent and provide a general verification theorem, from which we explicitly find the best response of the representative agent when the hazard rate of the random horizon is a constant. Under the homogeneous setting, in Subsection~\ref{subsec:MFH}, we follow a fixed-point argument and characterize the equilibrium feedback strategy of the whole system. Furthermore, we establish the almost sure convergence of the Nash equilibrium of the homogeneous $n$-agent game to the MFE as $n\to\infty$.

	\subsection{Model formulation}\label{subsec:MF-model}
	We start by defining the type vector and the type space. For the \(n \)-agent game, we denote the parameters of the \(i\)-th agent by the following type vector
	\begin{align}
		\eta_i := (x_{i,0}, b_i, \xi_i, \sigma_i, \phi_i, \gamma_i, \mu_{i1}, \mu_{i2}), \qquad \text{for}~~i=1,\ldots,n.
	\end{align}
	where \(x_{i,0}\) is the wealth at the time \(t=0\).
	We allow these parameters to depend on \(n\), but for simplicity, we omit the index \(n\) to the type vector \(\eta_i\). These type vectors induce an empirical measure \(m_n\), called the type distribution, which is a probability measure on the type space
	\begin{align*}
			\mathcal{Z}^e:=\mathbb{R}\times(0,\infty)\times [0,\infty)^2\times[0,1]\times\mathbb{R}^+\times[0,\infty)^2,
	\end{align*}
	given by
	\begin{align}
		m_n(A) = \frac{1}{n} \sum_{i=1}^{n} {\bf{1}}_A(\eta_i) \quad \text{for Borel sets } A \subset \mathcal{Z}^e. \label{eq:mn}
	\end{align}
	
We assume that \textcolor{blue}{\(\{\eta_i\}_{i=1}^\infty\) are i.i.d random elements in \(\mathcal{Z}^e\),} and then by the law of large numbers, the empirical measure \(m_n\) converges weakly to a limit measure \(m\) on the space \(\mathcal{Z}^e\) in the sense that
	\begin{align}
		\int_{\mathcal{Z}^e} f \, \drm m_n \to \int_{\mathcal{Z}^e} f \, \drm m,  \label{eq:Mweak_converge}
	\end{align}
	for any bounded continuous function \(f\) on \(\mathcal{Z}^e\) when the number of agents tend to infinity, that is \(n\to\infty\). We define \(\eta = (x_0, b, \xi, \sigma, \phi, \gamma, \mu_1, \mu_2)\) as a random variable with this limiting distribution \(m\) on the space \(\mathcal{Z}^e\) and \(\xi+\sigma>0\) almost surely. The distribution of this random type vector reflects the law of type vectors of a continuum of agents, where a single realization of the random type vector represents the type assigned to a single representative agent.
	
	We start to formulate the MFG, and then derive the expression of the equilibrium feedback strategy for the MFG. Let \((\Omega, \mathcal{F},\mathbb{F}=\{\mathcal{F}_t\}_{t\geq 0}, \mathbb{P})\) be a filtered complete probability space satisfying the usual conditions, which supports a two-dimensional Brownian motion \((W, B)\) and a random type vector \(\eta = (x_0, b, \xi, \sigma, \phi, \gamma, \mu_1, \mu_2)\) with \(\xi+\sigma>0\) almost surely. Here, \(\eta\)  is measurable with respect to the filtration \(\mathcal{F}_0\) and is independent of both \(W\) and \(B\). Let \(\mathbb{F}^B = \{\mathcal{F}_t^B\}_{t \geq 0}\) denote the natural filtration generated by the Brownian motion \(B\). Moreover, we assume that all expectations appearing in this section are finite.

	We assume that the financial market is the same as that considered in Subsection \ref{subsec:model} except that the idiosyncratic risk factors for the stocks are replaced by \(W=\{W(t)\}_{t\geq 0}\). Furthermore, we assume that the risk-free interest rate \(r\) is identical for all agents. Let $\pi(t)$ represent the amount of money that the representative agent invests in the risky stock. Then, the dynamic of the representative agent's wealth process $X^{\boldsymbol{\pi}}=\{X^{{\boldsymbol{\pi}}}(t)\}_{t\geq 0}$ under strategy ${\boldsymbol{\pi}} = \{\pi(t)\}_{t \ge 0}$ satisfies the following SDE:
	\begin{align}\label{eq:MMFX}
		\drm X^{\bpi}(t)&= [rX^{\bpi}(t)  +(b-r)\pi(t)]\drm t+\xi\pi(t)\drm W(t)+\sigma\pi(t)\drm B(t),\quad X^{\bpi}(0)=x_0.
	\end{align}

We also assume the time horizon is random. Let \(\tau\) be a random variable defined in the above probability space, which is independent of the Brownian motion \((W,B)\) and the type vector \(\eta\), and we use \(\lambda(t)\)  to denote the deterministic, time-dependent hazard rate function of \(\tau\), with the corresponding survival probability given by
		\begin{align*}
			\mathbb{P}(\tau>t)=\exp \Big(-\int_0^t \lambda(s)\drm s\Big).
		\end{align*}
\begin{bluepar}
In the time-inconsistent problem, let $x$ denote the initial wealth of the representative agent at time $t$. Driven by the SDE \eqref{eq:MMFX}, the agent's wealth process $X^{\bpi}$ depends on both the idiosyncratic noise $W$ and the common noise $B$.	Within this mean-field game (MFG) framework, the population average wealth is defined conditional on the filtration $\mathbb{F}^B$, which is generated by the common noise $B$. (see \citet{lacker2019mean}, page 1015). Thus, the agent anticipates that the average wealth $\bar{X}^{\bpi}$ is an $\mathbb{F}^B$-adapted process, defined as $$\bar{X}^{\bpi}(t)=\mathbb{E}[X^{\bpi}(t)\mid\mathcal{F}_t^B],$$
 for all $t \ge 0$.
 	\end{bluepar}
 	
 To address the investment problem of the representative agent, we define the admissible strategy as follows
\begin{definition}[Admissible strategy]\label{def:mf_admissible}
A strategy \(\bm{\pi} = \{\pi(t)\}_{t\geq 0}\) is called admissible if it satisfies the following conditions:
\begin{itemize}
\item[$(1)$] $\boldsymbol{\pi}$ is an $\mathbb{F}$-progressively measurable real-valued process, and \( \int_0^t |\pi(s)|^2\, \drm s<\infty\) almost surely for all \( t \geq 0 \).
\item[$(2)$] The SDE \eqref{eq:MMFX} has a unique strong solution, such that
$\E_{0, {x_0 }} \big[\left({X}^{\bpi } (\tau)\right)^2\big]<\infty$.
\end{itemize}
We use \( \boldsymbol{\Pi}^m \) to denote the set of the representative agent's all admissible strategies.
\end{definition}

We assume that the representative agent has a mean-variance investment preference with a state-dependent risk aversion. Specifically, the objective of the representative agent is to  maximize
		\begin{align}	\label{eq:MMFMFJ}
			J(t, x, \bar{x}; {\bm{\pi}}) = (\mu_1 x+\mu_2)\mathbb{E}_{t,x,\bar{x},t+}[X^{\bpi}(\tau)-\phi\bar{X}^{\bpi}(\tau)]-\gamma \text{Var}_{t,x,\bar{x},t+}[X^{\bpi}(\tau)-\phi\bar{X}^{\bpi}(\tau)],
		\end{align}
		in which \(\mathbb{E}_{t,x,\bar{x},t+}[\cdot]:=\mathbb{E}[\cdot|{X}^{\bpi}(t)={x}, \bar{X}^{\bpi}(t)=\bar{x}, \tau > t]\).

\subsection{Best responses of the MFG}\label{subsec:MBR}
		In this subsection, we derive the best response of the representative agent when the average processes $\mathbb{E}[(b-r)\pi(t)\mid\mathcal{F}_t^B]$ and $\mathbb{E}[\sigma\pi(t)\mid\mathcal{F}_t^B]$ are given. We first show the definition of the mean-field equilibrium for the representative agent.
		
\begin{definition}[Time-consistent mean-field equilibrium]\label{def:StrongMFE}
For an admissible feedback strategy \(\widehat{\bm{\pi}}=\{\widehat{\pi}(t,X^{\widehat{\bpi}}(t),\bar{X}^{\widehat{\bpi}}(t))\}_{t\geq 0}\), fix a time $t\ge0$, a positive number \(\epsilon > 0\), and a real number $\pi$; then, define a feedback strategy $\bpi^{\eps}$ by
			\[
			\pi^{\eps}(s, x, \bar{x}) =
			\begin{cases}
				\pi, & \text{for } ~~   {t \leq s < (t+\epsilon) \wedge \tau,} \\
				\widehat{\pi}(s, x, \bar{x}), & \text{for }~~ (t+\epsilon) \wedge \tau \leq s \leq \tau.
			\end{cases}
			\]		
			If for all \((t, x, \bar{x}) \in [0,\infty) \times \mathbb{R} \times \mathbb{R}\),
			\begin{align}
				\liminf_{\epsilon \to 0^+} \frac{J(t, x, \bar{x}; \widehat{\bm{\pi}}) - J(t, x, \bar{x}; {\bm{\pi}^{\eps}})}{\epsilon} \geq 0,
			\end{align}
			then $\widehat{\bpi}$ is a time-consistent mean-field equilibrium (MFE) feedback strategy,
			and  the corresponding equilibrium value function is given by
			\[
			V(t, x, \bar{x}) = J(t, x, \bar{x}; \widehat{\bm{\pi}}).
			\]
	
		\end{definition}

\begin{bluepar}	
We proceed to formalize the average wealth process  $\bar{X}^{\bm{\pi}}$. By taking the conditional expectation of \eqref{eq:MMFX} with respect to $\mathcal{F}_t^B$, $\bar{X}^{\bm{\pi}}$ satisfies the following SDE:
	\begin{align}\label{eq:MMFbarX}
		\drm \bar{X}^{\bpi}(t)
		= \bigl[r\bar{X}^{\bpi}(t) + \overline{\underline{br}\pi}(t)\bigr]\drm t
		+ \overline{\sigma\pi}(t)\,\drm B(t), \quad t \geq 0,
	\end{align}
	where
\[
\overline{\underline{br}\pi}(t) := \mathbb{E}[(b-r)\pi(t)\mid\mathcal{F}_t^B]~~\mbox{and}~~\overline{\sigma\pi}(t) := \mathbb{E}[\sigma\pi(t)\mid\mathcal{F}_t^B].
\]
\end{bluepar}

 Both of them are $\mathcal{F}_t^B$-adapted processes.
\begin{remark}
Because the risk aversion in our model is state-dependent, the corresponding MFE strategy inherently relies on the current wealth of the representative agent. Furthermore, since the dynamics of the wealth process given in \eqref{eq:MMFX} are driven by the common noise $B$, the mean-field process itself is $\mathcal{F}_t^B$-adapted. This fundamentally distinguishes our framework from those in \citet{lacker2019mean} and \citet{bo2024mean}, which primarily focus on state-independent equilibria. In contrast, we seek an MFE in the Markovian feedback form $\widehat{\pi}(t, X^{\widehat{\bpi}}(t),\bar{X}^{\widehat{\bpi}}(t))$.
	\end{remark}

We now present a verification theorem to explicitly characterize the MFE. As the proof closely parallels that of Theorem \ref{th:verification}, it is omitted. We subsequently apply this theorem to explicitly characterize the MFE for the scenario when the random horizon has a constant hazard rate.
Moreover, we assume that the $\mathcal{F}_t^B$-adapted average processes $\overline{\sigma\pi}(t)$ and $\overline{\underline{br}\pi}(t)$ are provided. Under this assumption, each agent in the mean-field game will solve its own single optimization problem against these given average inputs, and the best response is characterized via a system of extended HJB equations.
	
For a given function $g \in C^{1,2,2}([0,\infty) \times \mathbb{R} \times \mathbb{R})$, the infinitesimal generator $\mathcal{L}_{\scriptscriptstyle m}^{{\pi}}$ is defined by:
	\begin{align}
		\mathcal{L}_{\scriptscriptstyle m}^{{\pi}} g(t,x,\bar{x}) = & g_t + g_x[rx + (b - r)\pi] + g_{\bar{x}}[r\bar{x} + \overline{\underline{br} \pi}(t)] \\
		& + \frac{1}{2}g_{xx}[\xi^2\pi^2 + \sigma^2 \pi^2] + \frac{1}{2}g_{\bar{x}\bar{x}}\overline{\sigma \pi}(t)^2 + g_{x\bar{x}}\sigma \pi \overline{\sigma \pi}(t).
	\end{align}
	Define a function \(f \in C^{1,2,2,2,2}([0,\infty) \times \mathbb{R}^4)\) by
	\begin{align}
		f(t, x, y, u, v) = (\mu_{1} x + \mu_{2}) u - \gamma (v - u^2),
	\end{align}
	where $u$ and $v$ are two functions depending on $(t,x,\bar{x})$.
	
	\begin{theorem}[Verification theorem]\label{th:Mverification}
		Suppose that there exist three real-valued functions $V^{\scriptscriptstyle m}(t,x,\bar{x})$, $G^{\scriptscriptstyle m}(t,x,\bar{x})$, $H^{\scriptscriptstyle m}(t,x,\bar{x})\in C^{1,2,2}({[0,\infty)}\times\mathbb{R}\times\mathbb{R})$ satisfying the following conditions:
		\begin{itemize}
			\item[$(1)$] For all \( (t,x,\bar{x})\in {[0,\infty)}\times \mathbb{R}\times \mathbb{R}\),
			\begin{align}				
				&\lambda(t)\left\{V^{\scriptscriptstyle m}(t,x,\bar{x})-(\mu_{1}x+\mu_{2})\big(x-\phi\bar{x}\big)+\gamma \left[G^{\scriptscriptstyle m}(t,x,\bar{x})-(x-\phi\bar{x})\right]^2\right\}\\	&=\sup_{{\pi}\in\boldsymbol{\Pi}^{\scriptscriptstyle m}}\left\{\mathcal{L}^{{\pi}}_{\scriptscriptstyle m}V^{\scriptscriptstyle m}(t,x,\bar{x})-\mathcal{L}^{{\pi}}_{\scriptscriptstyle m}f(t,x,\bar{x},G^{\scriptscriptstyle m},H^{\scriptscriptstyle m})+f_G(t,x,\bar{x},G^{\scriptscriptstyle m},H^{\scriptscriptstyle m})\mathcal{L}^{{\pi}}_{\scriptscriptstyle m} G^{\scriptscriptstyle m}(t,x,\bar{x})\right.\\
				&\left.\quad\quad\quad\quad+f_H(t,x,\bar{x},G^{\scriptscriptstyle m},H^{\scriptscriptstyle m})\mathcal{L}^{{\pi}}_{\scriptscriptstyle m}H^{\scriptscriptstyle m}(t,x,\bar{x})\right\}. \label{eq:MMF-HJB}
			\end{align}
			Let $\widehat{\pi}$ denote the function of $(t, x, \bar{x})$ that attains the supremum in \eqref{eq:MMF-HJB}.		
			\item[$(2)$] For all \( (t,x,\bar{x})\in {[0,\infty)}\times \mathbb{R}\times \mathbb{R}\),
			\begin{align}\label{eq:M-LG}
				\lambda(t)\left( G^{\scriptscriptstyle m}(t,x,\bar{x})-(x-\phi \bar{x})\right)=\mathcal{L}^{\widehat{{\pi}}}_{\scriptscriptstyle m}G^{\scriptscriptstyle m}(t,x,\bar{x}).
			\end{align}
			
			\item[$(3)$] For all \( (t,x,\bar{x})\in {[0,\infty)}\times \mathbb{R}\times \mathbb{R}\),
			\begin{align}\label{eq:M-LH}
				\lambda(t)\left[H^{\scriptscriptstyle m}(t,x,\bar{x})-\big(x-\phi \bar{x}\big)^2\right]=\mathcal{L}_{\scriptscriptstyle m}^{{\widehat{\pi}}}H^{\scriptscriptstyle m}(t,x,\bar{x}).
			\end{align}
			
			\item[$(4)$] A transversality condition holds. Specifically, for all \( (t,x,
			\bar{x})\in {[0,\infty)}\times \mathbb{R}\times \mathbb{R}\),
			\begin{align}\label{eq:M-T}
				\lim_{s\to \infty}\mathbb{E}_{t,x,\bar{x}}\Big[e^{-\int_t^s \lambda(v)\drm v} j(s,X^{\bm{\widehat{\pi}}}(s),\bar{X}^{\bm{\widehat{\pi}}}(s))\Big]=0,
			\end{align}
			for \(j=G^{\scriptscriptstyle m}\), \(H^{\scriptscriptstyle m}\), \(V^{\scriptscriptstyle m}\), and \(f\).
		\end{itemize}

		Then
		\begin{align*}
			V^{\scriptscriptstyle m}(t,x,\bar{x})&=	(\mu_1 x+\mu_2)\mathbb{E}_{t,x,\bar{x},t+}[X^{\widehat{\bm{\pi}}}(\tau)-\phi\bar{X}^{\bm{\widehat{\pi}}}(\tau)]-\gamma Var_{t,x,\bar{x},t+}[X^{\widehat{\bm{\pi}}}(\tau)-\phi\bar{X}^{\bm{\widehat{\pi}}}(\tau)],\\
			G^{\scriptscriptstyle m}(t,x,\bar{x})&=\mathbb{E}_{t,x,\bar{x},t+}[X^{\widehat{\bm{\pi}}}(\tau)-\phi \bar{X}^{\bm{\widehat{\pi}}}(\tau)],\\
			H^{\scriptscriptstyle m}(t,x,\bar{x})&=\mathbb{E}_{t,x,\bar{x},t+}\left[\big(X^{\widehat{\bm{\pi}}}(\tau)-\phi \bar{X}^{\bm{\widehat{\pi}}}(\tau)\big)^2\right],
		\end{align*}
		in which $X^{\widehat{\bm{\pi}}}$ is the representative agent's wealth process  under the MFE strategy $\widehat{\bm{\pi}}$. \footnote{Notice that the above extended HJB system is random due to the $\mathcal{F}_0$-measurable type parameters.}
	\end{theorem}
	
	We use Theorem \ref{th:Mverification} to deduce the mean-field equilibrium feedback strategy and the corresponding value function for the representation agent. Furthermore, we assume that the hazard rate \(\lambda(t)\) is a positive constant, that is \(\lambda(t)\equiv \lambda\). Under this assumption, the problem is time homogeneous and the three PDEs given in Theorem \ref{th:Mverification} degenerate to a system of ODEs.
		
	Similar to Subsection \ref{subsec:verif}, we assume that $\lambda(t)$ is a positive constant, that is, $\lambda(t) \equiv \lambda$. Under this assumption, we apply Theorem \ref{th:Mverification} to derive the MFE and the corresponding value function for the MFG.
	Define
	\begin{align}
		Q_{\scriptscriptstyle m}:=\gamma \tilde{a}_{\scriptscriptstyle m}[r-\lambda+\rho_{\scriptscriptstyle m} p_{\scriptscriptstyle m}(b-r)]-a_{\scriptscriptstyle m}\rho_{\scriptscriptstyle m}(b-r)^2[2\gamma(a_{\scriptscriptstyle m}-1)+\mu_{1}]\neq 0,\label{eq:MFQ}
		\end{align}
	and assume that \(Q_{\scriptscriptstyle m} \neq 0\), and
	\begin{align}
		a_{\scriptscriptstyle m}=&\frac{\lambda}{\lambda-r-\rho_{\scriptscriptstyle m} p_{\scriptscriptstyle m}(b-r)},\label{eq:MFa1}\\
		 \tilde{a}_{\scriptscriptstyle m}=&\frac{\lambda }{\lambda-[2r+2\rho_{\scriptscriptstyle m} p_{\scriptscriptstyle m}(b-r)+\frac{1}{2}\rho_{\scriptscriptstyle m} {p_{\scriptscriptstyle m}}^2]}.\label{eq:MFbara}\\
		\rho_{\scriptscriptstyle m}=&\frac{1}{2(\xi^2+\sigma^2)}, \label{eq:MFrho}
	\end{align}
	and	 \(p_{\scriptscriptstyle m}=z_{\scriptscriptstyle m}/[\rho_{\scriptscriptstyle m}(b-r)]\), with \(z_{\scriptscriptstyle m}\) satisfying
	\begin{align}	
		&\big(\gamma-\mu_{1}/2\big){z^3_{\scriptscriptstyle m}}+\big\{\mu_{1}\big[ (\lambda-r)/2-2\rho_{\scriptscriptstyle m}(b-r)^2\big]+\gamma[2\rho_{\scriptscriptstyle m}(b-r)^2-\lambda+2r]\big\}{z^2_{\scriptscriptstyle m}}\\
		&+\{\rho_{\scriptscriptstyle m}(b-r)^2[\mu_{1}(3\lambda-4r)+4r\gamma]+\gamma(\lambda-r)^2\}z_{\scriptscriptstyle m}+\rho_{\scriptscriptstyle m}(b-r)^2[2\gamma r^2-\mu_{1}(\lambda-r)(\lambda-2r)]=0.\\\label{eq:MF-p1}
	\end{align}	
	Define
	\begin{align}
		k_{1,\scriptscriptstyle m}&:={Q^{-1}_{\scriptscriptstyle m}}\rho_{\scriptscriptstyle m}\sigma(r-\lambda)(D_{\scriptscriptstyle m}-\mu_{1}c_{\scriptscriptstyle m}-2\gamma a_{\scriptscriptstyle m}c_{\scriptscriptstyle m}),\label{eq:MFk1}\\	
		k_{2,\scriptscriptstyle m}&:={Q^{-1}_{\scriptscriptstyle m}}\rho_{\scriptscriptstyle m} (b -r)[2\gamma (a_{\scriptscriptstyle m} -1)c_{\scriptscriptstyle m} +\mu_{1}c_{\scriptscriptstyle m} -D_{\scriptscriptstyle m} ],\label{eq:MFk2}\\
		k_{3,\scriptscriptstyle m}&:=-{Q^{-1}_{\scriptscriptstyle m}}\mu_{2}\rho_{\scriptscriptstyle m} \lambda(b -r).\label{eq:MFk3}
	\end{align}
	in which
	\begin{align}
		c_{\scriptscriptstyle m} =&\frac{a_{\scriptscriptstyle m}\rho_{\scriptscriptstyle m} q_{\scriptscriptstyle m}(b-r)-\lambda\phi}{\lambda-r},\label{eq:MFc1}\\
		D_{\scriptscriptstyle m}=&\frac{\mu_1 c_{\scriptscriptstyle m}r-\gamma  \tilde{a}_{\scriptscriptstyle m}\rho_{\scriptscriptstyle m} p_{\scriptscriptstyle m}q_{\scriptscriptstyle m}+2(a_{\scriptscriptstyle m}-1)\gamma \lambda(c_{\scriptscriptstyle m}+\phi)+\lambda\mu_1\phi}{2r-\lambda},\label{eq:MFD1}
	\end{align}
	and \(q_{\scriptscriptstyle m}\) satisfies
	\begin{align}
		&\{\gamma \tilde{a}_{\scriptscriptstyle m}(\lambda-r)[2r-\lambda+(b-r)\rho_{\scriptscriptstyle m} p_{\scriptscriptstyle m}]-[2\gamma\lambda(a_{\scriptscriptstyle m}-1)+\mu_1(\lambda-r)]a_{\scriptscriptstyle m}\rho_{\scriptscriptstyle m}(b-r)^2\}q_{\scriptscriptstyle m}\\
		&=-2\gamma\lambda\phi r(a_{\scriptscriptstyle m}-1)(b-r).\label{eq:MFq2}
	\end{align}

\begin{assumption}\label{assum:mf}
		$\lambda >
		2r+2\rho_{\scriptscriptstyle m} p_{\scriptscriptstyle m}(b-r)+\tfrac{1}{2}\rho_{\scriptscriptstyle m} p_{\scriptscriptstyle m}^2$.
	\end{assumption}
	
\begin{theorem}
Let \(Q_{\scriptscriptstyle m}\neq 0\), and the average processes \(\overline{\underline{br}\widehat{\pi}}(t)\) and  \(\overline{\sigma\widehat{\pi}}(t)\) be given. Then, under Assumption \ref{assum:mf}, if the following candidate investment strategy ${\bpi^o} = \{{\pi^o}\}$ is admissible, that is, ${\bpi^o} \in \boldsymbol{\Pi}^m$,
\begin{align}\label{eq:M-pi_i}
{\pi^o}(x, \bar{x}) = \rho_{\scriptscriptstyle m} p_{\scriptscriptstyle m} x + \rho_{\scriptscriptstyle m} q_{\scriptscriptstyle m} \bar{x} + k_{1,m} \overline{\sigma\widehat{\pi}} + k_{2,m}\overline{\underline{br}\widehat{\pi}}+k_{3,m},
\end{align}
then, the representative agent's equilibrium investment strategy \(\widehat{\bpi}\) equals ${\bpi^o}$,
in which \(\rho_{\scriptscriptstyle m}\) is given in \eqref{eq:MFrho},  \(p_{\scriptscriptstyle m}=z_{\scriptscriptstyle m}/[\rho_{\scriptscriptstyle m}(b-r)]\) with \(z_{\scriptscriptstyle m}\) solving  \eqref{eq:MF-p1}, \(q_{\scriptscriptstyle m}\) is the solution of \eqref{eq:MFq2}, and
			\(k_j\) $($for \(j=1,2,3\)$)$ are given in \eqref{eq:MFk1}, \eqref{eq:MFk2}, and \eqref{eq:MFk3}, respectively.
			Moreover, the corresponding equilibrium value function is given by
			\begin{align}\label{eq:MF-VV}
				V(x,\bar{x})=A_{\scriptscriptstyle m} x^2+C_{\scriptscriptstyle m} \bar{x}^2+D_{\scriptscriptstyle m} x\bar{x}+E_{\scriptscriptstyle m} x+F_{\scriptscriptstyle m} \bar{x}+I_{\scriptscriptstyle m},
			\end{align}
			in which
			\begin{align}
				A_{\scriptscriptstyle m}&=\mu_1 a_{\scriptscriptstyle m}+\gamma {a^2_{\scriptscriptstyle m}}{-\gamma  \tilde{a}_{\scriptscriptstyle m},}\\
				C_{\scriptscriptstyle m}&=\frac{\lambda\gamma(c_{\scriptscriptstyle m}+\phi)-\frac{\rho_{\scriptscriptstyle m}}{2}\gamma \tilde{a}_{\scriptscriptstyle m}{q^2_{\scriptscriptstyle m}}}{2r-\lambda},\\
				E_{\scriptscriptstyle m}&=\frac{\mu_1\alpha_{\scriptscriptstyle m}r-D_{\scriptscriptstyle m}\overline{\underline{br}\widehat{\pi}}-\rho_{\scriptscriptstyle m}\gamma  \tilde{a}_{\scriptscriptstyle m}\varrho_{\scriptscriptstyle m}p_{\scriptscriptstyle m}+2(a_{\scriptscriptstyle m}-1)\gamma\lambda \alpha_{\scriptscriptstyle m}-\mu_2\lambda}{r-\lambda},\\
				F_{\scriptscriptstyle m}&=\frac{2\lambda\gamma \alpha_{\scriptscriptstyle m}(c_{\scriptscriptstyle m}+\phi)+\lambda\mu_2\phi-\rho_{\scriptscriptstyle m}\gamma \tilde{a}_{\scriptscriptstyle m}\varrho_{\scriptscriptstyle m}q_{\scriptscriptstyle m}-2C_{\scriptscriptstyle m}\overline{\underline{br}\widehat{\pi}}}{r-\lambda},\\
				I_{\scriptscriptstyle m}&=\frac{(C_{\scriptscriptstyle m}-\gamma {c^2_{\scriptscriptstyle m}})\overline{\sigma\widehat{\pi}}^2-\lambda \gamma {\alpha^2_{\scriptscriptstyle m}}+{\frac{\rho_{\scriptscriptstyle m}}{2}\gamma \tilde{a}_{\scriptscriptstyle m}\varrho^2_{\scriptscriptstyle m}}+F_{\scriptscriptstyle m}\overline{\underline{br}\widehat{\pi}}}{\lambda},
			\end{align}
			with
			\begin{align}
				\alpha_{\scriptscriptstyle m}&=\la^{-1}\big({c_{\scriptscriptstyle m}\overline{\underline{br}\widehat{\pi}}+a_{\scriptscriptstyle m}\rho_{\scriptscriptstyle m} \varrho_{\scriptscriptstyle m}(b-r)}\big), \label{eq:MFe1}\\
				\varrho_{\scriptscriptstyle m}&={\left[k_1\overline{\sigma\widehat{\pi}}+k_2\overline{\underline{br}\widehat{\pi}}+k_3\right]/\rho_{\scriptscriptstyle m}}.
			\end{align}
	\end{theorem}

	Define
	\begin{align}\label{eq:MFPsi}
		\Psi=\E\Big[\frac{\phi\sigma^2}{\xi^2+\sigma^2}\Big] \le 1,\quad  {\Theta=\E\Big[\frac{\sigma b}{\gamma(\xi^2+\sigma^2)}\Big]}.
	\end{align}

	The following corollary states the MFE and the value function when the risk aversion function is state-independent and the risk-free interest rate equals zero, that is, \blue{\(\mu_{1}=0\)}, \(\mu_{2}=2\), and \(r = 0\).
	\begin{corollary}\label{cor:MF-020}
		Suppose that \(\mu_{1}=0\), \(\mu_{2}=2\), and \(r=0\),
		and then for the mean-variance problem in the mean-field game with a state-independent risk aversion function,
		\begin{itemize}
			\item[$(1)$] if \(\Psi<1\), the MFE is an $\mathcal{F}_0$-measurable random variable, given by
			\begin{align}\label{eq:MF-01pi}
				\widehat{\pi}=\frac{\phi\sigma}{\xi^2+\sigma^2}\cdot \frac{\Theta}{1-\Psi}+\frac{b}{\gamma(\xi^2+\sigma^2)},
			\end{align}		
			where \(\Psi\) and \(\Theta\) are given in \eqref{eq:MFPsi}.
			\item[$(2)$] if  \(\Psi=1\), the MFE does not exist.
		\end{itemize}
	\end{corollary}
	\begin{remark}
		In this remark, we consider the relative performance investment problem for a continuum of agents with CARA preferences. We assume that the financial market is the same as that in Subsection  \ref{subsec:MF-model} except that the interest rate \(r = 0\), and the value function is given by
		\begin{align}
			\sup_{\bpi\in\bm{\Pi}^{\scriptscriptstyle m}}\E\Bigg[-\exp\Bigg(-\frac{1}{\gamma}\Big(X^{\bpi}(\tau)-\phi\bar{X}^{\bpi}(\tau)\Big)\Bigg)\Bigg],
		\end{align}
		in which the definition of the admissible set $\bm{\Pi}^m$ is the same as that in Subsection \ref{subsec:MF-model}.
		Then, one can show that the equilibrium investment strategy for this problem is also given by \eqref{eq:MF-01pi}.
		Thus, as the $n$-agent game scenario, we also obtain that when the risk-free interest rate $r=0$, the equilibrium investment strategy under CARA preferences is the same as that of the mean-variance problem with a state-independent risk-aversion within MFG framework.
	\end{remark}

\subsection{Homogeneous mean-field game}\label{subsec:MFH}
In this subsection, we restrict attention to the homogeneous population, that is, all agents share an identical
type vector, i.e., \(\eta_i \equiv \eta_0\) for any \(i = 1, \ldots, n\). The empirical measure \(m_n\) associated with the type vectors \(\{\eta_i\}_{i=1}^n\) is defined in \eqref{eq:mn}. Because all agents are identical, $m_n$ converges weakly to the Dirac measure \(\delta_{\eta_0}\) as \(n \to \infty\). Thus, in the homogeneous mean-field game, the representative agent is characterized by a constant vector $\eta_0 = (x, b, \xi, \sigma, \phi, \gamma, \mu_1, \mu_2) \in \mathcal{Z}^e$. We show the equilibrium investment strategy of the homogeneous mean-field game in the following theorem. Define
\begin{align}
\delta_{\scriptscriptstyle m} =  \sigma k_{1,m}+(b-r)k_{2,m}.
\end{align}
\begin{theorem}\label{th:homo_mefield}
Assume that \(Q_{\scriptscriptstyle m}\neq 0\) and $1-\delta_{\scriptscriptstyle m}\neq 0$, then
the equilibrium investment strategy of the homogeneous mean-field game is given as follows:
\begin{align}\label{eq:M-Hpifinal}
\widehat{\pi}(x,\bar{x}) =\rho_{\scriptscriptstyle m} p_{\scriptscriptstyle m} x+\frac{\rho_{\scriptscriptstyle m}q_{\scriptscriptstyle m}+\delta_{\scriptscriptstyle m}\rho_{\scriptscriptstyle m} p_{\scriptscriptstyle m}}{1-\delta_{\scriptscriptstyle m}}\bar{x}+\frac{k_{3,m}}{1-\delta_{\scriptscriptstyle m}},
\end{align}
in which \(\rho_{\scriptscriptstyle m}\) is given in \eqref{eq:MFrho},  \(p_{\scriptscriptstyle m}=z_{\scriptscriptstyle m}/[\rho_{\scriptscriptstyle m}(b-r)]\) with \(z_{\scriptscriptstyle m}\) solving  \eqref{eq:MF-p1}, \(q_{\scriptscriptstyle m}\) is the solution of \eqref{eq:MFq2}, and
\(k_j\) $($for \(j=1,2,3\)$)$ are given in \eqref{eq:MFk1}, \eqref{eq:MFk2}, and \eqref{eq:MFk3}, respectively.
\end{theorem}
\begin{proof}
See Appendix \ref{app:homo_mf}.
\end{proof}
We close this section with the following proposition showing that the Nash equilibrium of the homogeneous \(n\)-agent game converges to the MFE of the mean-field game as the number of agents goes to infinity.

\begin{proposition}\label{th:hre}
Let the \(n\)-agent game system be homogeneous, so that \(\eta_i = \eta_0\) for all
\(i = 1, \ldots, n\), where \(\eta_0 \in \mathcal{Z}^e\) is the common type vector in the \(n\)-agent game.  In the homogeneous mean-field game, $\eta_0$ is identified as the
representative agent's type vector $\eta$.  We assume that  the parameters are bounded in the type space. If $n + \delta \neq 0$, $(1 - \delta)n + \delta \neq 0$ and \(1-\delta_{\scriptscriptstyle m}\neq 0\), then we have
		\begin{align}
			\widehat{\pi}_i\to\widehat{\pi},~~{\rm{ a.s.}} ~~\text{as}~~ n\to\infty. \label{eq:limit_pi}
		\end{align}
	\end{proposition}
\begin{proof}
	See Appendix \ref{app:hre}.
\end{proof}	
	
\begin{remark}
The convergence result in Proposition~\ref{th:hre} is established for the homogeneous case, in which all agents share the identical type vector \(\eta_0\). In a heterogeneous setting, however, the components of the initial type vector \(\eta\) are drawn from a non-degenerate distribution rather than being uniform constants. Thus, obtaining an explicit, closed-form solution for the mean-field game becomes highly intractable.
		
Furthermore, the best response strategies derived for both the \(n\)-agent game and the limiting mean-field game rely heavily on the assumption that the average interaction terms are given exogenously. This essentially requires evaluating the convergence of one approximate equilibrium against another. Because the explicit dynamics of these average terms remain elusive in the heterogeneous case, providing a precise quantification of the approximation error or establishing a rigorous convergence rate is exceedingly difficult. We leave the formal proof of convergence under the fully heterogeneous framework to future research.
	\end{remark}


\section{Numerical Analysis}\label{sec:NA}	
In this section, we provide several numerical examples to illustrate the impacts of competitive components on the equilibrium feedback strategy for the homogeneous \(n\)-agent game. We first detail our simulation methodology, including the parameter calibration and the equilibrium selection criterion. Subsequently, we conduct a sensitivity analysis with respect to key model parameters  \(\phi\), \(\gamma\), \(\mu_1\) and \(\mu_2\).  Finally, we present the convergence analysis of the finite-player equilibrium strategy as the number of agents \(n\) tends to infinity.

\begin{bluepar}	
\subsection{Numerical Simulation and Equilibrium Selection}\label{sub:NS}

In this subsection, we present the numerical implementation for the homogeneous \(n\)-agent game. We consider a homogeneous system of $n = 1000$ agents. Under the homogeneous assumption, all agents share an identical type vector $( b, \xi, \sigma, \phi, \gamma, \mu_1, \mu_2)$, which is drawn from the uniform distribution \(\mathcal{U}\big([0.1,0.2]\times[0.01,0.05]^2\times[0,1]\times[1,10]\times[0,1]^2\big)\). The initial wealth of each agent is drawn independently from $\mathcal{U}[100,1000]$. Additionally, we set the risk-free interest rate $r=0.02$.

 We focus on the $500$-th agent for our subsequent analysis, whose initial wealth $x_{500} = 919.8344$  is a realization drawn from \(\mathcal{U}[100,1000]\). For this selected agent, the parameters are fixed at $(b, \xi, \sigma, \phi,\gamma, \mu_1, \mu_2) = (0.1375, 0.0480, 0.0393, 0.1465, 6.3879, 0.1560, 0.1560)$. The hazard rate $\lambda$ is fixed throughout the simulation
and calibrated to satisfy Assumption~\ref{assum:1a}. Since $p_i$ is determined by the cubic equation \eqref{eq:p1}, whose coefficients depend explicitly on $\lambda$, the admissibility condition in  Assumption~\ref{assum:1a} imposes a joint constraint on the pair $(\lambda, p)$. We search over a uniform grid $\lambda \in [0.3, 20]$ with step size $0.00985$, yielding approximately $2000$ candidate values. For each candidate
$\lambda$, we first solve \eqref{eq:p1} to obtain all real roots, retain those satisfying the condition \eqref{eq:conditionP}, and select the one maximizing $V^{\mathrm{loc}}(x,y_{\smn i}) \approx A x_i^2 + Cy_{\smn i}^2$, where $A$ and $C$ are given in \eqref{eq:A} and \eqref{eq:C1}. We then verify whether the selected \(p\) satisfies Assumption~\ref{assum:1a}. Among all admissible $\lambda$, we adopt the smallest one, that is, $\lambda=0.3$, which is fixed in the subsequent analysis.

Given the realized wealth $x_{500}$ drawn from $\mathcal{U}[0.1,1]$ and the empirical mean of the remaining agents
$y_{\smn 500}=\frac{1}{1000}\sum_{j=1,j\neq 500}^{1000}x_j$, we evaluate the equilibrium strategy $\widehat{\pi}_{500}(x_{500}, y_{\smn 500})$ for each parameter of interest $(\phi, \gamma, \mu_1, \mu_2)$ with all others fixed at baseline.  At each grid point, the coefficients \(p\), \(q\), \(Q\), and \(k_j\) (for \(j=1,2,3\)) are fully recomputed,  and both Assumption \ref{assum:1a} and the non-degeneracy conditions $Q \neq 0$, $n+ \delta \neq 0$, and \((1-\delta)n+\delta\neq 0\) of Theorem~\ref{th:main}  are  re-verified.

For the sensitivity analysis, we vary each parameter of $(\phi, \gamma, \mu_1, \mu_2)$ over its respective range while holding the others fixed at their baseline values. For each parameter variation, we repeat the simulation and the above verification processes.  Any configuration failing the admissibility conditions is discarded. Finally, to evaluate the impact of the mean-field interaction, we compare the equilibrium strategy $\widehat{\pi}_{500}$ under competition ($\phi = 0.1465$) against the non-competitive optimal strategy $\tilde{\pi}_{500}$ under no competition ($\phi \equiv 0$).
	
\end{bluepar}	
	
	\subsection{Sensitivity Analysis}\label{subsec:SA}
	In this subsection, we investigate the sensitivity of the \(500\)-th agent's equilibrium strategy $\widehat{\pi}_{500}$ with respect to the key parameters $\phi$, $\gamma$, $\mu_{1}$, and $\mu_{2}$, with all remaining parameters fixed at the baseline values presented in Subsection~\ref{sub:NS}.
	
	\blue{In Figures \ref{fig:phi} through \ref{fig:mu2}, we plot graphs for the equilibrium investment strategies with and without competition. The solid blue line represents the equilibrium strategy $\widehat{\pi}_{500}$ in the presence of competition. Conversely, the dashed red line depicts the optimal strategy $\tilde{\pi}_{500}$ in the scenario  without competition ($\phi \equiv 0$). By definition, this non-competitive optimal strategy is equivalent to the equilibrium strategy evaluated at $\phi=0$, which is consistent with the analytical solution derived in Corollary \ref{cor:phi0}.}

	\begin{figure}[htbp]
		\centering 
		\begin{minipage}[t]{0.45\textwidth}
			\centering
		\includegraphics[width=\textwidth]{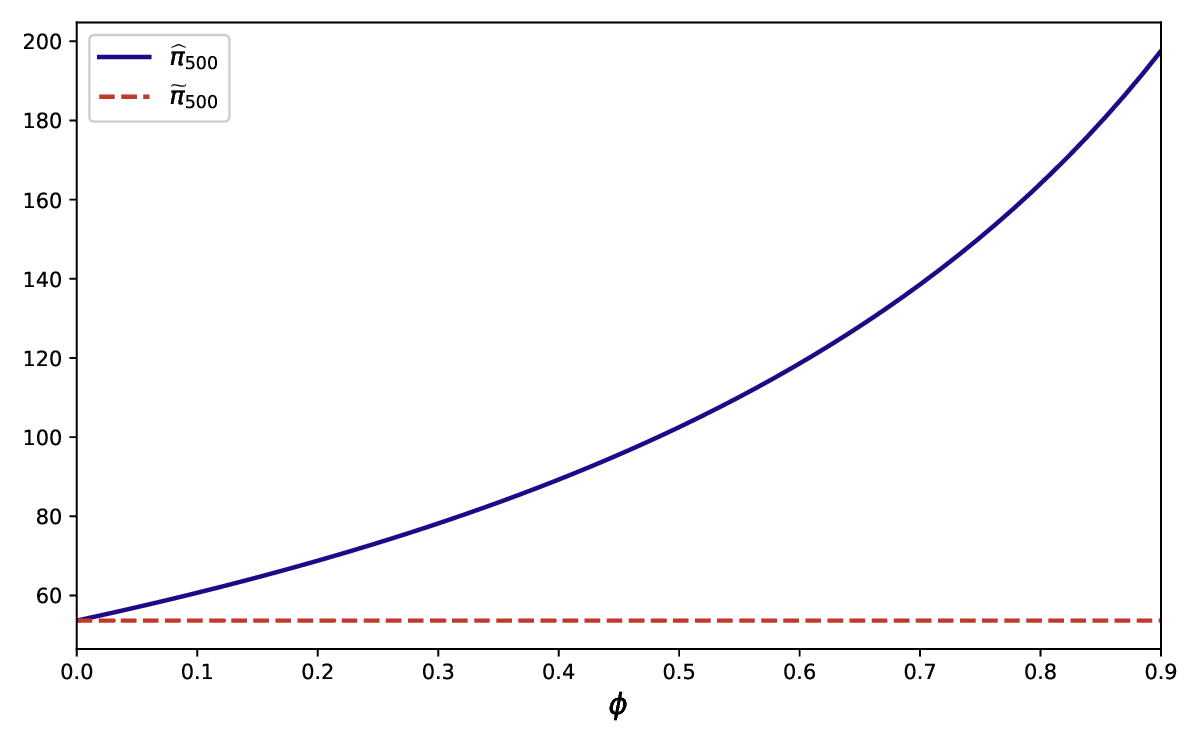}  
		\caption{\small{$\widehat{\pi}_{500}$ and $\tilde{\pi}_{500}$ vary with $\phi$.}}
		\label{fig:phi}
		\end{minipage}
		\hfill 
		\begin{minipage}[t]{0.45\textwidth}
			\centering
		\includegraphics[width=\textwidth]{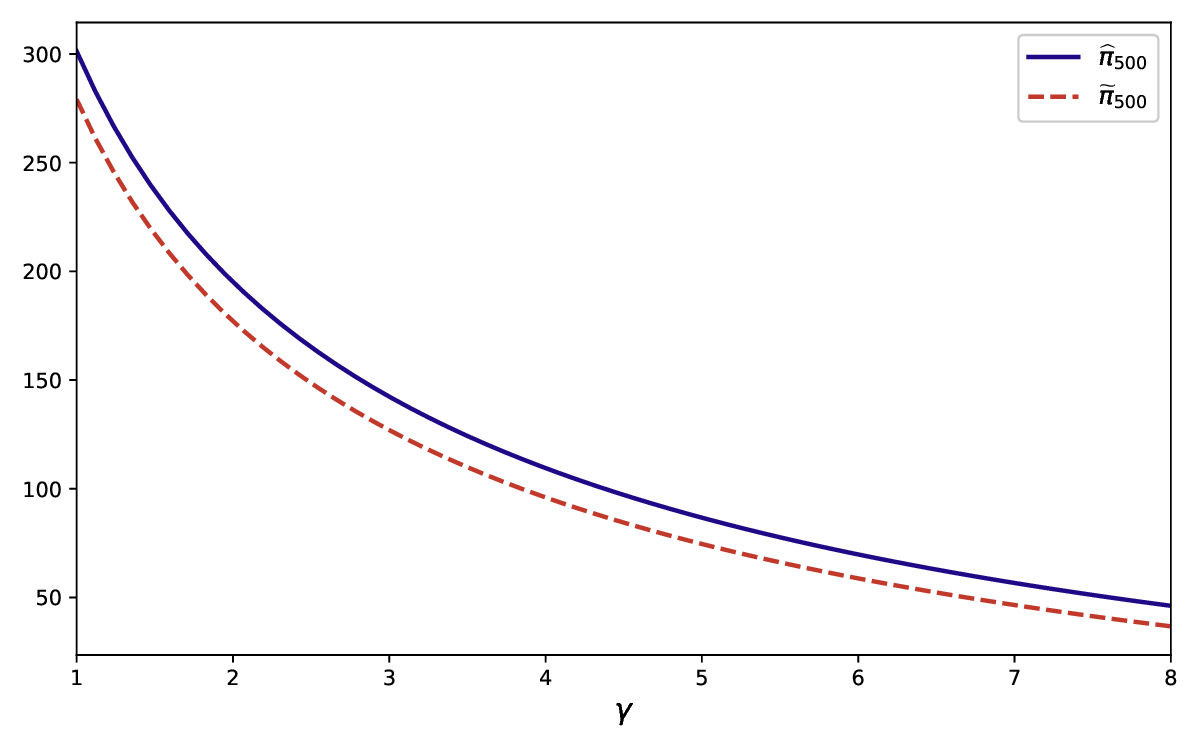}  
		\caption{\small{$\widehat{\pi}_{500}$ and $\tilde{\pi}_{500}$ vary with $\gamma$.}}
		\label{fig:gamma}
		\end{minipage}
	\end{figure}
	
	Figure \ref{fig:phi} illustrates the impact of the competition parameter $\phi $ on the equilibrium strategy \(\widehat{\pi}_{500}\).  The figure shows a clear monotonic relationship:   $\widehat{\pi}_{500}$ increases with $\phi$. Economically, stronger competitive pressure
	heightens each agent's concern for relative performance, driving the agent to take aggressive positions  in risky assets to maximize his/her objective. By contrast, in the absence of competition ($\phi= 0$), the optimal strategy $\tilde{\pi}_{500}$ remains constant, as the agent behaves purely according to his/her own risk-return trade-off without strategic considerations. This observation is consistent with  the results documented in \citet{lacker2019mean} and \citet{bo2024mean}.

	Figure \ref{fig:gamma} illustrates the impact of the risk aversion parameter $\gamma$ on the equilibrium strategy \(\widehat{\pi}_{500}\) and the non-competitive optimal strategy \(\tilde{\pi}_{500}\).  As shown in the figure, both  strategies exhibit a strictly decreasing and convex trend with respect to $\gamma$. Intuitively, as the agent becomes more risk-averse, he/she naturally reduces risky asset exposure to lower the variance of his/her wealth. The convexity further reveals a diminishing marginal sensitivity: a slight increase in risk aversion triggers a massive portfolio rebalancing for an aggressive investor, whereas a highly conservative investor's strategy is inherently bounded and thus less responsive to further increases in $\gamma$. Crucially, the competitive strategy $\widehat{\pi}_{500}$ consistently lies above the non-competitive optimal strategy $\tilde{\pi}_{500}$. This gap implies that competition induces an additional incentive for risk-taking. Specifically, the fear of falling behind peers partially offsets the agent's absolute risk aversion, compelling even highly conservative investors to maintain larger risky positions than they would in isolation.

	\begin{figure}[htbp]
		\centering 
		\begin{minipage}[t]{0.45\textwidth}
			\centering
		\includegraphics[width=\textwidth]{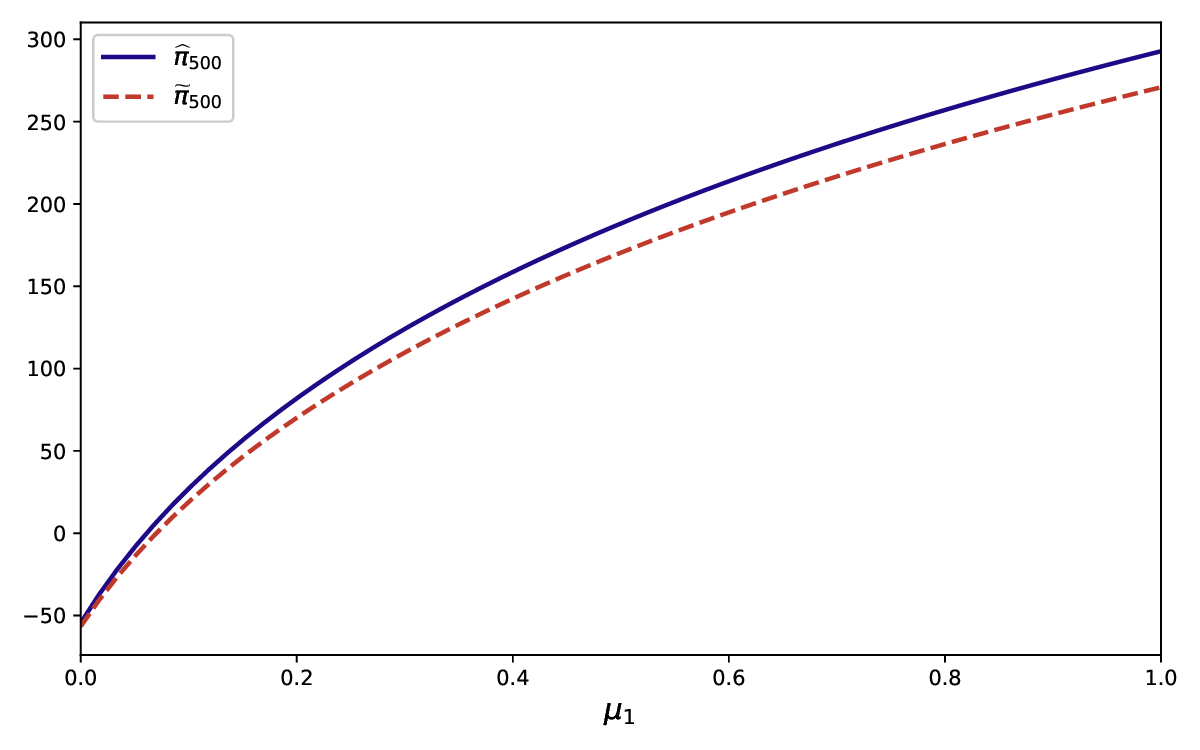}  
		\caption{\small{$\widehat{\pi}_{500}$ and $\tilde{\pi}_{500}$ vary with $\mu_1$.}}
		\label{fig:mu1}
		\end{minipage}
		\hfill 
		\begin{minipage}[t]{0.45\textwidth}
			\centering
		\includegraphics[width=\textwidth]{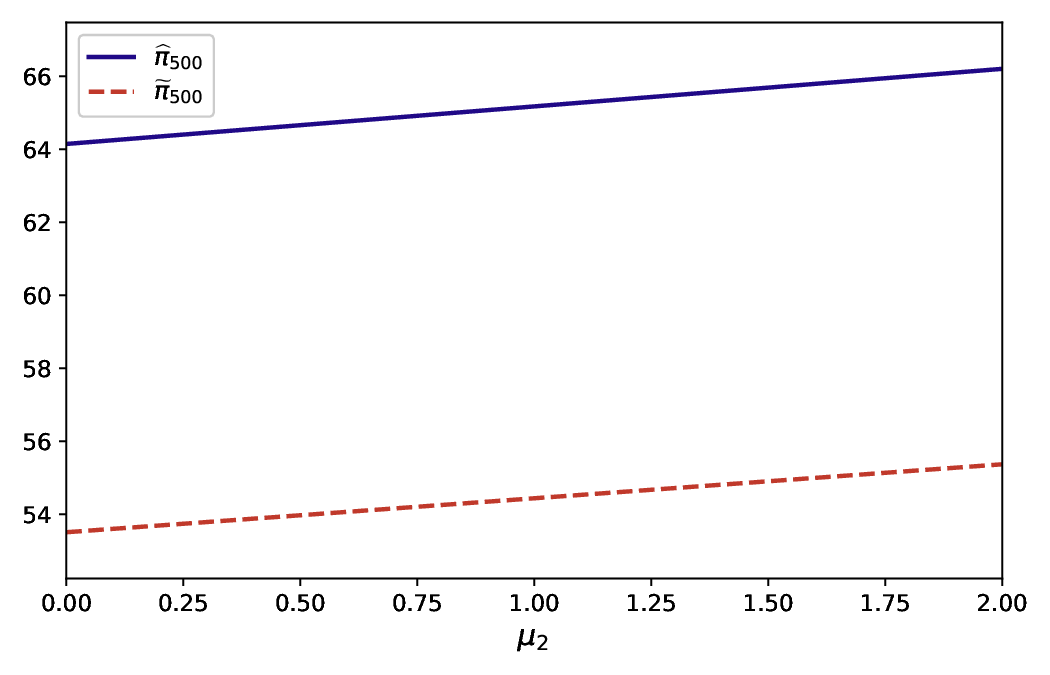}  
		\caption{\small{$\widehat{\pi}_{500}$ and $\tilde{\pi}_{500}$ vary with $\mu_2$.}}
		\label{fig:mu2}
		\end{minipage}
	\end{figure}
	
Figures \ref{fig:mu1} and \ref{fig:mu2} illustrate the influence of $\mu_{1}$ and \(\mu_2\) on the equilibrium strategy \(\widehat{\pi}_{500}\) and the non-competitive optimal strategy \(\tilde{\pi}_{500}\), respectively. In these figures, we observe that \(\widehat{\pi}_{500}\) and \(\tilde{\pi}_{500}\) increase as $\mu_1$ and $\mu_2$ increase. Since $\mu_1$ and $\mu_2$ influence the value function through the term $\mu_1 x_i + \mu_2$, higher values of either parameter strengthens the incentive to maximize expected terminal wealth, thereby encouraging more aggressive investment in the risky asset. Moreover, \(\widehat{\pi}_{500}\) and \(\tilde{\pi}_{500}\) are more sensitive to changes in $\mu_1$ than those in $\mu_2$. Specifically, in Figure \ref{fig:mu1}, the equilibrium investment strategies increase concavely with $\mu_1$, whereas in Figure \ref{fig:mu2}, they increase linearly with $\mu_2$. Furthermore, in both figures, the competitive strategy $\widehat{\pi}_{500}$ consistently lies above the non-competitive optimal strategy  $\tilde{\pi}_{500}$.  The persistent gap between the two curves confirms that competition induces additional risk-taking beyond the non-competitive baseline across all values of both $\mu_1$ and $\mu_2$.
	\begin{bluepar}
	\subsection{Convergence Analysis}\label{subsec:CA}
	In this subsection, we examine the convergence of the time-consistent equilibrium strategy \(\widehat{\pi}^{(n)}\) in the \(n\)-agent game toward the mean-field equilibrium (i.e. MFE) \(\widehat{\pi}_m\) of the mean-field game as the number of agents \(n\) tends to infinity.
	
	\begin{figure}[h!]
		\centering
		\includegraphics[width=0.5\textwidth]{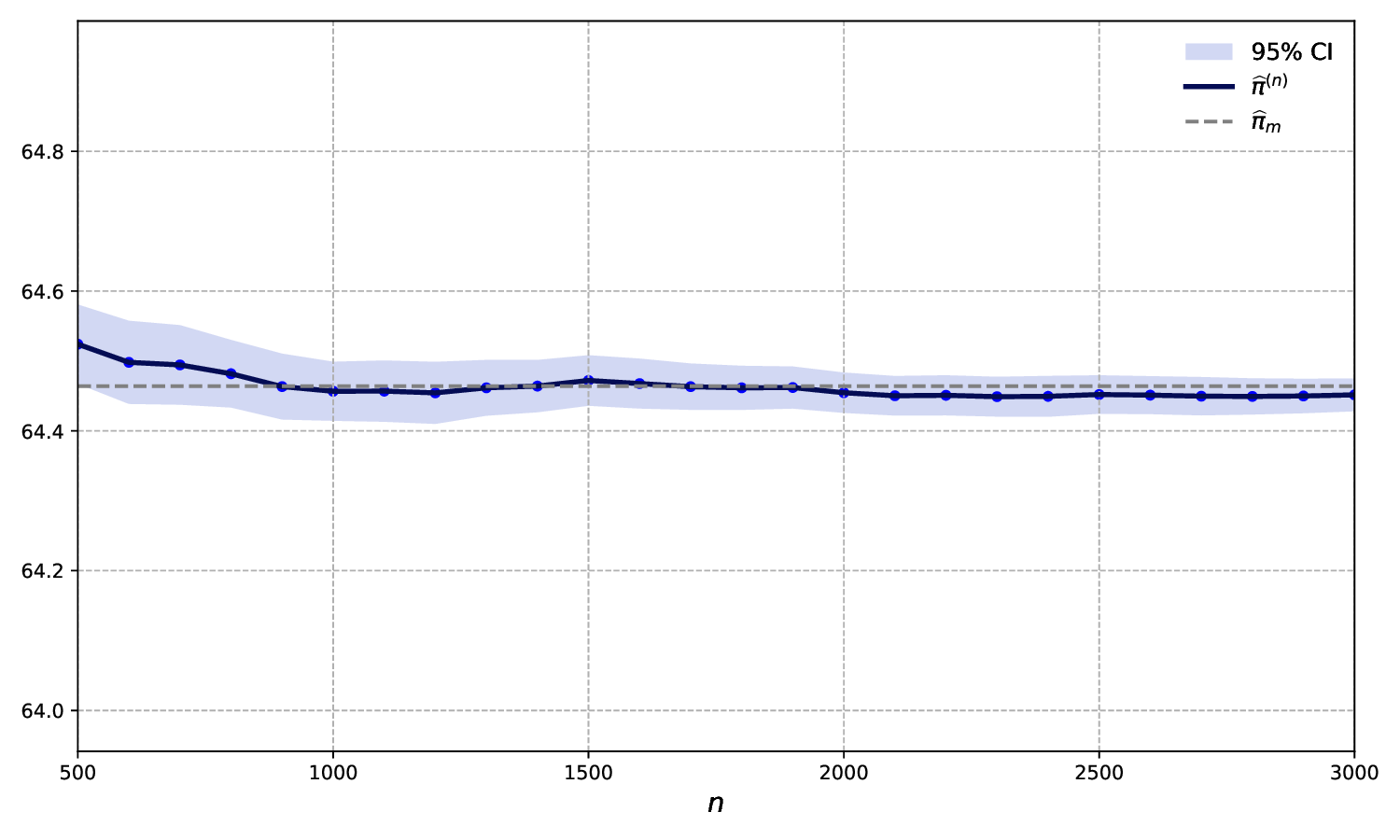}
		\caption{\small{The convergence of \(\widehat{\pi}^{(n)}\) to \(\widehat{\pi}_{\scriptscriptstyle m}\) as $n \to \infty$}}
		\label{fig:N}
	\end{figure}
Figure \ref{fig:N} illustrates the convergence of the time-consistent equilibrium strategy $\widehat{\pi}^{(n)}$ toward the MFE $\widehat{\pi}_m$ as $n$ increases, with all parameters fixed at the baseline values specified in Subsection \ref{sub:NS}. The solid blue curve represents $\widehat{\pi}^{(n)}$ computed for each $n$, while the horizontal dashed line indicates the mean-field limit $\widehat{\pi}_m$. This limit is obtained by taking $n \to \infty$ in the equilibrium system and solving the corresponding mean-field game as derived in Theorem \ref{th:homo_mefield}. As $n$ grows, $\widehat{\pi}^{(n)}$ converges to $\widehat{\pi}_m$, a result that aligns with Proposition \ref{th:hre}. This asymptotic behavior confirms that the mean-field equilibrium provides an accurate approximation to the finite-player game for sufficiently large populations.
	
\end{bluepar}

	\bibliographystyle{apalike}
	\bibliography{refs}
	
	\appendix
	
\section{Proof of Theorem \ref{th:verification}} \label{app:A}

	\setcounter{equation}{0}
	\renewcommand{\theequation}{A.\arabic{equation}}
	In this appendix, we will provide the proof of Theorem \ref{th:verification}.
	\begin{proof}
		Let \(\widehat{\pi}_i\) be the function that attains the supremum in \eqref{eq:EHJB}, and $\widehat{\bpi}_i$ be the corresponding feedback strategy. Suppose that \(\bm{\widehat{\pi}}_i\),
		\(G_i\) and \(H_i\) satisfy \eqref{eq:EHJB},  \eqref{eq:LG}, \eqref{eq:LH}, and \eqref{eq:T}. Fixed a time \(t\), we define a sequence of stopping times \(\{\tau_n\}_{n \ge 1}\) as follows
		\begin{align}
			\tau_n=\left(\inf\left\{s\geq 0:\int_0^s \big(\widehat{\pi}_i(v,X_i^{\widehat{\bm{\pi}}_i}(v),Y^{\widehat{\bpi}_{\smn i}}_{\smn i}(v))\big)^2 \drm v \geq n\right\}\right) \vee t.
		\end{align}
{\bf{Step 1.}} First, we prove that
\[	
G_i(t,x_i,y_{\smn i})=\mathbb{E}_{t,x_i,y_{\smn i},t+}\left[(1-\phi_i/n)X_i^{\widehat{\bm{\pi}}_i}(\tau)-\phi_i Y^{\widehat{\bpi}_{\smn i}}_{\smn i}(\tau)\right].
 \]
 Let \(k>t\) be a fixed number. By applying It\^{o}'s formula to \(\e^{-\int_t^{\tau_n\wedge k}\lambda(s)\drm s} G_i(\tau_n\wedge k,X_i^{\widehat{\bm{\pi}}_i}(\tau_n\wedge k),Y^{\widehat{\bpi}_{\smn i}}_{\smn i}(\tau_n\wedge k))\), we have
		\begin{align*}
			&\e^{-\int_t^{\tau_n\wedge k}\lambda(s)\drm s} G_i(\tau_n\wedge k,X_i^{\widehat{\bm{\pi}}_i}(\tau_n\wedge k),Y^{\widehat{\bpi}_{\smn i}}_{\smn i}(\tau_n\wedge k))\\
			&=G_i(t,x_i,y_{\smn i})+\int_t^{\tau_n\wedge k} \e^{-\int_t^s \lambda(v)\drm v}\Big(\mathcal{L}_i^{\widehat{\pi}}G_i(s,X_i^{\widehat{\bm{\pi}}_i}(s),Y^{\widehat{\bpi}_{\smn i}}_{\smn i}(s))-\lambda(s)G_i(s,X_i^{\widehat{\bm{\pi}}_i}(s),Y^{\widehat{\bpi}_{\smn i}}_{\smn i}(s))\Big)\drm s\\
			&\quad +\int_t^{\tau_n\wedge k}(\cdots)\big( \drm W_i(s) + \drm B(s)\big).
		\end{align*}
By taking conditional expectation \(\mathbb{E}_{t, x_i, y_{\smn i}}[\cdot]\) on both sides of the above equation, we obtain
		\begin{align}
			&	\mathbb{E}_{t,x_i,y_{\smn i}}\Big[\e^{-\int_t^{\tau_n\wedge k}\lambda(s)\drm s} G_i\big(\tau_n\wedge k,X_i^{\widehat{\bm{\pi}}_i}(\tau_n \wedge k),Y^{\widehat{\bpi}_{\smn i}}_{\smn i}(\tau_n \wedge k)\big)\Big]\\
			&=G_i(t,x_i,y_{\smn i})+\mathbb{E}_{t,x_i,y_{\smn i}}\bigg[\int_t^{\tau_n\wedge k} \e^{-\int_t^s \lambda(v)\drm v}\left(\mathcal{L}_i^{\widehat{\pi}}G_i\big(s,X_i^{\widehat{\bm{\pi}}_i}(s),Y^{\widehat{\bpi}_{\smn i}}_{\smn i}(s)\big)-\lambda(s)G_i\big(s,X_i^{\widehat{\bm{\pi}}_i}(s),Y^{\widehat{\bpi}_{\smn i}}_{\smn i}(s)\big)\right)ds\bigg]\\
			&=G_i(t,x_i,y_{\smn i})-\mathbb{E}_{t,x_i,y_{\smn i}}\left[\int_t^{\tau_n\wedge k} \lambda(s)\e^{-\int_t^s \lambda(v)\drm v}\left((1-\phi_i/n)X^{\widehat{\bm{\pi}}_i}_i(s)-\phi_iY^{\widehat{\bpi}_{\smn i}}_{\smn i}(s)\right)\drm s\right],
		\end{align}
in which the second equality follows from \eqref{eq:LG}. \blue{Since $\widehat{\bm{\pi}}_i$ and $\widehat{\bm{\pi}}_j, j\neq i$ are all admissible strategies, we know that
\begin{align}
& \mathbb{E}_{t,x_i,y_{\smn i}}\left[\int_t^{\tau_n\wedge k} \lambda(s)\e^{-\int_t^s \lambda(v)\drm v}\left((1-\phi_i/n)X^{\widehat{\bm{\pi}}_i}_i(s)-\phi_iY^{\widehat{\bpi}_{\smn i}}_{\smn i}(s)\right)\drm s\right] \\
&\le \mathbb{E}_{t,x_i,y_{\smn i}}\left[(1-\phi_i/n)X^{\widehat{\bm{\pi}}_i}_i(\tau)-\phi_iY^{\widehat{\bpi}_{\smn i}}_{\smn i}(\tau)\right] < \infty.
\end{align}
Thus, by applying the dominated convergence theorem and letting \(n\to\infty\),} it follows that
\begin{align}
	&	\mathbb{E}_{t,x_i,y_{\smn i}}\Big[\e^{-\int_t^{k}\lambda(s)\drm s} G_i\big(\tau_n\wedge k,X_i^{\widehat{\bm{\pi}}_i}(\tau_n\wedge k),Y^{\widehat{\bpi}_{\smn i}}_{\smn i}(\tau_n\wedge k)\big)\Big]\\
	&=G_i(t,x_i,y_{\smn i})-\mathbb{E}_{t,x_i,y_{\smn i}}\left[\int_t^{ k} \lambda(s)\e^{-\int_t^s \lambda(v)\drm v}\big((1-\phi_i/n)X^{\widehat{\bm{\pi}}_i}_i(s)-\phi_iY^{\widehat{\bpi}_{\smn i}}_{\smn i}(s)\big)\drm s\right].
\end{align}
Furthermore,  by letting  \(k \to \infty\) and applying the transversality condition given in \eqref{eq:T}, we get
		\begin{align}
			G_i(t,x_i,y_{\smn i})=\mathbb{E}_{t,x_i,y_{\smn i},t+}\left[(1-\phi_i/n)X_i^{\widehat{\bm{\pi}}_i}(\tau)-\phi_i Y^{\widehat{\bpi}_{\smn i}}_{\smn i}(\tau)\right].
		\end{align}	
{\bf{Step 2.}} In this step, we prove that
	\[
	H_i(t,x_i,y_{\smn i})=\mathbb{E}_{t,x_i,y_{\smn i},t+}\left[\left((1-\phi_i/n)X_i^{\widehat{\bm{\pi}}_i}(\tau)
	-\phi_i Y^{\widehat{\bpi}_{\smn i}}_{\smn i}(\tau)\right)^2\right].
	\]
	Let \(k>t\) be a fixed number. By applying It\^{o}'s formula to the process \(\e^{-\int_t^{\tau_n\wedge k}\lambda(s)\drm s} H_i(\tau_n\wedge k,X_i^{\widehat{\bm{\pi}}_i}(\tau_n\wedge k),Y^{\widehat{\bpi}_{\smn i}}_{\smn i}(\tau_n\wedge k))\), we obtain
	\begin{align*}
		&\e^{-\int_t^{\tau_n\wedge k}\lambda(s)\drm s} H_i(\tau_n\wedge k,X_i^{\widehat{\bm{\pi}}_i}(\tau_n\wedge k),Y^{\widehat{\bpi}_{\smn i}}_{\smn i}(\tau_n\wedge k))\\
		&=H_i(t,x_i,y_{\smn i})+\int_t^{\tau_n\wedge k} \e^{-\int_t^s \lambda(v)\drm v}\Big(\mathcal{L}_i^{\widehat{\pi}}H_i(s,X_i^{\widehat{\bm{\pi}}_i}(s),Y^{\widehat{\bpi}_{\smn i}}_{\smn i}(s))-\lambda(s)H_i(s,X_i^{\widehat{\bm{\pi}}_i}(s),Y^{\widehat{\bpi}_{\smn i}}_{\smn i}(s))\Big)\drm s\\
		&\quad +\int_t^{\tau_n\wedge k}(\cdots)\big( \drm W_i(s) + \drm B(s)\big).
	\end{align*}
	Taking the conditional expectation \(\mathbb{E}_{t, x_i, y_{\smn i}}[\cdot]\) on both sides of the above equation yields
	\begin{align}
		&	\mathbb{E}_{t,x_i,y_{\smn i}}\Big[\e^{-\int_t^{\tau_n\wedge k}\lambda(s)\drm s} H_i\big(\tau_n\wedge k,X_i^{\widehat{\bm{\pi}}_i}(\tau_n \wedge k),Y^{\widehat{\bpi}_{\smn i}}_{\smn i}(\tau_n \wedge k)\big)\Big]\\
		&=H_i(t,x_i,y_{\smn i})+\mathbb{E}_{t,x_i,y_{\smn i}}\bigg[\int_t^{\tau_n\wedge k} \e^{-\int_t^s \lambda(v)\drm v}\left(\mathcal{L}_i^{\widehat{\pi}}H_i\big(s,\cdot\big)-\lambda(s)H_i\big(s,\cdot\big)\right)\drm s\bigg]\\
		&=H_i(t,x_i,y_{\smn i})-\mathbb{E}_{t,x_i,y_{\smn i}}\left[\int_t^{\tau_n\wedge k} \lambda(s)\e^{-\int_t^s \lambda(v)\drm v}\big((1-\phi_i/n)X^{\widehat{\bm{\pi}}_i}_i(s)-\phi_iY^{\widehat{\bpi}_{\smn i}}_{\smn i}(s)\big)^2\drm s\right],
	\end{align}
	where the second equality follows from \eqref{eq:LH}.
\blue{As to step 1, since $\widehat{\bm{\pi}}_i$ and $\widehat{\bm{\pi}}_j, j\neq i$ are admissible strategies, we use the dominated convergence theorem, letting \(n\to\infty\), and then, letting \(k \to \infty\) and applying the transversality condition given in \eqref{eq:T}, we obtain}
	\begin{align}
		H_i(t,x_i,y_{\smn i})=\mathbb{E}_{t,x_i,y_{\smn i},t+}\left[\left((1-\phi_i/n)X_i^{\widehat{\bm{\pi}}_i}(\tau)-\phi_i Y^{\widehat{\bpi}_{\smn i}}_{\smn i}(\tau)\right)^2\right].
	\end{align}

 \noindent {\bf{Step 3.}}  Next, we show that the value function \(V_i(t,x_i,y_{\smn i})\) equals \(J_i(t,x_i,y_{\smn i};\widehat{\bm{\pi}}_1,\ldots,\widehat{\bm{\pi}}_i,\ldots,\widehat{\bm{\pi}}_n)\). From \eqref{eq:LG} and \eqref{eq:LH}, we rewrite \eqref{eq:EHJB} as
		\begin{align}\label{eq:EHJB1}
			&	\lambda(t)\left\{V_i(t,x_i,y_{\smn i})-(\mu_{i1}x_i+\mu_{i2})\left((1-\phi_i/n)x_i-\phi_i y_{\smn i}\right)+\gamma_i \left[G_i(t,x_i,y_{\smn i})-\left((1-\phi_i/n)x_i-\phi_iy_{\smn i}\right)\right]^2\right\}\nonumber\\
			&=\mathcal{L}_i^{\widehat{\pi}}V_i(t,x_i,y_{\smn i})-\mathcal{L}_i^{\widehat{\pi}}f(t,x_i,y_{\smn i},G_i,H_i)+f_G(t,x_i,y_{\smn i},G_i,H_i)\mathcal{L}_i^{\widehat{\pi}} G_i(t,x_i,y_{\smn i})\nonumber\\
			&\quad\quad+f_H(t,x_i,y_{\smn i},G_i,H_i)\mathcal{L}_i^{\widehat{\pi}}H_i(t,x_i,y_{\smn i}).
		\end{align}

		By the definition of \(f\) given in \eqref{eq:f}, we obtain
		\[
		f_G(t, x_i, y_{\smn i}, G_i, H_i) = \mu_{i1}x+\mu_{i2} + 2\gamma_i G_i(t, x_i, y_{\smn i})
		\]
		and
		\[
		f_H(t, x_i, y_{\smn i}, G_i, H_i) = -\gamma_i.
		\]
		By substituting these two partial derivatives into \eqref{eq:EHJB1} and using \eqref{eq:LG} and \eqref{eq:LH}, we obtain
		\begin{align}
			&	\lambda(t)\Big\{V_i(t,x_i,y_{\smn i})+\gamma_i G^2_i(t,x_i,y_{\smn i})-2\gamma_i\left((1-\phi_i/n)x_i-\phi_i y_{\smn i}\right) G_i(t,x_i,y_{\smn i})\\
			&\quad\quad+\gamma_i\left((1-\phi_i/n)x_i-\phi_i y_{\smn i}\right)^2-(\mu_{i1}x_i+\mu_{i2})\big((1-\phi_i/n)x_i-\phi_i y_{\smn i}\big)\Big\}\\
			&=\mathcal{L}_i^{\widehat{\pi}}V_i(t,x_i,y_{\smn i})-\mathcal{L}_i^{\widehat{\pi}}f(t,x_i,y_{\smn i},G_i,H_i)+\left(\mu_{i1}x+\mu_{i2}+2\gamma_i G_i(t,x_i,y_{\smn i})\right)\mathcal{L}_i^{\widehat{\pi}} G_i(t,x_i,y_{\smn i})\\
			&\quad\quad-\gamma_i\mathcal{L}_i^{\widehat{\pi}}H_i(t,x_i,y_{\smn i})\\
			&=\mathcal{L}_i^{\widehat{\pi}}V_i(t,x_i,y_{\smn i})-\mathcal{L}_i^{\widehat{\pi}}f(t,x_i,y_{\smn i},G_i,H_i)+\left(\mu_{i1}x+\mu_{i2}+2\gamma_i G_i(t,x_i,y_{\smn i})\right)\lambda(t)\big(G_i(t,x_i,y_{\smn i})\\
			&\quad\quad-(1-\phi_i/n)x_i+\phi_iy_{\smn i}\big)-\gamma_i \lambda(t)\left[H_i(t,x_i,y_{\smn i})-\left((1-\phi_i/n)x_i-\phi_iy_{\smn i}\right)^2\right]. \label{eq:HJB_extend}
		\end{align}
		Rearranging terms of \eqref{eq:HJB_extend}, it follows that
		\begin{align}\label{eq:EHJB2}
			\lambda(t)V_i(t,x_i,y_{\smn i})=\mathcal{L}^{\widehat{\pi}}_iV_i(t,x_i,y_{\smn i})-\mathcal{L}^{\widehat{\pi}}_if\left(t,x_i,y_{\smn i},G_i,H_i\right)+\lambda(t)f.
		\end{align}
Moreover, by repeating the proof of step (i) for \(V_i(t,x_i,y_{\smn i})\) and \(f(t,x_i,y_{\smn i},G_i,H_i)\), we obtain
\begin{align}
			&V_i(t,x_i,y_{\smn i})=\mathbb{E}_{t,x_i,y_{\smn i}}\left[\int_t^\infty \e^{-\int_t^s \lambda(v)\drm v}\left(\lambda(s)V_i\big(s,X_i^{\widehat{\bm{\pi}}_i}(s),Y^{\widehat{\bpi}_{\smn i}}_{\smn i}(s)\big) -\mathcal{L}_i^{\widehat{\pi}}V_i\big(s,X^{\widehat{\bm{\pi}}_i}(s),Y^{\widehat{\bpi}_{\smn i}}_{\smn i}(s)\big)\right)\drm s\right],
\\
\label{eq:V-1}
\end{align}
		and
		\begin{align}
f&\big(t,x_i,y_{\smn i},G_i(t,x_i,y_{\smn i}),H_i(t,x_i,y_{\smn i})\big)\\
=&\mathbb{E}_{t,x_i,y_{\smn i}}\left[\int_t^\infty \e^{-\int_t^s \lambda(v)\drm v}\left(\lambda(s)f\left(s,X_i^{\widehat{\bm{\pi}}_i}(s),Y^{\widehat{\bpi}_{\smn i}}_{\smn i}(s),G_i\left(s,X_i^{\widehat{\bm{\pi}}_i}(s),Y^{\widehat{\bpi}_{\smn i}}_{\smn i}(s)\right),H_i\left(s,X_i^{\widehat{\bm{\pi}}_i}(s),Y^{\widehat{\bpi}_{\smn i}}_{\smn i}(s)\right)\right)\right.\right.\\
&\quad\quad\quad\quad\quad \left.\left.-\mathcal{L}_i^{\widehat{\pi}}f\left(s,X^{\widehat{\bm{\pi}}_i}(s),Y^{\widehat{\bpi}_{\smn i}}_{\smn i}(s),
G_i\left(s,X_i^{\widehat{\bm{\pi}}_i}(s),Y^{\widehat{\bpi}_{\smn i}}_{\smn i}(s)\right),H_i\left(s,X_i^{\widehat{\bm{\pi}}_i}(s),Y^{\widehat{\bpi}_{\smn i}}_{\smn i}(s)\right)\right)\right)\drm s\right]. \\
\label{eq:f-1}
\end{align}
By combining \eqref{eq:V-1}, \eqref{eq:f-1}, and \eqref{eq:EHJB2} together, we finally obtain
		\begin{align}
			&V_i(t, x_i, y_{\smn i})\\
			&=f\left(t, x_i, y_{\smn i}, G_i(t, x_i, y_{\smn i}), H_i(t, x_i, y_{\smn i})\right)\\
			&=\left(\mu_{i1}x_i+\mu_{i2}\right)G_i(t,x_i,y_{\smn i})-\gamma_i\left(H_i(t,x_i,y_{\smn i})-G^2_i(t,x_i,y_{\smn i})\right)\\
			&=\left(\mu_{i1}x_i+\mu_{i2}\right)\mathbb{E}_{t,x_i,y_{\smn i},t+}\left[\left(1-\phi_i/n\right)X_i^{\widehat{\bm{\pi}}_i}(\tau)-\phi_iY^{\widehat{\bpi}_{\smn i}}_{\smn i}(\tau)\right]\\
			&\quad-\gamma_i\left(\mathbb{E}_{t,x_i,y_{\smn i},t+}\left[\left(\left(1-\phi_i/n\right)X_i^{\widehat{\bm{\pi}}_i}(\tau)-\phi_iY^{\widehat{\bpi}_{\smn i}}_{\smn i}(\tau)\right)^2\right]-\mathbb{E}_{t,x_i,y_{\smn i},t+}^2\left[\left(1-\phi_i/n\right)X_i^{\widehat{\bm{\pi}}_i}(\tau)-\phi_iY^{\widehat{\bpi}_{\smn i}}_{\smn i}(\tau)\right]\right)\\
			&=J_i\big(t, x_i, y_{\smn i}; \widehat{\bm{\pi}}_1, \ldots, \widehat{\bm{\pi}}_i, \ldots, \widehat{\bm{\pi}}_n\big).
		\end{align}
{\bf{Step 4.}}  At last, we demonstrate that \(\widehat{\bm{\pi}}_i\) is indeed the equilibrium feedback strategy. For any \(\epsilon > 0\), define an admissible strategy \(\bpi_i^\epsilon\) as the one in Definition \ref{definitionE}, and we show that the limit in \eqref{eq:limit_equi} holds. For simplicity, define
\begin{align}
Z_i^\epsilon (s)&= (1 - \phi_i/n) X_i^{\bm{\pi}^\epsilon_i}(s) - \phi_i Y^{\widehat{\bpi}_{\smn i}}_{\smn i}(s),\\
\widehat{Z}_i (s) &= (1 - \phi_i/n) X_i^{\widehat{\bm{\pi}}_i}(s) - \phi_i Y^{\widehat{\bpi}_{\smn i}}_{\smn i}(s),\\
z_i &= (1 - \phi_i/n)x_i - \phi_i y_{\smn i}.
\end{align}
Recalling that
		\(
		\mathbb{P}(\tau > t + \epsilon \mid \tau > t) = \e^{-\int_t^{t+\epsilon} \lambda(s)\, ds},
		\)
we have
		\begin{align*}
			G_i&(t,x_i,y_{\smn i})\\
			=&\mathbb{E}_{t,x_i,y_{\smn i}}\left[ \int_t^{\infty} \lambda(s) \e^{-\int_t^s\lambda(v)\drm v} Z_i^\epsilon(s)\drm s\right]\\
			=&\mathbb{E}_{t,x_i,y_{\smn i}}\left[\int_t^{t+\epsilon} \lambda(s)\e^{-\int_t^s\lambda(v)\drm v} Z_i^\epsilon(s)\drm s\right]+\e^{-\int_t^{t+\epsilon}\lambda(v)\drm v}\mathbb{E}_{t,x_i,y_{\smn i}}\left[\int_{t+\epsilon}^{\infty} \lambda(s) \e^{-\int_t^s\lambda(v)\drm v} Z_i^\epsilon(s)\drm s\right]\\
			=&\epsilon \lambda(t)z_i+\big(1-\epsilon \lambda(t)\big)\mathbb{E}_{t,x_i,y_{\smn i}}\left[\mathbb{E}_{t+\epsilon,X_i^{{\bm{\pi}}^\epsilon_i}(t+\epsilon),Y^{\widehat{\bpi}_{\smn i}}_{\smn i}(t+\epsilon),(t+\epsilon)+} \left[\widehat{Z}_i(\tau)\right]\right]+o(\epsilon)\\
			=&\epsilon \lambda(t)z_i+\big(1-\epsilon \lambda(t)\big)\mathbb{E}_{t,x_i,y_{\smn i}}\left[G_i\left(t+\epsilon,X^{\pi_i}_i(t+\epsilon),Y^{\widehat{\bpi}_{\smn i}}_{\smn i}(t+\epsilon)\right)\right]+o(\epsilon),
		\end{align*}
		in which \(X^{\pi_i}_i(t+\epsilon)\) denotes the random value of \(X_i(t+\epsilon)\) if the constant strategy \(\pi_i\) is followed between times \(t\) to \(t+\epsilon\).
		
		By performing a similar discussion for \( H \) and \( G^2 \), we obtain
		\begin{align*}
			H_i(t,x_i,y_{\smn i})
			=&\mathbb{E}_{t,x_i,y_{\smn i},t+}\big[\big(Z_i^\epsilon(\tau)\big)^2\big]\\
			=&\epsilon \lambda(t)z_i^2+\big(1-\epsilon \lambda(t)\big)\mathbb{E}_{t,x_i,y_{\smn i}}\left[\mathbb{E}_{t+\epsilon,X_i^{{\bm{\pi}}^\epsilon_i}(t+\epsilon),Y^{\widehat{\bpi}_{\smn i}}_{\smn i}(t+\epsilon),(t+\epsilon)+} \left[\left(\widehat{Z}_i(\tau)\right)^2\right]\right]+o(\epsilon)\\
			=&\epsilon \lambda(t)z_i^2+\big(1-\epsilon \lambda(t)\big)\mathbb{E}_{t,x_i,y_{\smn i}}\left[H_i\left(t+\epsilon,X^{\pi_i}_i(t+\epsilon),Y^{\widehat{\bpi}_{\smn i}}_{\smn i}(t+\epsilon)\right)\right]+o(\epsilon),
		\end{align*}
		and
		\begin{align*}
			G_i^2(t,x_i,y_{\smn i})
			=&\left\{\mathbb{E}_{t,x_i,y_{\smn i},t+}\left[ Z_i^\epsilon(\tau)\right]\right\}^2\\
			=&\left\{\epsilon \lambda(t)z_i+\big(1-\epsilon\lambda(t)\big)\mathbb{E}_{t,x_i,y_{\smn i}}\left[\mathbb{E}_{t+\epsilon,X_i^{{\bm{\pi}}_i^\epsilon}(t+\epsilon),Y^{\widehat{\bpi}_{\smn i}}_{\smn i}(t+\epsilon),(t+\epsilon)+} \left[\widehat{Z}_i(\tau)\right]\right]\right\}^2+o(\epsilon)\\
			=&\left\{\epsilon \lambda(t)z_i+\big(1-\epsilon\lambda(t)\big)\mathbb{E}_{t,x_i,y_{\smn i}}\left[G_i\left(t+\epsilon,X^{\pi_i}_i(t+\epsilon),Y^{\widehat{\bpi}_{\smn i}}_{\smn i}(t+\epsilon)\right)\right]\right\}^2+o(\epsilon).
		\end{align*}
		
		Then, by substituting the above expressions related to \(G_i\), \(H_i\), and \(G^{2}_i\) into the expression of \(J_i\left(t,x_i,y_{\smn i};\widehat{\bm{\pi}}_1,\ldots,\bm{\pi}_i^\epsilon,\ldots,\widehat{\bm{\pi}}_n\right)\) and simplifying, we get
		\begin{align}\label{eq:JJ}
			J_i&\left(t,x_i,y_{\smn i};\widehat{\bm{\pi}}_1,\ldots,\bm{\pi}_i^\epsilon,\ldots,\widehat{\bm{\pi}}_n\right)\nonumber\\
			=&(\mu_{i1}x_i+\mu_{i2})G_i(t,x_i,y_{\smn i})-\gamma\left\{H_i(t,x_i,y_{\smn i})-G^2_i(t,x_i,y_{\smn i})\right\}\nonumber\\
			\nonumber	=&\epsilon \lambda(t)\left[(\mu_{i1}x_i+\mu_{i2})z_i-\gamma_iz_i^2\right]+\left(1-\epsilon\lambda(t)\right)\Big\{\mathbb{E}_{t,x_i,y_{\smn i}}\big[f\big(t+\epsilon,X_i^{\pi_i}(t+\epsilon),Y^{\widehat{\bpi}_{\smn i}}_{\smn i}(t+\epsilon),G_i(t+\epsilon,\\
			\nonumber	&X_i^{\pi_i}(t+\epsilon),Y^{\widehat{\bpi}_{\smn i}}_{\smn i}(t+\epsilon)),H_i(t+\epsilon,X_i^{\pi_i}(t+\epsilon),Y^{\widehat{\bpi}_{\smn i}}_{\smn i}(t+\epsilon))\big)\big]\Big\}+\left(1-\epsilon\lambda(t)\right)f\Big(t,x_i,y_{\smn i},\mathbb{E}_{t,x_i,y_{\smn i}}\big[G_i\big(t+\epsilon,\\
			\nonumber	&X_i^{\pi_i}(t+\epsilon),Y^{\widehat{\bpi}_{\smn i}}_{\smn i}(t+\epsilon)\big)\big],\mathbb{E}_{t,x_i,y_{\smn i}}[H_i(t+\epsilon,X_i^{\pi_i}(t+\epsilon),Y^{\widehat{\bpi}_{\smn i}}_{\smn i}(t+\epsilon))]\Big)-\epsilon\lambda(t)\mathbb{E}^2_{t,x_i,y_{\smn i}}\big[G_i(t+\epsilon,X_i^{\pi_i}(t+\epsilon),\\
			\nonumber	&Y^{\widehat{\bpi}_{\smn i}}_{\smn i}(t+\epsilon))\big]-\left(1-\epsilon\lambda(t)\right)\mathbb{E}_{t,x_i,y_{\smn i}}\big[f\big(t+\epsilon,X_i^{\pi_i}(t+\epsilon),Y^{\widehat{\bpi}_{\smn i}}_{\smn i}(t+\epsilon),G_i(t+\epsilon,X_i^{\pi_i}(t+\epsilon),Y^{\widehat{\bpi}_{\smn i}}_{\smn i}(t+\epsilon)),\\
			&H_i(t+\epsilon,X_i^{\pi_i}(t+\epsilon),Y^{\widehat{\bpi}_{\smn i}}_{\smn i}(t+\epsilon))\big)\big]+2\epsilon\gamma_iz_i\lambda(t)\mathbb{E}_{t,x_i,y_{\smn i}}\left[G_i(t+\epsilon,X_i^{\pi_i}(t+\epsilon),Y^{\widehat{\bpi}_{\smn i}}_{\smn i}(t+\epsilon))\right]+o(\epsilon).
		\end{align}
		From the identity
		\begin{align*}
			&V_i\left(t+\epsilon,X_i^{\pi_i}(t+\epsilon),Y^{\widehat{\bpi}_{\smn i}}_{\smn i}(t+\epsilon)\right)\\
			&=f\left(t+\epsilon,X_i^{\pi_i}(t+\epsilon),Y^{\widehat{\bpi}_{\smn i}}_{\smn i}(t+\epsilon),G_i\left(t+\epsilon,X_i^{\pi_i}(t+\epsilon),Y^{\widehat{\bpi}_{\smn i}}_{\smn i}(t+\epsilon)\right),H_i\left(t+\epsilon,X_i^{\pi_i}(t+\epsilon),Y^{\widehat{\bpi}_{\smn i}}_{\smn i}(t+\epsilon)\right)\right),
		\end{align*}
		we further deduce
		\begin{align*}
			J_i&(t,x_i,y_{\smn i};\widehat{\bm{\pi}}_1,\ldots,\bm{\pi}_i^\epsilon,\cdot,\widehat{\bm{\pi}}_n)\\
			=&\epsilon \lambda(t)\left[(\mu_{i1}x_i+\mu_{i2})z_i-\gamma_iz_i^2\right]+\left(1-\epsilon\lambda(t)\right)\mathbb{E}_{t,x_i,y_{\smn i}}\left[V_i\left(t+\epsilon,X_i^{\pi_i}(t+\epsilon),Y^{\widehat{\bpi}_{\smn i}}_{\smn i}(t+\epsilon)\right)\right]+\left(1-\epsilon\lambda(t)\right)\\
			&\cdot f\left(t,x_i,y_{\smn i},\mathbb{E}_{t,x_i,y_{\smn i}}\left[G_i\left(t+\epsilon,X_i^{\pi_i}(t+\epsilon),Y^{\widehat{\bpi}_{\smn i}}_{\smn i}(t+\epsilon)\right)\right],\mathbb{E}_{t,x_i,y_{\smn i}}[H_i(t+\epsilon,X_i^{\pi_i}(t+\epsilon),Y^{\widehat{\bpi}_{\smn i}}_{\smn i}(t+\epsilon))]\right)\\
			&-\left(1-\epsilon\lambda(t)\right)f\left(t+\epsilon,X_i^{\pi_i}(t+\epsilon),Y^{\widehat{\bpi}_{\smn i}}_{\smn i}(t+\epsilon),G_i(t+\epsilon,X_i^{\pi_i}(t+\epsilon),Y^{\widehat{\bpi}_{\smn i}}_{\smn i}(t+\epsilon)),H_i(t+\epsilon,X_i^{\pi_i}(t+\epsilon),\right.\\
			&\left.Y^{\widehat{\bpi}_{\smn i}}_{\smn i}(t+\epsilon))\right)-\epsilon\lambda(t)\mathbb{E}^2_{t,x_i,y_{\smn i}}\left[G_i\left(t+\epsilon,X_i^{\pi_i}(t+\epsilon),Y^{\widehat{\bpi}_{\smn i}}_{\smn i}(t+\epsilon)\right)\right]+2\epsilon\gamma_iz_i\lambda(t)\mathbb{E}_{t,x_i,y_{\smn i}}\left[G_i(t+\epsilon,\right.\\
			&\left.X_i^{\pi_i}(t+\epsilon),Y^{\widehat{\bpi}_{\smn i}}_{\smn i}(t+\epsilon))\right]+o(\epsilon).
		\end{align*}
		
		For a sufficiently integrable function \(j\), we have
		\[
		\mathbb{E}_{t,x_i,y_{\smn i}}\left[j(t+\epsilon, X_i^{\pi_i}(t+\epsilon), Y^{\widehat{\bpi}_{\smn i}}_{\smn i}(t+\epsilon))\right] = j(t, x_i, y_{\smn i}) + \epsilon \mathcal{L}_i^{\widehat{\pi}(i)} j(t, x_i, y_{\smn i}) + o(\epsilon),
		\]
		in which \(\widehat{\pi}(i)=(\widehat{\pi}_1,\ldots,\pi_i,\ldots,\widehat{\pi}_n)\).
		
		By applying this with \(j = V_i\), \(G_i\), \(H_i\), and \(f\), we obtain
		\begin{align}
			J_i&\left(t,x_i,y_{\smn i};\widehat{\bm{\pi}}_1,\ldots,\bm{\pi}_i^\epsilon,\ldots,\widehat{\bm{\pi}}_n\right)\\
			=&\epsilon \lambda(t)\left[(\mu_{i1}x_i+\mu_{i2})z_i-\gamma_iz_i^2\right]+\left(1-\epsilon\lambda(t)\right)\left[V_i(t,x_i,y_{\smn i})+\epsilon \mathcal{L}_i^{\widehat{\pi}(i)}V_i(t,x_i,y_{\smn i})\right]\\
			&+\left(1-\epsilon \lambda(t)\right)\left[f\left(t,x_i,y_{\smn i},G_i(t,x_i,y_{\smn i}),H_i(t,x_i,y_{\smn i})\right)+\epsilon f_{G}\mathcal{L}_i^{\widehat{\pi}(i)}G_i(t,x_i,y_{\smn i})+\epsilon f_{H}\mathcal{L}_i^{\widehat{\pi}(i)}H_i(t,x_i,y_{\smn i})\right]\\
			&-\left(1-\epsilon \lambda(t)\right)\left[f\left(t,x_i,y_{\smn i},G_i(t,x_i,y_{\smn i}),H_i(t,x_i,y_{\smn i})\right)+\epsilon f_{G}\mathcal{L}_i^{\widehat{\pi}(i)}f\left(t,x_i,y_{\smn i},G_i(t,x_i,y_{\smn i}),H_i(t,x_i,y_{\smn i})\right)\right]\\
			&-\epsilon\gamma_i\lambda(t)\left[G_i(t,x_i,y_{\smn i})+\epsilon \mathcal{L}_i^{\widehat{\pi}(i)}G_i(t,x_i,y_{\smn i})\right]+2\epsilon\gamma_i z_i\lambda(t)\left[G_i(t,x_i,y_{\smn i})+\epsilon \mathcal{L}_i^{\widehat{\pi}(i)}G_i(t,x_i,y_{\smn i})\right]+o(\epsilon)\\
			=&V_i(t,x_i,y_{\smn i})+\epsilon\left\{\mathcal{L}_i^{\widehat{\pi}(i)}V_i(t,x_i,y_{\smn i})-\mathcal{L}_i^{\widehat{\pi}(i)}f\left(t,x_i,y_{\smn i},G_i(t,x_i,y_{\smn i}),H_i(t,x_i,y_{\smn i})\right)+f_{G}\mathcal{L}_i^{\widehat{\pi}(i)}G_i(t,x_i,y_{\smn i})\right.\\
			&\left.+f_{H}\mathcal{L}_i^{\widehat{\pi}(i)}H_i(t,x_i,y_{\smn i})-\lambda(t)\left[V_i(t,x_i,y_{\smn i})+\gamma_i\left(G_i(t,x_i,y_{\smn i})-z_i\right)^2-(\mu_{i1}x_i+\mu_{i2})z_i\right]\right\}+o(\epsilon).\\
\label{eq:V-final}
		\end{align}
Because \(\pi_i\) is an arbitrary real number, and the \(\epsilon\)-term in equation \eqref{eq:V-final} is non-positive, from which it follows that
\begin{align*}
J_i(t,x_i,y_{\smn i};\widehat{\bm{\pi}}_1,\ldots,\bm{\pi}_i^\epsilon,\ldots,\widehat{\bm{\pi}}_n)\leq V_i(t,x_i,y_{\smn i})+o(\epsilon).
\end{align*}
Hence,
		\begin{align*}
			{\lim\sup}_{\epsilon\to 0} \frac{J_i(t,x_i,y_{\smn i};\widehat{\bm{\pi}}_1,\ldots,\bm{\pi}_i^\epsilon,\ldots,\widehat{\bm{\pi}}_n)-J_i(t,x_i,y_{\smn i};\widehat{\bm{\pi}}_1,\ldots,\widehat{\bm{\pi}}_i,\ldots,\widehat{\bm{\pi}}_n)}{\epsilon}\leq 0,
		\end{align*}
and we complete the proof.
	\end{proof}
\begin{bluepar}
\section{Proof of Theorem \ref{th:main}}\label{app:B}
\setcounter{equation}{0}
	\renewcommand{\theequation}{B.\arabic{equation}}
\begin{proof}
From the above construction, we see that $V_i(t,x_i,y_{\smn i}), G_i(t,x_i,y_{\smn i}), H_i(t,x_i,y_{\smn i})$ satisfy Conditions (1)-(3) in Theorem \ref{th:verification}, we only need to verify that $\widehat{\bm{\pi}}_i$ is an admissible strategy and the transversality condition (4) holds.

 Given $\widehat{\bm{\pi}}_j \in {\bm{\Pi}_j}, j\neq i$, we first show that $\widehat{\bm{\pi}}_i$ is an admissible strategy, that is, $\widehat{\bm{\pi}}_i \in {\bm{\Pi}_i}$.
We see $\widehat{\bm{\pi}}_i$ given in \eqref{eq:pi_i1} is $\mathbb{F}$-progressively measurable, and \eqref{eq:X} has a unique strong solution under the strategy $\widehat{\bm{\pi}}_i$.
In the following, we show that there exist constants $\mathfrak{a}_i, \mathfrak{b}_i$, and $\mathfrak{d}_i$, such that $\E_{t, {x_i}} \big[\left({X}^{\widehat{\bpi}_i}_i(s)\right)^2\big]\le \mathfrak{a}_i \e^{\mathfrak{d}_i (s-t)} + \mathfrak{b}_i$, for any $s\ge t$.
Define $\tau_n = \inf\{s \ge t: |{X}^{\widehat{\bpi}_i}_i(s)| \ge n\}$, and $\lim_{n \to \infty} \tau_n = \infty, a.s.$ We set $s_n= s\wedge \tau_n$.
By applying It\^o's formula to $(X_i^{\widehat{\bm{\pi}}_i}(s_n))^2$, we have
	\begin{align*}
		(X_i^{\widehat{\bm{\pi}}_i}(s_n))^2
		&= x^2_i+ \int_t^{s_n} \left\{ 2 X^{\widehat{\bpi}_i}_i(v)[r X^{\widehat{\bpi}_i}_i(v) + (b_i - r) \widehat{\pi}_i(v)] + \xi_i^2 \widehat{\pi}_i^2(v) + \sigma_i^2 {\widehat{\pi}}_i^2(v) \right\} \drm v \\
		&\quad + 2\int_t^{s_n} X^{\widehat{\bpi}_i}_i(v)\xi_i \widehat{\pi}_i(v) \drm W_i(v) + 2\int_t^{s_n} X^{\widehat{\bpi}_i}_i(v)\sigma_i \widehat{\pi}_i(v) \drm B(v).
	\end{align*}
Then, taking expectations on both sides and noting that the last two stochastic integration terms are martingales, we obtain
 \begin{align}
\EE_{t, x_i}[( X^{\widehat{\bpi}_i}_i(s_n))^2 ] = x^2_i + \EE_{t, x_i} \left( \int_t^{s_n}  \left[2X_i^{\widehat{\bm{\pi}}_i}(v)[r X^{\widehat{\bpi}_i}_i(v) + (b_i - r) \widehat{\pi}_i(v)] + \xi_i^2 \widehat{\pi}_i^2(v) + \sigma_i^2 \widehat{\pi}_i^2(v) \right] \drm v \right).
 \end{align}
 Since the right-hand side is less than
 \[
 x^2_i +2r\EE_{t, x_i}\left[\int_t^{s}(  X^{\widehat{\bpi}_i}_i(v) )^2\drm v\right]+(b_i-r)  \EE_{t, x_i}\left[\int_t^{s} ( X^{\widehat{\bpi}_i}_i(v) )^2 + \widehat{\pi}_i^2(v) \drm v\right]+(\xi_i^2+\sigma_i^2)\EE_{t, x_i}\left[\int_t^{s} \widehat{\pi}_i^2(v)\drm v\right],
 \]
 thus, by the dominated convergence theorem and letting $n \to \infty$, it follows that
\begin{align*}
&\EE_{t, x_i}[( X^{\widehat{\bpi}_i}_i(s))^2 ] \\
&= x^2_i + \EE_{t, x_i} \left( \int_t^{s}  \left[2X_i^{\widehat{\bm{\pi}}_i}(v)[r X^{\widehat{\bpi}_i}_i(v) + (b_i - r) \widehat{\pi}_i(v)] + \xi_i^2 \widehat{\pi}_i^2(v) + \sigma_i^2 \widehat{\pi}_i^2(v) \right] \drm v \right)\\
&\le x^2_i +2r\EE_{t, x_i}\left[\int_t^{s}(  X^{\widehat{\bpi}_i}_i(v) )^2\drm v\right]+(b_i-r)  \EE_{t, x_i}\left[\int_t^{s} ( X^{\widehat{\bpi}_i}_i(v) )^2 + \widehat{\pi}_i^2(v) \drm v\right]+(\xi_i^2+\sigma_i^2)\EE_{t, x_i}\left[\int_t^{s} \widehat{\pi}_i^2(v)\drm v\right]\\
&\le x^2_i + (b_i+r)\int_t^{s} \EE_{t, x_i}[(X^{\widehat{\bpi}_i}_i(v))^2]\drm v+ \big((b_i - r) + \xi_i^2 + \sigma_i^2 \big)\E_{t, x_i}\left[\int_t^{s} \widehat{\pi}_i^2(v)\drm v \right]\\
&\le x^2_i +\mathfrak{P}_i \int_t^{s} \EE_{t, x_i}[( X^{\widehat{\bpi}_i}_i(v))^2]\drm v +\mathfrak{k}_i \left( \mathfrak{A}\e^{\mathfrak{Q}({s}-t)} + \mathfrak{B}({s}-t)+\mathfrak{C}\right),
\end{align*}
where the last inequality is derived from \eqref{eq:up_bound_X}, and $\mathfrak{P}_i$ is given in \eqref{eq:fkp}. Furthermore, by using the Gronwall's inequality, we deduce that
\begin{align}
&\EE_{t,x_i}[( X_i^{\widehat{\bm{\pi}}_i}({s}))^2]  \\
&\le  x^2_i + \mathfrak{k}_i \left( \mathfrak{A}\e^{\mathfrak{Q}({s}-t)} + \mathfrak{B}({s}-t)+\mathfrak{C}\right)+ \mathfrak{P}_i \int^{s}_t \e^{\mathfrak{P}_i({s} - v)} \left[x^2_i + \mathfrak{k}_i \left( \mathfrak{A}\e^{\mathfrak{Q}(v-t)} + \mathfrak{B}(v-t)+\mathfrak{C}\right)\right] \drm v \\
& \le \mathfrak{S}_i \e^{\max\{\mathfrak{Q},\mathfrak{P}_i\}({s}-t)} + \mathfrak{X}_i,
\end{align}
where $\mathfrak{S}_i, \mathfrak{X}_i$ are constants depending on the model parameters.
Thus, under Assumption \ref{assum:1a}, $\widehat{\bm{\pi}}_i$ satisfies the conditions (1)-(3) of Definition \ref{def:admissible}, and is therefore an admissible strategy.
For the transversality condition, we see $V_i, G, H$, and $f$ are all linear combinations of $x^2_i, y^2_{-i}, x_iy_{-i}, x_i, y_{-i}$, thus as to above derivation, under Assumption \ref{assum:1a}, one can show that the transversality condition holds.
\end{proof}
\end{bluepar}
\section{Proof of Corollary \ref{coro:1.8}}\label{app:C}
\setcounter{equation}{0}
	\renewcommand{\theequation}{C.\arabic{equation}}
\begin{proof}
First, from \eqref{eq:bara}, \eqref{eq:a1}, \eqref{eq:c1}, \eqref{eq:e1}, \eqref{eq:p1},  and \eqref{eq:q2}, through straightforward calculations, we derive that, as \(\lambda\to\infty\), for $i=1, 2$,
\begin{align*}
 \tilde{a}_i\to \left(1-\frac{\phi_i}{2}\right)^2, \quad a_i \to 1-\frac{\phi_i}{2}, \quad c_i\to-\phi_i, \quad \alpha_i\to 0,
\quad p_i\to\frac{\mu_{i1}(b_i-r)}{\gamma_i\left(1-\frac{\phi_i}{2}\right)}, \quad q_i\to 0,
\end{align*}
also from \eqref{eq:A}, \eqref{eq:C1}, \eqref{eq:D1}, \eqref{eq:E1}, \eqref{eq:F1}, and \eqref{eq:I1}, we deduce that, as \(\lambda\to\infty\),
\begin{align*}
A_i\to \mu_{i1}\left(1-\frac{\phi_i}{2}\right), \quad C_i\to 0, \quad D_i\to-\mu_{i1}\phi_i, \quad E_i\to \mu_{i2}\left(1-\frac{\phi_i}{2}\right),
\quad F_i\to -\mu_{i2}\phi_i, \quad I_i \to 0.
\end{align*}
Furthermore, using \eqref{eq:Qi}, \eqref{eq:k1},  \eqref{eq:k2}, and \eqref{eq:k3}, we derive that, as $\la \to \infty$,
\begin{align}\label{eq:limit}
k_{i1}\to\frac{2\phi_i\rho_i\sigma_i}{1-\frac{\phi_i}{2}},\quad k_{i2}\to0,\quad k_{i3}\to \frac{\mu_{i2}\rho_i(b_i-r)}{\gamma_i\left(1-\frac{\phi_i}{2}\right)},
\end{align}
and from \eqref{eq:Ni} and the expressions of $\rho_i$ and $\psi_{i,2}$ given in \eqref{eq:rho} and \eqref{eq:Psi}, it follows that
        \begin{align}\label{eq:limN}
        \Xi_i\to\frac{2\phi_i\rho_i\sigma_i^2+2\left(1-\frac{\phi_i}{2}\right)}{2\left(1-\frac{\phi_i}{2}\right)}, \quad
        \rho_i \Xi^{-1}_i \to \dfrac{1-\frac{\phi_i}{2}}{2  \psi_{i,2}}.
        \end{align}

Next, by substituting \eqref{eq:limit} and \eqref{eq:limN} into \eqref{eq:hpi1} and \eqref{eq:hpi2}, we obtain the expression for
\(\widehat{\pi}_i\), as \(\lambda\to\infty\),
\begin{align}\label{eq:limpi2}
\widehat{\pi}_i(x_1, x_2) \to \frac{\mu_{i1}(b_i-r)}{2\gamma_i\psi_{i,2} }x_i+\frac{\phi_i\sigma_i}{\psi_{i,2}} \cdot
\lim_{\la \to \infty} \overline{\sigma \widehat{\pi}}^{(2)} + \frac{\mu_{i2}(b_i-r)}{2\gamma_i\psi_{i,2}}.
 \end{align}
To determine the limiting expression of \(\widehat{\pi}_i\), we need to derive the limit of \(\overline{\sigma\widehat{\pi}}^{(2)}\) as $\la \to \infty$. By analogy with the above computation, from \eqref{eq:coeff1}, \eqref{eq:coeff2}, and \eqref{eq:coeff3}, we deduce that
 \begin{align}\label{eq:sigpi}
 \overline{\sigma \rho px}^{(2)} \to \Upsilon_2,  \quad \overline{\sigma\rho qy}^{(2)} \to 0,
 \quad \overline{\sigma k_1}^{(2)} \to \Psi_2,   \quad \overline{\sigma k_2}^{(2)} \to 0,
  \quad \overline{\sigma k_3}^{(2)} \to \Phi_2,
 \end{align}
in which \(\Psi_2\), \(\Phi_2\), and \(\Upsilon_2\) are given in \eqref{eq:Psi}, \eqref{eq:Phii0}, and \eqref{eq:Upsilon}.
Then, by substituting \eqref{eq:sigpi} into \eqref{eq:system1}, we derive
	\begin{align}\label{eq:rela2}
		\left(1-\Psi_2\right) \lim_{\la \to \infty} \overline{\sigma \widehat{\pi}}^{(2)}=\Upsilon_2+\Phi_2.
	\end{align}
Thus, the value of $\lim_{\la \to \infty} \overline{\sigma \widehat{\pi}}^{(2)}$ depends on whether \(\Psi_2 < 1\) or \(\Psi_2=1\), and can be categorized into the following three cases:
\begin{itemize}
\item[$(1)$] If \(\Psi_2 < 1\), then from \eqref{eq:rela2}, we derive that
 \[\overline{\sigma\widehat{\pi}}^{(2)} \to \frac{\Upsilon_2+\Phi_2}{1 -\Psi_2}.\]
 Thus, the limiting equilibrium feedback strategy given in \eqref{eq:limpi2} is well-defined. By further substituting \(\Psi_2\), \(\Phi_2\), and \(\Upsilon_2\) into \eqref{eq:limpi2}, and rearranging terms, we obtain the expressions of $\widehat{\pi}_i, i=1,2$ given in \eqref{eq:pi1.8} and \eqref{eq:pi2.8}. 	
\item[$(2)$]  If \(\Psi_2=1\) and \(\Upsilon_2+\Phi_2\neq0\), then \eqref{eq:rela2} has no solution. Thus, the limiting equilibrium feedback strategy does not exist.
\item[$(3)$] The remaining case is \(\Psi_2 = 1\) and \(\Upsilon_2+\Phi_2=0\). In this case, \eqref{eq:rela2} has infinitely many solutions. Then, the limiting equilibrium feedback strategy is given in \eqref{eq:picor1.80}, where \(\overline{\sigma\widehat{\pi}}^{(2)}\) is an arbitrary real number.
 \end{itemize}
\end{proof}

\section{Proof of Theorem \ref{th:homo_mefield}}\label{app:homo_mf}
\setcounter{equation}{0}
\renewcommand{\theequation}{D.\arabic{equation}}
	\begin{proof}
We postulate that for any fixed realization \(\omega \in \Omega\), the average processes \(\overline{\sigma\widehat{\pi}}\) and \(\overline{\underline{br}\widehat{\pi}}\) take the following affine form:
	\begin{align}\label{eq:average_pi}
		\overline{\sigma\hat{\pi}}  = \mathfrak{H}_1\bar{x}  + \mathfrak{H}_0,
		\qquad
		\overline{\underline{br}\hat{\pi}} = \mathfrak{T}_1\bar{x}  + \mathfrak{T}_0,
	\end{align}
	where  $\mathfrak{H}_i$ and \(\mathfrak{T}_i\) (for \(i=0,1\)) are constants and need to be determined.
Then, by substituting \eqref{eq:average_pi} into \eqref{eq:M-pi_i}, we can rewrite the expression of \(\widehat{\pi}(x,\bar{x})\) as follows
\begin{align}
\widehat{\pi}(x, \bar{x})&= \rho_{\scriptscriptstyle m} p_{\scriptscriptstyle m} x + \rho_{\scriptscriptstyle m} q_{\scriptscriptstyle m} \bar{x} + k_{1,\scriptscriptstyle m} \left(\mathfrak{H}_1\,\bar{x} +\mathfrak{H}_0\right)+ k_{2,\scriptscriptstyle m}\left(\mathfrak{T}_1\,\bar{x} + \mathfrak{T}_0\right) +k_{3,\scriptscriptstyle m}\\
		&=\rho_{\scriptscriptstyle m}p_{\scriptscriptstyle m} x+\left(\rho_{\scriptscriptstyle m} q_{\scriptscriptstyle m} +k_{1,\scriptscriptstyle m}\mathfrak{H}_1+k_{2,\scriptscriptstyle m}\mathfrak{T}_1\right)\bar{x}+k_{1,\scriptscriptstyle m}\mathfrak{H}_0+k_{2,\scriptscriptstyle m}\mathfrak{T}_0+k_{3,\scriptscriptstyle m}. \label{eq:M-pi}
	\end{align}
By multiplying $\sigma$ or $(b-r)$ on both sides of \eqref{eq:M-pi}, and taking the conditional expectation, we obtain
\begin{align*}
\overline{\sigma\widehat{\pi}}&=\sigma\left(\rho_{\scriptscriptstyle m}p_{\scriptscriptstyle m}+\rho_{\scriptscriptstyle m} q_{\scriptscriptstyle m} +k_{1,\scriptscriptstyle m}\mathfrak{H}_1+k_{2,\scriptscriptstyle m}\mathfrak{T}_1\right)\bar{x}+\sigma \left(k_{1,\scriptscriptstyle m}\mathfrak{H}_0+k_{2,\scriptscriptstyle m}\mathfrak{T}_0+k_{3,\scriptscriptstyle m}\right),\\
\overline{\underline{br}\widehat{\pi}}&=(b-r)\left(\rho_{\scriptscriptstyle m}p_{\scriptscriptstyle m}+\rho_{\scriptscriptstyle m} q_{\scriptscriptstyle m} +k_{1,\scriptscriptstyle m}\mathfrak{H}_1+k_{2,\scriptscriptstyle m}\mathfrak{T}_1\right)\bar{x}+(b-r) \left(k_{1,\scriptscriptstyle m}\mathfrak{H}_0+k_{2,\scriptscriptstyle m}\mathfrak{T}_0+k_{3,\scriptscriptstyle m}\right).
\end{align*}
By comparing them with \eqref{eq:average_pi}, we deduce the following system of equations
\begin{align}
\begin{cases}
\mathfrak{H}_1&=\sigma\left(\rho_{\scriptscriptstyle m}p_{\scriptscriptstyle m}+\rho_{\scriptscriptstyle m} q_{\scriptscriptstyle m} +k_{1,\scriptscriptstyle m}\mathfrak{H}_1+k_{2,\scriptscriptstyle m}\mathfrak{T}_1\right),\\
\mathfrak{H}_0&=\sigma \left(k_{1,\scriptscriptstyle m}\mathfrak{H}_0+k_{2,\scriptscriptstyle m}\mathfrak{T}_0+k_{3,\scriptscriptstyle m}\right),\\
\mathfrak{T}_1&=(b-r)\left(\rho_{\scriptscriptstyle m}p_{\scriptscriptstyle m}+\rho_{\scriptscriptstyle m} q_{\scriptscriptstyle m} +k_{1,m}\mathfrak{H}_1+k_{2,m}\mathfrak{T}_1\right),\\
\mathfrak{T}_0&=(b-r) \left(k_{1,\scriptscriptstyle m}\mathfrak{H}_0+k_{2,\scriptscriptstyle m}\mathfrak{T}_0+k_{3,m}\right).
\end{cases}
\end{align}
Given that $1-\delta_{\scriptscriptstyle m}\neq 0$, we solve the above system of equations and obtain
	\begin{equation}
		\begin{aligned}
			\mathfrak{H}_1&=\dfrac{\rho_{\scriptscriptstyle m} \sigma (p_{\scriptscriptstyle m}+q_{\scriptscriptstyle m})}{1 - \delta_m},\qquad
			\mathfrak{T}_1=\dfrac{\rho_{\scriptscriptstyle m}(b-r)(p_{\scriptscriptstyle m}+q_{\scriptscriptstyle m})}{1-\delta_m (b-r)},\\
			\mathfrak{H}_0&= \dfrac{\sigma k_{3,\scriptscriptstyle m}}{1-\delta_m}, \qquad \qquad \qquad
			\mathfrak{T}_0=\dfrac{(b-r) k_{3,\scriptscriptstyle m}}{1-\delta_m}.
		\end{aligned}
	\end{equation}
	By substituting these coefficients into \eqref{eq:M-pi} and rearranging terms, we derive the explicit expressions of $\widehat{\pi}$ given in \eqref{eq:M-Hpifinal}.
To show the self-consistency of the obtained ansatz. Define
	\begin{align*}
		\mathfrak{G}_{\scriptscriptstyle m}&:=\frac{\rho_{\scriptscriptstyle m}q_{\scriptscriptstyle m}+\delta_{\scriptscriptstyle m}\rho_{\scriptscriptstyle m} p_{\scriptscriptstyle m}}{1-\delta_{\scriptscriptstyle m}},\\
		\mathfrak{R}_{\scriptscriptstyle m}&:=\frac{k_{3,m}}{1-\delta_{\scriptscriptstyle m}}.
	\end{align*}
	Multiplying \eqref{eq:M-Hpifinal} by $\sigma$ or $(b-r)$ and taking the conditional expectation on both sides of the equation, it yields that
	\begin{align*}
		&\sigma \rho_{\scriptscriptstyle m} p_{\scriptscriptstyle m}+  \sigma\mathfrak{G}_{\scriptscriptstyle m}=\mathfrak{H}_1,\qquad \qquad \qquad\sigma \mathfrak{R}_{\scriptscriptstyle m}=\mathfrak{H}_0,\\
&(b-r) \rho_{\scriptscriptstyle m} p_{\scriptscriptstyle m}+  (b-r)\mathfrak{G}_{\scriptscriptstyle m}=\mathfrak{T}_1,\quad (b-r) \mathfrak{R}_{\scriptscriptstyle m}=\mathfrak{T}_0.
	\end{align*}
Hence, we complete our proof.
\end{proof}

 \begin{bluepar}
\section{Proof of Proposition \ref{th:hre}}\label{app:hre}
\setcounter{equation}{0}
\renewcommand{\theequation}{D.\arabic{equation}}

	\begin{proof}
	Since the game is homogeneous, all agents share the same type distribution. Therefore, we can treat any arbitrary agent $i$ (for $i=1,\dots,n$) as a representative agent, which allows us to set $\eta_i=\eta$ with $\eta$ denoting the generic type vector of the representative agent.  Recall the coefficients \(\rho\), \(p\), \(q\), \(Q\), \(k_{j}\) (for \(j=1,2,3\)) given in \eqref{eq:rho}, \eqref{eq:p1}, \eqref{eq:q2}, \eqref{eq:Qi}, \eqref{eq:k1}, \eqref{eq:k2}, \eqref{eq:k3}, and the corresponding MFG coefficients \(\rho_{\scriptscriptstyle m}\), \(p_{\scriptscriptstyle m}\), \(q_{\scriptscriptstyle m}\), \(Q_{\scriptscriptstyle m}\), \(k_{j,\scriptscriptstyle m}\) given in \eqref{eq:MFrho},\eqref{eq:MF-p1},\eqref{eq:MFq2}, \eqref{eq:MFQ}, \eqref{eq:MFk1}, \eqref{eq:MFk2}, \eqref{eq:MFk3}, we obtain \(\rho = \rho_{\scriptscriptstyle m} \),  and as $n\to\infty$, from direct elementary algebraic calculations, we get
	\begin{align*}
	p&\to p_{\scriptscriptstyle m}  , \qquad
	 q\to q_{\scriptscriptstyle m} ,\\
	 Q&\to Q_{\scriptscriptstyle m},  \qquad
	k_j \to k_{j,m},
	\end{align*}
	where \(j=1,2,3\).

For any arbitrary $t \ge 0$, we write $x_i := X_i(t)$ to denote the $i$-th agent's wealth at time $t$, which is different from the initial wealth $x_{i,0}$ recorded in $\eta$. Since all agents share the same parameters and are driven by mutually independent idiosyncratic Brownian motions $\{W_i\}_{i=1}^n$, thus, $\{x_i\}_{i=1}^n$ are conditionally i.i.d.\ given $\mathcal{F}_t^B$, which is the filtration generated by the common noise $\{B\}$. Hence, by the conditional Law of Large Numbers, $y_{ \smn i} = \frac{1}{n}\sum_{j=1,j \neq i}^n x_j \to \bar{x}$ almost surely as $n\to\infty$, where $\bar{x}= \E\left[x\mid\mathcal{F}_t^B\right]$ is the mean-field state of the representative agent.
	
Moreover, since $\delta\to\delta_{\scriptscriptstyle m}$ and all the coefficients are bounded, as $n \to \infty$, it follows that
	\begin{align*}
		&	\mathfrak{G} = \rho p + \frac{(n-1)\delta\rho (q + \delta p)}{(\delta+(1-\delta)n)(\delta+n)}\rightarrow \rho_{\scriptscriptstyle m} p_{\scriptscriptstyle m}, \\
		& \mathfrak{D}=\frac{n^2 \rho (q + \delta p)}{(\delta+(1-\delta)n)(\delta+n)}\rightarrow \frac{\rho_{\scriptscriptstyle m} q_{\scriptscriptstyle m}+\rho_{\scriptscriptstyle m} \delta_{\scriptscriptstyle m} p_{\scriptscriptstyle m}}{1-\delta_{\scriptscriptstyle m}}, \\
		&\mathfrak{R} = \frac{nk_3}{\delta + (1 - \delta)n}\rightarrow \frac{k_{3,m}}{1-\delta_{\scriptscriptstyle m}}.
	\end{align*}
Since in the homogeneous $n$-agent game, the equilibrium investment strategy of an agent with type $\eta_i$ is given by:
	\begin{align}
		\widehat{\pi}_i(x_i, y_{-i}) & = \mathfrak{G} x_i+\mathfrak{D} y_{\smn i}+\mathfrak{R},
	\end{align}
thus, letting $n\to\infty$, we derive that
	\begin{align*}
\widehat{\pi}_i(x_i, y_{-i}) \to \rho_{\scriptscriptstyle m}p_{\scriptscriptstyle m} x+\frac{\rho_{\scriptscriptstyle m} q_{\scriptscriptstyle m}+\rho_{\scriptscriptstyle m} \delta_{\scriptscriptstyle m} p_{\scriptscriptstyle m}}{1-\delta_{\scriptscriptstyle m}}\bar{x}+\frac{k_{3,m}}{1-\delta_{\scriptscriptstyle m}},\quad \text{a.s.}
	\end{align*}
in which the right side is exactly the mean-field equilibrium strategy $\widehat{\pi}_{\scriptscriptstyle m}$ given in \eqref{eq:M-Hpifinal}.
\end{proof}
\end{bluepar}
\begin{bluepar}
\section{The table of the notations}\label{App:G}
\setcounter{equation}{0}
\renewcommand{\theequation}{G.\arabic{equation}}
In this section, we will provide a list of notations we used in our paper.
\begin{table}[htbp]
	\centering
	\caption{Notation and Definitions}
	\begin{tabular}{l l}
		\toprule
		\textbf{Notation} & \textbf{Definition} \\
		\midrule
		$r$ & Constant risk-free interest rate \\
		$b_i$ & Expected return of stock $i$ \\
		$\xi_i$ & Idiosyncratic volatility of stock $i$ \\
		$\sigma_i$ & Common volatility of stock $i$ \\
		$\pi_i$ & Amount invested in stock $i$ \\
		$X_i^{\bm{\pi}_i}$ & Wealth process of agent $i$ \\
		$\bm{\pi}$ & Joint strategy of $n$ agents \\
		$\bm{X}^{\bm{\pi}}$ & Wealth processes of $n$ agents \\
		$\phi_i$ & Competition parameter of agent $i$ \\
		$\gamma_i$ & Risk aversion coefficient of agent $i$ \\
		$\mu_{i1}, \mu_{i2}$ & Positive constants in the affine weight of agent $i$ \\
		$x_i$ & Initial wealth of agent $i$ \\
		$y_{\smn i}$ & Average wealth excluding agent $i$ \\
		$Y_{\smn i}^{\bm{\pi}_{\smn i}}$ & Average wealth process excluding agent $i$ \\
		$\widehat{\underline{br}\pi}$ & $\frac{1}{n}\sum_{j\neq i}(b_j-r)\pi_j$ \\
		$\widehat{\sigma\pi}$ & $\frac{1}{n}\sum_{j\neq i}\sigma_j\pi_j$ \\
		$\widehat{\xi\pi}^2$ & $\sum_{j\neq i}\left(\frac{1}{n}\xi_j\pi_j\right)^2$ \\
		$\widehat{\sigma\pi}^2$ & $\sum_{j\neq i}\left(\frac{1}{n}\sigma_j\pi_j\right)^2$ \\
		$\rho_i$ & $1/[2(\xi_i^2+\sigma_i^2)]$ \\
		$\varphi_n$ & $1-\phi_i/n$ \\
		$\iota_i$ & $\frac{1}{2}\frac{(b_i-r)^2}{\xi_i^2+\sigma_i^2}$ \\
		$\psi_{i,n}$ & $(1-\phi_i/n)\xi_i^2+\sigma_i^2$ \\
		$p_i$ & Coefficient of $x_i$ in $\widehat{\pi}_i$, defined in \eqref{eq:p1} \\
		$q_i$ & Coefficient of $y_{\smn i}$ in $\widehat{\pi}_i$, defined in \eqref{eq:q2} \\
		$Q_i$ & Defined in \eqref{eq:Qi} \\
		$k_{i1}$ & Coefficient of $\widehat{\sigma\widehat{\pi}}$ in $\widehat{\pi}_i$, defined in \eqref{eq:k1} \\
		$k_{i2}$ & Coefficient of $\widehat{\underline{br}\widehat{\pi}}$ in $\widehat{\pi}_i$, defined in \eqref{eq:k2} \\
		$k_{i3}$ & Constant term in $\widehat{\pi}_i$, defined in \eqref{eq:k3} \\
		$\Xi_i$ & $1+\frac{k_{i1}\sigma_i+k_{i2}(b_i-r)}{2}$. \\
		$\delta$ & $k_1\sigma+k_2(b-r)$, defined in the homogeneous $n$-agent game \\
		$\mathfrak{G}$ & Coefficient of $x_i$ in $\widehat{\pi}$, defined in \eqref{eq:fG} \\
		$\mathfrak{D}$ & Coefficient of $y_{\smn i}$ in $\widehat{\pi}$, defined in \eqref{eq:fD} \\
		$\mathfrak{R}$ & Constant term in $\widehat{\pi}$, defined in \eqref{eq:fR} \\
		$\mathfrak{k}_i$ & $(b_i-r)+\xi_i^2+\sigma_i^2$ \\
		$\mathfrak{P}_i$ & $b_i+r+2\rho_i^2 p_i^2\mathfrak{k}_i$ \\
		$\Theta_n$ & $\sum_{i=1}^n\frac{\sigma_i b_i}{n\gamma_i\psi_{i,n}}$ \\
		$\Psi_n$ & $\sum_{i=1}^n\frac{\phi_i\sigma_i^2}{n\psi_{i,n}}$ \\
		$\Phi_n$ & $\sum_{i=1}^n\frac{\mu_{i2}\sigma_i(b_i-r)}{2n\gamma_i\psi_{i,n}}$ \\
		$\overline{\sigma o}^{(2)}$ & $\frac{1}{2}\sum_{j=1}^{2}\sigma_j o_j$, for $o\in\{\widehat{\pi},\rho p x, \rho q y, k_{j1}, k_{j2}, k_{j3}\}$ \\
		$\overline{\underline{br}\,o}^{(2)}$ & $\frac{1}{2}\sum_{j=1}^{2}(b_j-r)o_j$, for $o$ as above \\
		\bottomrule
	\end{tabular}
	\label{tab:notation}
\end{table}
\begin{table}[htbp]
	\centering
	\caption{Summary of notations used in Section 3}
	\label{tab:notation2}
	\renewcommand{\arraystretch}{1.3}
	\begin{tabular}{l l}
		\toprule
		\textbf{Notation} & \textbf{Definition} \\
		\midrule
		$r$ & The common interest rate \\
		$x_{i,0}$ & Initial wealth of agent $i$ \\
		$b_i$ & Expected return of stock \\
		$\xi_i$ & Idiosyncratic volatility \\
		$\sigma_i$ & Common volatility \\
		$\pi$ & Amount invested in stock \\
		$\eta_i$ & Type vector $(x_{i,0}, b_i, \dots, \mu_{2i})$ \\
		$m_n$ & The empirical measure \\
		$m$ & The limit measure of $m_n$ \\
		$\eta$ & Limiting type vector of $\eta_i$ \\
		$X^{\bm{\pi}}$ & Wealth process (rep. agent) \\
		$\phi$ & Competition parameter ($\in[0,1]$) \\
		$\gamma$ & Risk aversion ($>0$) \\
		$\bar{X}^{\bm{\pi}}$ & Average wealth process \\
		$x$ & Wealth at time $t$ \\
		$\bar{x}$ & Average wealth at time $t$, i.e.,\ $\bar{x}=\mathbb{E}[x\mid \mathcal{F}_t^B]$ \\
		$\overline{\underline{br}\pi}$ & $\mathbb{E}[(b-r)\pi(\cdot)\mid \mathcal{F}_t^B]$ \\
		$\overline{\sigma\pi}$ & $\mathbb{E}[\sigma\pi(\cdot)\mid \mathcal{F}_t^B]$ \\
		$\widehat{\pi}$ & The MFE strategy \\
		$\overline{\sigma\widehat{\pi}}$ & $\mathbb{E}[\sigma\widehat{\pi}\mid \mathcal{F}_t^B]$ \\
		$\overline{\underline{br}\widehat{\pi}}$ & $\mathbb{E}[(b-r)\widehat{\pi}\mid \mathcal{F}_t^B]$ \\
		$\rho_{\scriptscriptstyle m}$ & $1/[2(\xi^2+\sigma^2)]$ \\
		$p_{\scriptscriptstyle m}$ & Coefficient of $x$ in $\widehat{\pi}$, defined in \eqref{eq:MF-p1} \\
		$q_{\scriptscriptstyle m}$ & Coefficient of $\bar{x}$ in $\widehat{\pi}$, defined in \eqref{eq:MFq2} \\
		$Q_{\scriptscriptstyle m}$ & Defined in \eqref{eq:MFQ} \\
		$k_{1,\scriptscriptstyle m}$ & Coefficient of $\overline{\sigma\widehat{\pi}}$ in $\widehat{\pi}$, defined in \eqref{eq:MFk1} \\
		$k_{2,\scriptscriptstyle m}$ & Coefficient of $\overline{\underline{br}\widehat{\pi}}$ in $\widehat{\pi}$, defined in \eqref{eq:MFk2} \\
		$k_{3,\scriptscriptstyle m}$ & Constant term in $\widehat{\pi}$, defined in \eqref{eq:MFk3} \\
		$\delta_{\scriptscriptstyle m}$ & $\sigma k_{1,\scriptscriptstyle m}+(b-r)k_{2,\scriptscriptstyle m}$ \\
		$\Psi$ & $\mathbb{E}\left[\frac{\phi\sigma^2}{\xi^2+\sigma^2}\right]$ \\
		$\Theta$ & $\mathbb{E}\left[\frac{\sigma b}{\gamma(\xi^2+\sigma^2)}\right]$ \\
		\bottomrule
	\end{tabular}
\end{table}

\end{bluepar}

\end{document}